\documentclass[11pt]{amsart}
\usepackage[margin=1.25in]{geometry}

\usepackage{mathrsfs}
\usepackage{amsfonts}
\usepackage{latexsym,epsfig}
\usepackage{mathabx}
\usepackage{amsmath, amscd, amsthm, amssymb}
\usepackage{comment}
\usepackage{marginnote}
\usepackage[colorlinks,citecolor=blue,linkcolor=blue]{hyperref}
\usepackage[nameinlink]{cleveref}
\usepackage{caption}
\usepackage{subcaption}
\usepackage{enumerate}
\newcommand{\Teich}{\mathcal{T}}
\newcommand{\QD}[0]{\mathcal{QD}}
\newcommand{\define}[1]{\textbf{#1}}
\newcommand{\el}{\mathcal{EL}}

\newcommand{\GQD}[0]{\mathcal{GQD}}
\newcommand{\MQD}{\mathcal{MQD}}
\newcommand{\PMF}[0]{\mathcal{PMF}}
\newcommand{\PML}[0]{\mathcal{PML}}
\newcommand{\MF}[0]{\mathcal{MF}}
\newcommand{\ML}[0]{\mathcal{ML}}
\newcommand{\fol}[0]{\mathcal{F}}

\newcommand{\face}[0]{\mathbf{F}}
\newcommand{\tri}{\tau}
\newcommand{\Mod}{\mathrm{Mod}}
\newcommand{\AC}{\mathcal{AC}}
\newcommand{\tet}{\mathfrak{t}}
\newcommand{\diag}{\Delta}
\newcommand{\tetface}{\Delta}
\newcommand{\AH}{\mathrm{AH}}
\newcommand{\DD}{\mathrm{DD}}
\newcommand{\Scirc}{{\mathring{S}}}
\newcommand{\qcirc}{{\mathring{q}}}
\newcommand{\phicirc}{{\mathring{\varphi}}}

\newcommand{\support}{\operatorname{Supp}}
\newcommand{\semigroup}[1]{{\langle \support(#1) \rangle_+}}

\newcommand{\quotient}[2]{{\raisebox{0.2em}{$#1$}
           \!\!\left/\raisebox{-0.2em}{\!$#2$}\right.}}

\newcommand{\RR}[0]{\mathbb{R}}

\newcommand{\HH}[0]{\mathbb{H}}

\newcommand{\A}[0]{\mathcal{A}}

\newcommand{\CC}[0]{\mathbb{C}}

\newcommand{\ZZ}[0]{\mathbb{Z}}
\newcommand{\NN}[0]{\mathbb{N}}

\newcommand{\QQ}[0]{\mathbb{Q}}

\newcommand{\F}{\mathcal{F}}

\newcommand{\T}{\mathcal{T}}

\newcommand{\C}{\mathcal{C}}

\newcommand{\R}{\mathbb{R}}
\renewcommand{\A}{\mathcal{A}}

\newcommand{\dd}{\mathrm{DD}}
\newcommand{\e}{\mathcal{E}}

\newcommand{\del}{\partial}
\renewcommand{\setminus}{\smallsetminus}

\newcommand{\bdy}{{\partial}}
\newcommand{\from}{{\colon}}

\newcommand{\mc}[1]{\mathcal{#1}}
\newcommand{\mr}[1]{\mathrm{#1}}

\usepackage{hyperref}

\newtheorem{theorem}{Theorem}[section]
\newtheorem{lemma}[theorem]{Lemma}
\newtheorem{proposition}[theorem]{Proposition}
\newtheorem{corollary}[theorem]{Corollary}
\newtheorem{claim}[theorem]{Claim}

\theoremstyle{definition}

\newtheorem{remark}[theorem]{Remark}
\newtheorem{fact}[theorem]{Fact}

\numberwithin{equation}{section}

\begin{document}

\title{Random veering triangulations are not geometric}
\author[D. Futer]{David Futer}
\author[S.J. Taylor]{Samuel J. Taylor}
\address{Department of Mathematics\\ 
Temple University\\ 
1805 N. Broad St\\ 
Philadelphia, PA 19122}
\email{\href{mailto:dfuter@temple.edu}{dfuter@temple.edu}}
\email{\href{mailto:samuel.taylor@temple.edu}{samuel.taylor@temple.edu}}
\author[W. Worden]{William Worden}
\address{Department of Mathematics\\
Rice University MS-136\\
1600 Main St.\\
Houston, TX 77251}
\email{\href{mailto:william.worden@rice.edu}{william.worden@rice.edu}}

\date{\today}
\thanks{Futer was partially supported by NSF grants DMS--1408682 and DMS--1907708.} 
\thanks{Taylor was partially supported by NSF grants DMS--1400498 and DMS--1744551.}
\maketitle

\begin{abstract}
Every pseudo-Anosov mapping class $\varphi$ defines an associated veering triangulation $\tri_\varphi$ of a punctured mapping torus. We show that generically, $\tri_\varphi$ is not geometric. Here, the word ``generic'' can be taken either with respect to random walks in mapping class groups or with respect to  counting geodesics in moduli space. Tools in the proof include Teichm\"uller theory, the Ending Lamination Theorem, study of the Thurston norm, and rigorous computation.
\end{abstract}

%% !TEX root =rndm_veering.tex

\section{Introduction}\label{sec:intro}

In 2011, Agol introduced the notion of a layered veering triangulation for certain hyperbolic mapping tori \cite{Ago11}. 
Given a hyperbolic surface $S$ and a pseudo-Anosov homeomorphism $\varphi \from S \to S$, 
the mapping torus $M_\varphi$ with fiber $S$ and monodromy $\varphi$ is always hyperbolic. 
Drilling out 
 the singularities of the $\varphi$--invariant foliations on $S$ produces a punctured surface $\Scirc$ and a restricted pseudo-Anosov map $\phicirc = \varphi \vert_{\Scirc}$, whose mapping torus $\mathring{M}_\varphi = M_\phicirc$ is a surgery parent of $M_\varphi$. Agol's construction uses splitting sequences of train tracks to produce an ideal triangulation of $\mathring{M}_\varphi$ (that is, a decomposition of $\mathring{M}_\varphi$ into simplices whose vertices have been removed) called the \define{veering triangulation associated to $\varphi$}.

In \Cref{Sec:VeeringBackground}, we give a detailed description of the  veering triangulation $\tri = \tri_\varphi$ from an alternate point of view, introduced by Gu\'eritaud \cite{Gue16}. For now, we mention that $\tri$  has very strong combinatorial and topological properties.
The triangulation $\tri$ is layered, meaning that every edge is isotopic to an essential arc on the punctured fiber $\mathring{S}$. The triangulation 
$\tri$  contains a product region $\Sigma \times I$ for every large-distance subsurface  $\Sigma \subset \mathring{S}$ \cite{MT17}. 
Finally,
$\tau_\varphi$ decorated with layering data is a complete invariant of the conjugacy class $[\varphi]  \subset \Mod(S)$ \cite[Corollary 4.3]{Ago11}, which yields a fast practical solution to the conjugacy problem for pseudo-Anosovs \cite{Bel13, MSY}.
Given these combinatorial properties, it is natural to ask whether $\tri$ also has desirable geometric properties in the complete hyperbolic metric on $\mathring{M}_\varphi$. 

%    
%    A \define{straight ideal simplex} in $\HH^n$ is the convex hull of $k \leq n+1$ points on the boundary sphere. In $\HH^3$, the shape of a straight ideal tetrahedron $\tet$ is dtermined up to isometry by the cross-ratio $z_\tet$ of the 4 ideal vertices. By convention, $\tet$ is positively oriented if  $Im(z_\tet) > 0$, and negatively  oriented if $Im(z_\tet) < 0$. 
%    
%    The triangulation $\tri$ is called \define{geometric} if the complete hyperbolic structure on $\mathring{M}_\varphi$ can be obtained by taking positively oriented tetrahedra in $\HH^3$ in bijection with the $3$--simplices of $\tri$, and gluing them by isometry in the combinatorial pattern of $\tri$. An equivalent definition uses Gromov's idea of a straightening homotopy that moves every simplex of $\tau$ to the convex hull of its ideal vertices. The triangulation $\tri$ is geometric if and only if the straightening homotopy can be accomplished by isotopy.

Since every edge of $\tri$ is homotopically non-trivial, 
it is possible to homotope every ideal tetrahedron  $\tet \subset \tri$ to a straight simplex $\tet'$, whose lift to the universal cover $\HH^3$ is the convex hull of $4$ points on $\bdy \HH^3$. This homotopy is natural, in the sense that it extends continuously to all of $\tri$. The triangulation $\tri$ is called \define{geometric} if the straightening homotopy can be accomplished by isotopy. 
Equivalently, $\tri$ is called geometric if the complete hyperbolic structure on $\mathring{M}_\varphi$ can be obtained by taking positively oriented tetrahedra in $\HH^3$ in bijection with the $3$--simplices of $\tri$, and gluing them by isometry in the combinatorial pattern of $\tri$. 

%    \comm{Second attempt to define \define{geometric}. This introduces shape parameters -- and it seems we need to do so.}
%    
%    Geometricity can also be described in terms of shape parameters. A straight ideal simplex $\tet' \subset \HH^3$ is determined up to isometry by its a \define{shape parameter} $z(\tet')$, namely the cross-ratio of the $4$ ideal vertices. 
%    
%    Since every edge of $\tri$ is homotopically non-trivial, 
%    it is possible to homotope every ideal tetrahedron  $\tet \subset \tri$ to a straight simplex, whose lift to the universal cover $\HH^3$ is the convex hull of $4$ points on $\bdy \HH^3$. The cross-ratio of these $4$ points, called the \define{shape parameter} of $\tet$ and denoted $z(\tet)$, determines the straight simplex up to orientation-preserving isometry. We say that $\tri$ is \define{geometric} if $z(\tet)$ has positive imaginary part for every tetrahedron $\tet \subset \tri$. Equivalently, $\tri$ is called geometric if the complete hyperbolic structure on $\mathring{M}_\varphi$ can be obtained by taking positively oriented straight simplices in $\HH^3$ in bijection with the tetrahedra of $\tri$, and gluing them by isometry in the combinatorial pattern of $\tri$. 
%    
%    \comm{End of attempts}

Agol asked whether veering triangulations are always geometric \cite[Section 5]{Ago11}. Hodgson, Issa, and Segerman showed that the answer can be negative \cite{HIS16}, by finding a veering triangulation with 13 tetrahedra, in which one tetrahedron is negatively oriented. (In the straightening homotopy, two opposite edges of this tetrahedron must pass through each other before the tetrahedron can become straight.) In describing their example, they write, 
\begin{quote}
\emph{It seems unlikely that a counterexample would have been found without a computer search, and it is still something of a mystery why veering triangulations are so frequently geometric.}
\end{quote}

It is now clear that  geometric veering triangulations are exceedingly rare. This was shown experimentally by Worden  \cite{Wor18}, who tested over 800,000 examples on a high-performance computing cluster. 
Given a hyperbolic surface $S$ of complexity $\xi(S) \geq 2$, he found that for randomly sampled long words in $\Mod(S)$, the probability of the associated veering triangulation being geometric decays exponentially with the length of the word. See \Cref{Fig:Data}.

\begin{figure}
\includegraphics[width=2.9in]{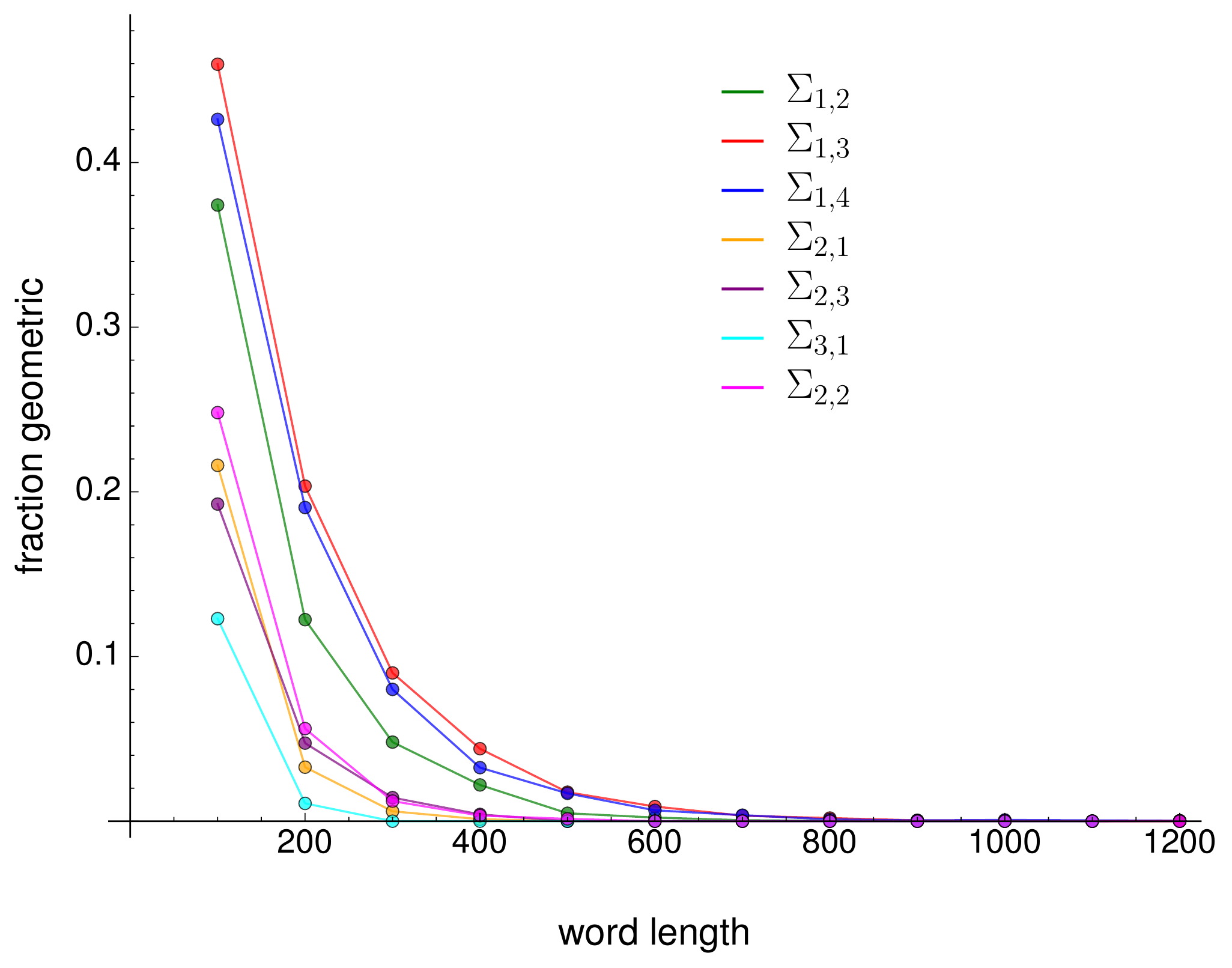}
\includegraphics[width=2.9in]{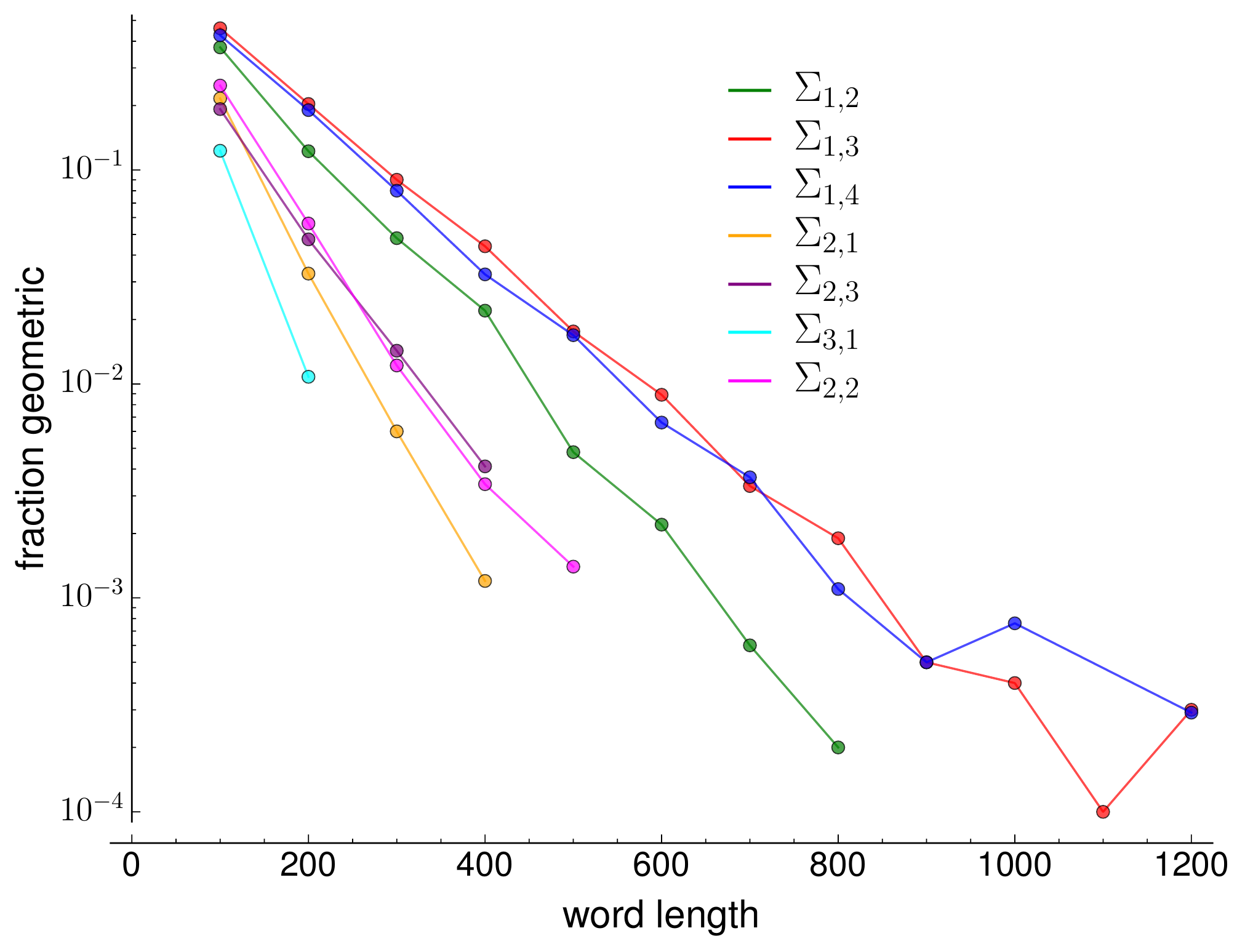}
\caption{For a simple random walk in $\Mod(S)$, with generators shown in \Cref{fig:surf_gens},
 the probability that the veering triangulation is geometric decays exponentially with the length of the walk. 
% Words are in the generators shown in \Cref{fig:surf_gens}. 
 Both graphs show the same data, with a linear plot on the left and a log-linear plot on the right. Each dot represents several thousand mapping classes. Figure from Worden \cite{Wor18}.}
\label{Fig:Data}
\end{figure}

The main result of this paper is a proof of the qualitative pattern visible in  \Cref{Fig:Data}. While we do not prove exponential decay, we do prove that the proportion of geometric triangulations decays to $0$. 
We establish this in two separate probabilistic regimes: first, with respect to  random walks on $\Mod(S)$ (\Cref{thm:genericity}), and second, with respect to counting closed geodesics in moduli space (\Cref{thm:counting}).

We use the symbol $\Sigma_{g,n}$ to denote the surface of genus $g$ with $n$ punctures. Every surface  $S$ mentioned below is presumed  homeomorphic to some $\Sigma_{g,n}$; in particular, $S$ is presumed connected and orientable. We define the \define{complexity} $\xi(\Sigma_{g,n}) = 3g - 3 +n$.

For a surface $S$ as above, we show that with overwhelming probability, 
a random walk on $\Mod(S)$ produces a pseudo-Anosov mapping class with non-geometric veering triangulation.

\begin{theorem}
\label{thm:genericity}
Let $S$ be a surface of complexity $\xi(S) \geq 2$, and consider
%    
%    one of the following surfaces:
%    \begin{itemize}
%    \item $\Sigma_{g,n}$ for $g\ge 1$, excluding $\Sigma_{1,0}$ and $\Sigma_{1,1}$.
%    \item $\Sigma_{0,n}$ for $n \geq 5$ and $n\equiv 3 \mod p$, where $2\le p\le 19$.
%    \end{itemize}
%    Consider 
a simple random walk on  $\Mod(S)$ with respect to any finite generating set. Then, for almost every infinite sample path $(\varphi_n)$, there is a positive integer $n_0$ such that for all $n \geq n_0$, the mapping class $\varphi_n$ is pseudo-Anosov and the veering triangulation of $\mathring{M}_{\varphi_n}$ is non-geometric.
\end{theorem}

In fact, the same result holds true for sample paths defined by a more general probability measure. See \Cref{Cor:NonGeoConvergence} for a precise statement.

We remark that every pseudo-Anosov on a surface satisfying $\xi(S) < 2$ has a geometric veering triangulation. (See  \Cref{Prop:NongeomExists} and the ensuing discussion.) Thus \Cref{thm:genericity} applies to the largest possible collection of (orientable) surfaces.

%    crucial to \Cref{thm:genericity}. The only (connected, orientable) hyperbolic surfaces satisfying $\xi(S) < 2$ are $\Sigma_{0,3}$, $\Sigma_{0,4}$, and $\Sigma_{1,1}$. However, $\Mod(\Sigma_{0,3})$ is finite and contains no pseudo-Anosov elements at all. Meanwhile, Akiyoshi \cite{Aki99} and Lackenby \cite{Lackenby:Bundle} proved that  all pseudo-Anosov mapping classes on $\Sigma_{1,1}$ and $\Sigma_{0,4}$  have geometric veering triangulations. Gu\'eritaud  gave a direct argument for the same conclusion \cite{Gue06}.

%    \comm{$\mathring{M}_\varphi$ might be better notation than $M_{\phicirc}$. At least, it looks better aesthetically.}

We will prove \Cref{thm:genericity} by combining two separate, logically independent ingredients.
The first ingredient is  \Cref{Prop:NongeomExists}: when $\xi(S) \geq 2$, there is at least one principal mapping class on $S$ whose associated veering triangulation is non-geometric. A pseudo-Anosov mapping class $\varphi \in \Mod(S)$ is called \define{principal} if its invariant Teichm\"uller geodesic lies in the principal stratum. (See \Cref{Sec:TeichBackground} for a discussion of strata and Teichm\"uller geodesics.) Equivalently, $\varphi$ is principal if its stable foliation has $3$--prong singularities at interior points of $S$ and $1$--prong singularities at punctures of $S$. By a theorem of  Gadre and Maher \cite{GM17}, principal pseudo-Anosovs are generic from the point of view of random walks in $\Mod(S)$.

The second ingredient is a convergence result, \Cref{thm:convergence}, which shows that every principal mapping class occurs in a suitable sense as the limit of a random process, where the combinatorics of the triangulation and the geometry of the mapping torus both converge to the desired limit. This result works for any hyperbolic surface. In particular, given a principal mapping class $\varphi$ with non-geometric veering triangulation, almost every sample path of a random walk also has non-geometric veering triangulation.

\smallskip

%    
%    Based on Worden's results \cite{Wor18}, we expect the conclusion of \Cref{thm:genericity} to hold for all surfaces of complexity $\xi(S) \geq 2$; that is, for all hyperbolic surfaces except $\Sigma_{0,3}$, $\Sigma_{0,4}$, and $\Sigma_{1,1}$. Indeed, \Cref{thm:convergence} applies to all hyperbolic surfaces. However, the first step of our argument (specifically, the proof of \Cref{prop:candidates}) currently works for ``most'' but not all planar surfaces. The planar surfaces reached by our methods are $\Sigma_{0,n}$ where $n$ occupies over
%    $82\%$ of the residue classes modulo 
%    $\mathop{lcm}(2, \ldots, 19) = 116,396,280$.

By replacing random walk techniques with work of Hamenst\"adt \cite{hamenstadt2013bowen} and Eskin--Mirzakhani \cite{eskin2011counting}, we prove our second result concerning the scarcity of geometric veering triangulations. For $L > 0$, let $\mathcal{G}(L)$ be the finite set of conjugacy classes of pseudo-Anosov mapping classes in $\Mod(S)$ whose Teichm\"uller translation length is at most $L$. 
Equivalently, $\mathcal{G}(L)$ is the set of all conjugacy classes of pseudo-Anosovs whose dilatation is at most $e^{L}$.
Recall that the veering triangulation $\tri_\varphi$ of the punctured mapping torus $\mathring M_\varphi$ only depends on the conjugacy class $[\varphi]$, i.e.\ on an element of $\mathcal{G}(L)$ for some $L$.

\begin{theorem} \label{thm:counting}
Let $S$ be a surface with complexity $\xi(S) \ge 2$. Then
\[
\lim_{L \to \infty} \,
\frac{1}{|\mathcal{G}(L)|} \left |\left \{[\varphi] \in \mathcal{G}(L) : \text{the veering triangulation of } \mathring M_\varphi \text{ is not geometric} \right \} \right| 
\: = \: 1.
%    \longrightarrow 1,
\]
%    as $n \to \infty$.
\end{theorem}

Just as with  \Cref{thm:genericity}, the proof of \Cref{thm:counting} combines an existence statement with a convergence statement. The existence statement is again  \Cref{Prop:NongeomExists}: there is a principal mapping class $\varphi \in \Mod(S)$ whose associated veering triangulation is non-geometric. The convergence statement roughly says that the axis of a typical element of $\mathcal{G}(L)$ fellow-travels the axis of $\varphi$ for a very long distance. This statement, combined with ingredients from the proof of \Cref{thm:convergence}, implies the desired result. 
%    This will be proven in \Cref{sec:counting}, after establishing all of the preliminary work needed for \Cref{thm:genericity}. 
We refer to \Cref{sec:counting} for more details.

\subsection{Existence of non-geometric triangulations}

As described above, we begin the proof of \Cref{thm:genericity,thm:counting} by finding some pseudo-Anosov element of $\Mod(S)$ whose associated veering triangulation is non-geometric. In fact, we show the following.

\begin{theorem}\label{Prop:NongeomExists}
Let $S \cong \Sigma_{g,n}$ be a hyperbolic surface. Then  $\xi (S) \geq 2$ if and only if
%    	not $\Sigma_{1,1}$ or $\Sigma_{0,4}$, 
 there exists a principal pseudo-Anosov $\varphi \in \Mod(S)$ such that the associated veering triangulation of the mapping torus $\mathring M_{\varphi}$ is non-geometric.
\end{theorem}

The ``if'' direction of \Cref{Prop:NongeomExists} is previously known. The only (connected, orientable) hyperbolic surfaces with $\xi(S) < 2$ are $\Sigma_{0,3}, \Sigma_{0,4}$, and $\Sigma_{1,1}$.  Akiyoshi \cite{Aki99} and Lackenby \cite{Lackenby:Bundle} proved that  all pseudo-Anosov mapping classes on $\Sigma_{1,1}$ and $\Sigma_{0,4}$  have geometric veering triangulations. Gu\'eritaud  gave a direct argument for the same conclusion \cite{Gue06}.  Meanwhile, $\Mod(\Sigma_{0,3})$ is finite, hence $\Sigma_{0,3}$ has no pseudo-Anosov mapping classes at all.  Thus the new content of \Cref{Prop:NongeomExists} is the ``only if'' direction of the statement.

%    With stronger hypotheses on $S$, we can prove a stronger conclusion. A pseudo-Anosov mapping class $\varphi \in \Mod(S)$ is called \define{principal} if its invariant Teichm\"uller geodesic lies in the principal stratum. (See \Cref{Sec:TeichBackground} for a discussion of strata and Teichm\"uller geodesics.) Equivalently, $\varphi$ is principal if its stable foliation has $3$--prong singularities at interior points of $S$ and $1$--prong singularities at punctures of $S$. By a theorem of  Gadre and Maher \cite{GM17}, principal pseudo-Anosovs are generic from the point of view of random walks in $\Mod(S)$.

%    
%    \begin{proposition}\label{prop:candidates}
%    Let $S=\Sigma_{g,n}$ be a hyperbolic surface satisfying one of the following:
%    \begin{itemize}
%    \item $\Sigma_{g,n}$ for $g\ge 1$, excluding  $\Sigma_{1,1}$, or
%    \item $\Sigma_{0,n}$ for $n \geq 5$ and $n\equiv 3 \mod p$, where $2\le p\le 19$.
%    \end{itemize}
%    Then  there exists a principal pseudo-Anosov $\varphi \in \Mod(S)$, such that the associated veering triangulation of the mapping torus $\mathring M_{\varphi}$ is non-geometric.
%    \end{proposition}
%    

In the proof of  \Cref{Prop:NongeomExists}, we characterize geometric triangulations using shape parameters. Given a tetrahedron $\tet \subset \mathring{M}$, endowed with an ordering of its ideal vertices, we lift $\tet$ to $\HH^3$ and define the \define{shape parameter} $z_\tet$ to be the cross-ratio of the 4 vertices on the sphere at infinity. The cross-ratio $z_\tet$ determines the isometry type of the straightened tetrahedron $\tet'$ homotopic to $\tet$. In particular, $\tet'$ is positively oriented if and only if $\mathrm{Im}(z_\tet) > 0$. Shape parameters can be computed using  \texttt{Snappy} \cite{CDGW09}, in either floating-point or interval arithmetic, making it possible to test whether a particular triangulation $\tri$ is geometric.

To prove the ``only if'' direction of  \Cref{Prop:NongeomExists} for a finite list of fiber surfaces, we essentially follow the method of Hodgson--Issa--Segerman \cite{HIS16}. We find a suitable mapping class $\varphi \in \Mod(S)$ using a brute-force search, and use \texttt{flipper} \cite{Bel13} to certify that $\varphi$ is a principal pseudo-Anosov. Then, we use rigorous interval arithmetic in  \texttt{Snappy}, including routines derived from \texttt{HIKMOT}  \cite{HIKMOT16}, to certify that the shape parameter of each $\tet \subset \tri_\varphi$ lies inside a small box in $\CC$. One of these boxes has strictly negative imaginary part, implying that $\tri_\varphi$ is non-geometric. See  \Cref{sec:poison} for details.

To extend our knowledge from finitely many surfaces to \emph{all} the surfaces in \Cref{Prop:NongeomExists}, 
we exploit the fact that many fibered $3$--manifolds fiber in infinitely many ways, organized via the Thurston norm. 
%% DF: Many, yes. Typical, no. A generic fibered 3-manifold has only peripheral homology, so our examples are somewhat atypical.
(See \Cref{sec:norm} for definitions and further details.) All the fibers that appear in a single fibered cone of the Thurston norm ball 
have associated monodromies that induce the same veering triangulation of the same drilled manifold $\mathring{M}$. As a consequence, we can prove \Cref{Prop:NongeomExists} for all $g\ge 1$, $n\ge 1$ (excluding $\Sigma_{1,1}$), using only two explicit examples. That is, we find two fibered manifolds with principal pseudo-Anosovs whose veering triangulations are non-geometric, and show that every such surface $\Sigma_{g,n}$, appears as (a cover of) a fiber for at least one of our two examples. The verification that these two fibered manifolds have all the desired properties is assisted by \texttt{Regina} \cite{BBP99}.
Similar tricks handle the other surfaces with $\xi(S) \geq 2$, namely closed surfaces and punctured spheres.

%    The above proposition improves on \Cref{prop:candidates} in that we get all hyperbolic surfaces that admit pseudo-Anosovs, except for $\Sigma_{1,1}$ and $\Sigma_{0,4}$. However, this comes at a cost: some of the pseudo-Anosov maps are not in the principal stratum. 

\subsection{Convergence to any principal pseudo-Anosov}

The following convergence theorem is the main technical result of this paper. Although  the statement here is for the random walk model, it is derived from a more general result (\Cref{Prop:TetraShapesConverge}) that also applies to counting geodesics in moduli space.

Consider a probability measure $\mu$ on $\Mod(S)$. We use the notation $\semigroup{\mu}$ to denote the semigroup generated by the support of $\mu$. Say that $\semigroup{\mu}$ is \define{non-elementary} if it contains at least two pseudo-Anosov elements with distinct axes. In the setting of a simple random walk, $\mu$ is the uniform probability measure on a symmetric generating set, hence $\semigroup{\mu} = \Mod(S)$ is non-elementary.

\begin{theorem}
\label{thm:convergence}
Let $S$ be a hyperbolic surface, and fix a principal pseudo-Anosov $\varphi \in \Mod(S)$. Lift the veering triangulation $\tri_\varphi$ of the mapping torus $\mathring M_{\varphi}$ to a triangulation $\widetilde \tri$ of the infinite cyclic cover $\mathring N_{\varphi}$, corresponding to the fiber. Let $K \subset \widetilde \tri$ be any finite, connected sub-complex.

Let $\mu$ be a probability distribution on $\Mod(S)$ with finite first moment, such that $\semigroup{\mu}$ is non-elementary and contains $\varphi$.
Then, for almost every sample path $\omega = (\omega_n)$, there is a positive integer $n_0$ such that the following hold:
%    \begin{enumerate}[$\bullet$]
\begin{itemize}
\item For all $n \geq n_0$, $\omega_n$ is a principal pseudo-Anosov.
\item For all $n \geq n_0$, $K$ embeds as a sub-complex of  the veering triangulation $\tri_{\omega_n}$ of  the mapping torus $\mathring M_{\omega_n}$.
\item For every tetrahedron $\tet \subset K$, the shape of $\tet$ in $\mathring M_{\omega_n}$ converges to the shape of $\tet$ in $\mathring N_{\varphi}$ as $n \to \infty$.
\end{itemize}
%    \end{enumerate}
\end{theorem}

%\comm{Stupid question: is the ``non-elementary'' hypothesis necessary? If the subgroup is elementary, i.e.\ virtually cyclic, won't the result be true whenever $\omega_n$ is not torsion? Or is recurrence an issue?}

The argument used to prove \Cref{thm:convergence} can be summarized as follows.
The first two conclusions are proven by combining a result about fellow-traveling of sample paths in Teichm\"uller space (\Cref{thm:GM}, due to Gadre--Maher \cite{GM17}), together with \Cref{Cor:ComplexEmbeds}. Informally, \Cref{Cor:ComplexEmbeds} states that if appropriate quadratic differentials $q_i$ converge to $q$, then their associated veering triangulations also converge, in the sense appearing in \Cref{thm:convergence}. A more precise formulation of this result requires Gu\'eritaud's construction of the veering triangulation (given in \Cref{Sec:VeeringBackground}), and is postponed until \Cref{sec:convergence}. The upshot is that if $\varphi$ and $\psi$ are pseudo-Anosov homeomorphisms whose axes in Techm\"uller space fellow travel for sufficiently long, then their associated veering triangulations of $\mathring{M}_\varphi$ and $\mathring{M}_\psi$ have large isomorphic subcomplexes. 

The third conclusion of \Cref{thm:convergence} follows by relating the convergence of the quadratic differentials referenced above to the algebraic convergence of the hyperbolic structures on the associated manifolds. The main tool for this is the Ending Lamination Theorem of Brock--Canary--Minsky \cite{Min10, BCM12} together with a strengthening by Leininger--Schleimer \cite{LS09}. In short, algebraic convergence of the surface group representations yields convergence in $\bdy \mathbb{H}^3$ of the ideal endpoints of the veering tetrahedra, which means that the shapes of these tetrahedra converge as desired. The details are given in \Cref{sec:elc}.

One particular consequence of \Cref{thm:convergence} is the following statement.

\begin{corollary}\label{Cor:NonGeoConvergence}
Let $S$ be a hyperbolic surface. Let $\varphi \in \Mod(S)$ be a principal pseudo-Anosov  whose veering triangulation  $\tri_\varphi$ is non-geometric.

Let $\mu$ be a probability distribution on $\Mod(S)$ with finite first moment, such that $\semigroup{\mu}$ is non-elementary and contains $\varphi$.
Then, for almost every infinite sample path $\omega = (\omega_n)$, there is a positive integer $n_0$ such that for all  $n \geq n_0$, the veering triangulation $\tri_{\omega_n}$ is also non-geometric.
\end{corollary}

\begin{proof}
Let $K \subset \widetilde \tau$ be (the lift to $\mathring N_\varphi$ of)  a single tetrahedron $\tet \subset \tau_\varphi$ whose shape  is negatively oriented. \Cref{thm:convergence} says that $\tet$ also appears as a tetrahedron in $\tri_{\omega_n}$ for $n \gg 0$. Furthermore, the shape of $\tet$ in $\mathring M_{\omega_n}$ converges to a negatively oriented limit as $n \to \infty$, hence $\tet$ has to be negatively oriented in $\mathring M_{\omega_n}$ for all $n  \gg 0$.
\end{proof}

Now, observe that \Cref{thm:genericity} follows immediately by combining \Cref{Prop:NongeomExists} with \Cref{Cor:NonGeoConvergence}. The proof of \Cref{thm:counting} follows a similar pattern, but replaces \Cref{Cor:NonGeoConvergence} with \Cref{cor:open}. See \Cref{sec:counting} for the full details.

\subsection{Organization}
\Cref{sec:background} lays out definitions and background material from Teichm\"uller theory that will be needed in most of the subsequent arguments.  

The proof of \Cref{thm:convergence} spans \Cref{sec:qd_convergence,sec:transition,sec:convergence,sec:elc}. We discuss convergence of quadratic differentials in \Cref{sec:qd_convergence,sec:transition}, convergence of veering triangulations in \Cref{sec:convergence}, and finally convergence of geometric structures on $3$--manifolds in \Cref{sec:elc}. In \Cref{sec:counting}, we combine these ingredients with measure-theoretic tools to prove \Cref{thm:counting}.

Finally, \Cref{sec:poison,sec:norm} contain the proof of \Cref{Prop:NongeomExists}. 

\subsection{Acknowledgements}
We thank Matthias Goerner, Eiko Kin, Yair Minsky, Saul 
\linebreak 
Schleimer, and Henry Segerman for a number of enlightening conversations. We thank Matthew Stover and the referee for helpful comments that improved our exposition. We also thank Vaibhav Gadre for suggesting that we prove \Cref{thm:counting} and Ilya Gekhtman for help with the details.

%    The following theorem gives a partial answer to \Cref{conj:genericity}, in the affirmative. 

%    candidate surface, and fix a set of generators for $\Mod(S)$. Let $P_{k,S}$ be the probability that a simple random walk of length $k$ on $\Mod(S)$ yields a pseudo-Anosov $\varphi$ for which the layered veering triangulation of the mapping torus $M_{\phicirc}$ is geometric. Then 
%    $$
%    \lim_{k\to \infty} P_{k,S}\to 0
%    $$

%    
%    
%    We state the above theorem in terms of candidate surfaces to highlight the fact that the barrier to proving the theorem for \emph{all} surfaces (other than $\Sigma_{1,1}$ and $\Sigma_{0,4}$) lies in finding a single pseudo-Anosov in the principal stratum for which the veering triangulation is non-geometric. In other words, the proof does not depend on the surface $S$, just on the existence of an appropriate pseudo-Anosov map for each $S$.

%% !TEX root =rndm_veering.tex

\section{Background}\label{sec:background}
The primary goal of this section is to survey some background material on Teichm\"uller theory, quadratic differentials, and measured foliations that will be heavily used in the following few sections. The reader is referred to \cite{bonahon1988geometry, FM12, FLP12, strebel1984quadratic} for additional details.
After this general background, we describe Gu\'eritaud's construction of veering triangulations from quadratic differentials \cite{Gue16}.

Throughout, we let $S=\Sigma_{g,n}$ be a surface of genus $g$ with $n$ punctures, and assume that $\xi(S) = 3g-3+n  \geq 1$. This assumption implies that the Teichm\"uller space $\Teich(S)$, which is the space of complex structures on $S$ up to isotopy, has real dimension $2\xi(S) \ge 2$.

\subsection{Quadratic differentials and strata}\label{Sec:TeichBackground}
Let $X \in \Teich(S)$ be a complex structure on $S$. A \define{quadratic differential} $q$ on $X$ is a tensor locally defined in coordinates by $q=q(z)dz^2$ for some meromorphic function $q(z)$. The function $q(z)$ is required to be analytic inside $S$, but is allowed to have simple poles at the punctures of $S$. 
%The poles and zeros of $q$ are called \define{singularities}. 
By changing coordinates, we may assume that $q=dz^2$ in the neighborhood of a regular value of $q$, and $q=z^kdz^2$ in the neighborhood of a pole/zero ($k=-1$ for simple poles, and $k>0$ for zeros). These are called natural coordinates and have the property that, away from poles and zeros of $q$, the transition functions have the form $z \to \pm z+c$, for some complex number $c$. In particular, these  transition functions preserve the standard Euclidean metric on $\mathbb{C}$. Throughout the paper, we call the poles and zeros of $q$, as well as the punctures of $S$, \define{singularities} of $q$. This is because punctures of $S$ will play a role similar to other singularities of $q$ even though such punctures may correspond to regular values of $q$.

A quadratic differential $q$ determines a pair of transverse measured foliations $\fol_q^-$ and $\fol_q^+$, called the \define{horizontal} and \define{vertical} foliations. 
In the above natural coordinates $z=x+iy$ away from the singularities, these foliations are given by setting $y$ and $x$ (respectively) to be constant, with transverse measures
$|dy|$ and $|dx|$. Near a zero of order $k$ (where a pole corresponds to $k=-1$), 
each of the horizontal and vertical foliations has 
a $(k+2)$--pronged singularity.

Away from singularities, the transverse measures $|dx|$ and $|dy|$ induce a Euclidean metric $\sqrt{|dx|^2+|dy|^2}$ on $S$. The completion of this metric on $S$ is known as the \define{singular flat metric} corresponding to $q$. The area of $S$ endowed with this metric is denoted $\|q\|$, and defines a norm on the space $\QD(S)$ of quadratic differentials on $S$. We denote by $\QD^1(S)$ the set of elements $q\in \QD(S)$ with $\|q\|=1$. The projection $\QD(S) \to \Teich(S)$ sending a quadratic differential to its underlying complex structure can be identified with the cotangent bundle of $\Teich(S)$; see, for example, \cite{imayoshi2012introduction}. 

%\comm{Notation doesn't seem good. Elsewhere $\mathcal{P}$ means projective.\\
%Agreed. $\QD_p$ seems like a better choice, except that we use that for something else. Could we denote fully puncture quadratic differentials by $\QD^\circ$, or is that too many circles floating around? -W \\
%Or perhaps $\QD\mathcal{P}(S)$? -D}

The \define{principal stratum} of quadratic differentials $\GQD(S)$ is the subset of $\QD(S)$ that consists of all those quadratic differentials whose zeros are of order 1 (that is, 3--prong singularities), and whose punctures are all simple poles  (that is, 1--prong singularities). The symbol $\mathcal G$ in $\GQD(S)$ stands for ``generic.'' In general, $\QD(S)$ decomposes into strata characterized by the orders of the zeros and poles of $q(z)$. When $S \ncong \Sigma_{1,1}$, the principal stratum is open and dense, while the other strata have positive codimension. (When $S \cong \Sigma_{1,1}$, the principal stratum as previously defined is empty. All  nonzero quadratic differentials belong to a single stratum with a single $2$--prong singularity at the puncture.)

%    \comm{The above paragraph is not quite true on $\Sigma_{1,1}$. On the other hand, \Cref{thm:convergence} is still true (and interesting) on $\Sigma_{1,1}$. So we should take care to explain what ``principal'' means in this case.}

%In general, a stratum of $\QD(S)$ is a subset of $\QD(S)$ such that all or none of the elements are squared Abelian differentials, and such that all punctures and singularities have the same order data. That is, for any two members of the stratum, there is a bijection of singularities that preserves order, and a bijection of punctures that preserves order. 

%The orders $\hat{k}=(k_1,\dots,k_m,\pm)$ of the zeros and poles of $q$, along with a sign $\pm$ indicating whether $q$ is the square of an Abelian differential, determine a \define{stratum} $\mc{QD}_{\hat{k}}(S)$ of quadratic differentials. In particular, the stratum $(1,\dots,1,-)$ with all zeros and poles order 1 is called the \define{principal stratum} of quadratic differentials, and is denoted $\mc{QD}_1(S)$. The union of these strata is $\mc{QD}(S)$. 

%    We will consider $\Teich(S)$ with its Teichm\"uller metric. Informally, $d_\Teich(X,Y)$ measures the quasiconformal distortion between the marked complex structures $X$ and $Y$. Rather than recall the details here, we will only need a concrete description of Teichm\"uller geodesics in terms of quadratic differentials. 

\subsection{Teichm\"uller geodesics and flows}\label{Sec:Flow}
We recall the construction of the \define{Techm\"uller geodesic flow}, denoted $\Phi^t \colon \QD^1(S) \to \QD^1(S)$.
Given a unit-area quadratic differential $q \in \QD^1(S)$, and a number $t \in \RR$, the image $\Phi^t(q)$ is defined as follows. The underlying complex structure is $X_t=X_t(q)$, whose coordinate charts (away from singularities) are given by composing the natural coordinates for $q$ with the affine map 
\begin{equation}\label{Eqn:TeichMap}
\begin{pmatrix} 
e^t & 0 \\
0 & e^{-t}
\end{pmatrix}.
\end{equation}
Then, $\Phi^t(q) \in \QD^1(S)$ is the quadratic differential on $X_t$ given by $dz^2$ in these coordinates. The flow $\Phi^t$ plays an important role in \Cref{sec:counting}.

For a fixed $q \in \QD^1(S)$, the map $\RR \to \Teich(S)$ defined by $t \mapsto X_t(q)$ is called a  \define{Teichm\"uller geodesic}. Indeed, this line in $\Teich(S)$ is a parametrized geodesic for the \define{Teichm\"uller metric} $d_\Teich$.
%      that assigns $d_\Teich(X_0, X_t) = |t|$.
%    
%    
%    Given a quadratic differential $q \in \QD(S)$, the \define{Teichm\"uller geodesic} determined by $q$ is the projection to $\Teich(S)$ of the quadratic differentials $q_t$ defined as follows. The underlying complex structure for $q_t$ is $X_t$ whose coordinate charts (away from singularities) are given by composing the natural coordinates for $q$ with the affine map 
%    \begin{equation}\label{Eqn:TeichMap}
%    \begin{pmatrix} 
%    e^t & 0 \\
%    0 & e^{-t}
%    \end{pmatrix},
%    \end{equation}
%    and $q_t$ is the quadratic differential given by $dz^2$ in these coordinates. The resulting path $t \to X_t$ is a geodesic for the \define{Teichm\"uller metric} $d_\Teich$.
%    , and all geodesics for the Teichm\"uller metric arise in this manner. 
By Teichm\"uller's theorem, any pair of points $X,Y$ in $\T(S)$ are joined by a unique segment of a Teichm\"uller geodesic, of length $d_\Teich(X,Y)$, which we often denote by $[X,Y]$. The map $X = X_0 \to X_t = Y$ defined by \Cref{Eqn:TeichMap} is called the \define{Teichm\"uller map}.
If the quadratic differential $q$ associated to a Teichm\"uller geodesic $\gamma$ is in the principal stratum $\GQD(S)$, then we will say that $\gamma$ is in the principal stratum of Teichm\"uller space.

Consider now a pseudo-Anosov mapping class $\varphi\in \Mod(S)$.
%    , i.e. an infinite order mapping class which fixes no homotopy class of curve up to powers. 
Bers \cite{Ber78} showed that $\varphi$ preserves a unique geodesic axis $\gamma_\varphi \subset \Teich(S)$ consisting of  
points $X\in \Teich(S)$ such that $d_\Teich(X,\varphi(X))=\log \lambda_\varphi$, where $\lambda_\varphi > 1$ is the dilatation of $\varphi$. By \Cref{Eqn:TeichMap}, the geodesic $\gamma_\varphi$ corresponds to a one-parameter family $q_t$ of quadratic differentials. The complex structure $X_t$ underlying $q_t$ is a point along $\gamma_\varphi$, and the projective classes of $\F^+(q_t)$ and $\F^-(q_t)$ are constant and equal to the invariant foliations of $\varphi$. If some (hence every) $q_t$ lies in the principal stratum $\GQD(S)$, we say that $\varphi$ is a \define{principal pseudo-Anosov}.

%\begin{theorem}[Teichm\"uller]
%Given two marked Riemann surfaces $X,Y\in \Teich(S)$, $X\ne Y$, there is a unique quasi-conformal map $h:X\to Y$ called the \define{Teichm\"uller map}, along with quadratic differentials $q_X,q_Y\in \QD^1(S)$ on $X$ and $Y$, respectively, such that the following hold:
%\begin{itemize}
%\item[(1)] 	$d_{\Teich}(X,Y)= \frac{1}{2}\log K_h$
%\item[(2)] $q_X$ and $q_Y$ belong to the same stratum of $\QD(S)$
%\item[(3)] in natural coordinates for $q_X$ and $q_Y$ away from zeros and poles, $h$ can be written as 
%$$
%h(x+iy)=\sqrt{K_h}\cdot x+i\frac{1}{\sqrt{K_h}}\cdot y
%$$
%
%\end{itemize}
%Conversely, given a marked Riemann surface $X$, a quadratic differential $q\in\QD^1(S)$ on $X$, and $t\ge 0$, there is a Riemann surface $X_t$ with quadratic differential $q_t$ and a Teichm\"uller map $h_t:X\to X_t$ such that $d_\Teich(X,X_t)=t$.
%\end{theorem}

%The family $X_t$, $t>0$ is the Teichm\"uller geodesic ray $r(t)\subset \Teich(S)$ determined by $q$. If we take $h_t^{-1}$ in place of $h_t$, then we get the Teichm\"uller ray in the opposite direction, and the union of these two rays is a \define{Teichm\"uller geodesic}. If the quadratic differential $q$ associated to a Teichm\"uller geodesic $\gamma$ is in the principal stratum $\GQD(S)$, then we will say that $\gamma$ is in the principal stratum of Teichm\"uller space. It follows from the above theorem that there is a unique geodesic segment $[X,Y]$ between any two points $X\ne Y$ in $\Teich(S)$, which extends to a Teichm\"uller geodesic.

\subsection{Curves, foliations, and laminations}\label{sec:background:fols_and_lams}
One can study how conformal structures change along Teichm\"uller geodesics by understanding what happens to the lengths of curves and arcs. This perspective will be important in \Cref{sec:qd_convergence}.

The \define{arc and curve graph} $\AC(S)$ is the graph whose vertices are  isotopy classes of essential simple closed curves and simple proper arcs in $S$. Here, \emph{essential} means that the curve or arc is not isotopic into a small neighborhood of a point or a puncture. Two vertices are joined by an edge in $\AC(S)$ if they have disjoint representatives.
If we follow the same construction with vertices restricted to be closed curves on $S$, we obtain the \define{curve graph} $\C(S)\subset \AC(S)$, and similarly restricting to arcs yields the \define{arc graph} $\A(S)\subset \AC(S)$.

We have already encountered measured foliations as the vertical and horizontal foliations of a quadratic differential.
A \define{singular measured foliation} $\fol$ on $S$ is a singular foliation endowed with a transverse measure (see \cite{FLP12} for a more thorough definition). 
A Whitehead move on a foliation $\fol$  introduces or contracts a compact singular leaf on $\fol$, by either splitting a singularity into a pair of singularities joined by a compact leaf, or by contracting such a leaf to collapse two singularities into one. 
In general, we let $\MF(S)$ denote Thurston's space of measured foliations of $S$, up to Whitehead equivalence, whose topology comes from convergence of transverse measures; again see \cite{FLP12} for details.
The space $\PMF(S)$ of \define{projective measured foliations} is obtained from $\MF(S)$ by identifying measures which differ by scaling. 

%If the underlying topological singular foliation of $\fol$ supports a unique projective measure class, then $\fol$ is called \define{uniquely ergodic}. The subspace of uniquely ergodic foliations is denoted $\mc{UE}(S)\subset \PMF(S)$.

%    For our purposes, it will be convenient to pass back and forth between measured foliations and measured laminations.
By the Uniformization theorem, every conformal structure $X$ is realized by a unique hyperbolic metric.
A \define{geodesic measured lamination} on a hyperbolic surface $X$ is a non-empty collection of disjoint simple geodesics of $X$ whose union is closed in $X$, along with a transverse measure that is invariant as we flow along the geodesics. There is an exact correspondence between measured laminations and measured foliations (up to Whitehead equivalence). For a precise treatment of this correspondence between foliations and laminations, see Levitt \cite{Lev83}. We denote the space of (geodesic) measured laminations on $S$ by $\ML(S)$. In analogy with $\PMF(S)$, we define the space $\PML(S)$ of projective measured laminations to be $\ML(S)$ modulo scaling of the measure. We will use the identifications $\MF(S) \cong \ML(S)$ and $\PMF(S) \cong \PML(S)$ without further comment.

%Given a measured foliation $\fol$ on $S$, we can canonically associate a measured lamination $\lam$ by removing singular leaves from $\fol$ and pulling non-singular leaves tight to geodesics for some hyperbolic structure $X\cong S$. The measure on $\lam$ is induced by the measure on $\fol$ in a natural way. We denote the space of (geodesic) measured laminations on $S$ by $\ML(S)$. The space $\PML(S)$ of projective measured laminations is defined analogously to $\PMF(S)$: it is $\ML(S)$ modulo scaling of the measure. For a more precise treatment of the correspondence between foliations and laminations, see \cite{Lev83}.

Let $\C^0(S)$ be the vertex set of $\C(S)$. Endowing every curve with the counting measure embeds $\C^0(S)$ as a subset of $\ML(S)$. By the above correspondence, we also have a natural embedding 
%    , and in fact there is also a natural way to canonically associate a curve to a measured foliation, so we also have 
$\C^0(S)\subset \MF(S)$. Thurston proved that the projectivization of $\C^0(S)$ is dense in both $\PMF(S)$ and $\PML(S)$. Furthermore,  $\PMF(S)$ and $\PML(S)$ are compact \cite{FLP12}. 
%    Going forward we will frequently abuse notation by speaking of curves in $\C(S)$, when we really mean curves in $\C^0(S)$, and similarly for arcs.

A \define{filling lamination} is one that intersects every (essential) curve. 
(We will also call the corresponding foliations filling.)
The space of \define{ending laminations} of $S$, denoted $\el(S)$, is obtained by restricting to the subset of $\ML(S)$ consisting of filling laminations, and quotienting by forgetting the measures. Hence, $\el(S)$ is a quotient of a subspace of $\ML(S)$.
This space plays an important role in the theory of Kleinian groups; see \Cref{sec:elc}.

Finally, we say that a filling lamination (or the corresponding measured foliation) is  \define{uniquely ergodic} if the underlying topological lamination supports a unique projective measure class.
The subspace of uniquely ergodic foliations is denoted $\mc{UE}(S)\subset \PMF(S)$.

\subsection{Intersection pairing}\label{Sec:Intersection}

Given two vertices $a,b\in \AC^0(S)$, the \emph{geometric intersection number} of $a$ and $b$ is defined to be the minimal number of intersections between any pair of curves/arcs representing $a$ and $b$. In symbols, 
\[
i(a,b)=\min_{\alpha \in a, \, \beta \in b} |\alpha \cap \beta | . 
\]
Thurston showed that this function extends uniquely to a continuous, homogeneous function $i \colon \MF(S)\times \MF(S) \to \RR$, also called the \define{geometric intersection number}. See \cite{Thu78, bonahon1988geometry}. 
%If $\fol_1,\fol_2\in \MF(S)$ are foliations, then we set $i(\fol_1,\fol_2):=i(\lam_1,\lam_2)$ where $\lam_1$ and $\lam_2$ are the laminations corresponding to $\fol_1$ and $\fol_2$, respectively.
%In this context, if $a\in \AC(S)\subset\ML(S)$ is a closed curve or an arc, then $i(\fol,a)$ is exactly the length of $a$ in the transverse measure on $\fol$. 

For a quadratic differential $q$, recall the vertical and horizontal measured foliations $\fol_q^+$ and $\fol_q^-$. For $a\in \AC(S)$, let $h_q(a)$ denote the (horizontal) length of $a$ with respect to the transverse measure on $\fol_q^+$. Similarly, 
 $v_q(a)$ denotes the (vertical) length of $a$ with respect to the transverse measure on $\fol_q^-$. Then the \define{$\ell^1$ length} of $a$ with respect to the flat structure induced by $q$ is $\ell^1_q(a)=h_q(a)+v_q(a)$. The intersection pairing $i(\cdot,\cdot)$ satisfies
\[
h_q(a) = i(\fol_q^+, a) \quad \text{ and } \quad v_q(a) = i(\fol_q^-, a).
\]
Hence, $\ell^1_q(\cdot) = i(\fol_q^+, \cdot) +  i(\fol_q^-, \cdot)$ extends to a continuous function on $\mathcal{MF}(S)$. 
%\comm{DF: I do not understand the rest of this paragraph. Is the point to establish the continuity of $\ell^1_q(\alpha)$ in both $q$ and $\alpha$, for use in  \Cref{Claim:finiteness}?\\
%Yes, to establish continuity. --S \\
%Then can you clarify the continuity argument? - D} \\
In \Cref{sec:convergence}, we will need the stronger observation that the pairing 
\[
i^1 \colon \QD(S) \times \mathcal{MF}(S) \to \mathbb{R}.
\]
given by $(q,\fol) \mapsto \ell^1_q(\fol)$ is continuous in both parameters. This follows from the continuity of the intersection pairing along with the fact \cite{HM79} that  the assignment $q \mapsto (\fol_q^+, \fol_q^-)$ induces a homeomorphism $\QD(S) \to \mathrm{Fill}^2 \subset \mathcal{MF}(S) \times  \mathcal{MF}(S)$, where $\mathrm{Fill}^2 = \{(\fol_1,\fol_2) : i(\alpha, \fol_1) +  i(\alpha, \fol_2) > 0 \text{ for all } \alpha\in\C(S) \}$. 

% the $\ell^1$-norm extends to a continous pairing:
%\[
%i^1 \colon \QD(S) \times \mathcal{ML}(S) \to \mathbb{R}.
%\]
%Here, $\mathrm{Fill}^2 = \{(\fol_1,\fol_2) : i(\alpha, \fol_1) +  i(\alpha, \fol_2) > 0 \text{ for all } \alpha\in\AC(S) \}$. Note that we use the notation $i^1$ to distinguish this map from the usual intersection pairing $i(q, \fol) = \ell_q(\fol)$. That is, we are considering $\ell^1$ rather than $l^2$ length.

%    
%    \comm{Should we have a sub-section on ELC? The argument in favor is that a quadratic differential $q$ defines a pair of ending laminations and an associated hyperbolic manifold $N$. This is the manifold that is being triangulated in \Cref{thm:veering}. The argument against is that this section is already heavy, and this material is only needed in \Cref{sec:elc}.}

\subsection{Veering triangulations}\label{Sec:VeeringBackground}
We close this background section with a description of Gu\'eritaud's construction of veering triangulations \cite{Gue16}. Before giving the details, we note that this was not the original construction. Agol's original definition used periodic train track splitting sequences associated to the invariant foliations of a pseudo-Anosov map \cite{Ago11}. A very quick combinatorial characterization of veering triangulations appears in \cite{HRST11}. See also \cite{FG13, MT17} for other perspectives.

%    Given a mapping class $\varphi\in \Mod(S)$, we can construct the mapping torus 
%    \[
%    M_\varphi=\frac{S\times [0,1]}{(x,1)\sim (\varphi(x),0)}
%    \]
%     by thickening the surface $S$ then gluing the top of $S\times I$ to the bottom via $\varphi$. By Thurston's hyperbolization theorem for fibered manifolds, the mapping torus $M_\varphi$ is hyperbolic exactly when the monodromy $\varphi$ is pseudo-Anosov. That is, there are measured foliations $\fol^\pm$ of $S$ such that $\varphi (\fol^+) = \lambda_\varphi \fol^+$ and $\varphi (\fol^-) = \lambda_\varphi^{-1}\fol^-$, where $\lambda_\varphi > 1$ is the stretch factor of $\varphi$. 
%    
%    Veering triangulations were first constructed by Agol \cite{Ago11} using periodic train track splitting sequences associated to the invariant foliations of a pseudo-Anosov map. We will describe here an alternate construction due to Gu\'eritaud \cite{Gue16}, which will be better suited to our particular needs. In general, the construction produces an ideal triangulation on a punctured version of $M_\varphi$ which we now describe.

%    \vspace {5mm}

%% thesis %%
%, as described in \Cref{chap:ESVT}. We will describe here an alternate construction due to Gu\'eritaud, which will be better suited to our particular needs going forward.
%%%%%%%%%%%%

%% paper %%
%. We will describe here an alternate construction due to Gu\'eritaud, which will be better suited to our particular needs.
%%%%%%%%%%%

Let $S$ be a surface, and let $q \in \QD(S)$ be a quadratic differential. Let $\Scirc$ be the complement of the singularities of $q$. Then $\fol^-$ and $\fol^+$, the horizontal and vertical foliations of $q$, have singularities only at punctures of $\Scirc$. Recall that $q$ defines a singular flat metric on $S$, which restricts to an incomplete metric on $\Scirc$. A \define{saddle connection} of $q$ is a geodesic arc in the singular flat metric on $S$, with singularities at the endpoints but no singularities in its interior. Every saddle connection naturally yields an arc in $\Scirc$.
For the following construction, we will assume that $q$ has \emph{no horizontal or vertical saddle connections}; that is, no saddle connection is a leaf of $\fol^\pm$.

Consider an immersed rectangle $\mathcal{R} \to S$, with horizontal boundary mapped to $\fol^-$, vertical boundary mapped to $\fol^+$, and interior mapped to $\Scirc$. It follows that the interior of $\mathcal{R}$ must miss all singularities of $\fol^\pm$. We call $\mathcal{R}$ a \define{maximal (singularity-free) rectangle} of $q$ if it is maximal with respect to inclusion. Since there are no horizontal or vertical saddle connections, every side of a maximal rectangle must meet exactly one puncture of $\Scirc$. Observe that the punctures of $\Scirc$ must lie at interior points of edges: if a puncture occurred at a corner, $\mathcal{R}$ could be extended, violating maximality. See \Cref{fig:maximal_rect} and \cite[Figure 2]{MT17}.

\begin{figure}
	\centering
	\includegraphics[scale=1]{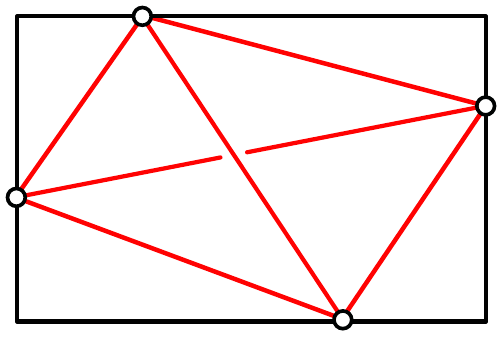}
	\caption{Gu\'eritaud's construction: a maximal singularity-free rectangle $\mc{R}$ defines an oriented ideal tetrahedron in $\Scirc \times \RR$ with a projection to $\mathcal{R}$.}
	\label{fig:maximal_rect}
\end{figure}

Every maximal singularity-free rectangle $\mc{R}$ defines an oriented tetrahedron $\tet$ with a map $\tet \to \mc{R}$, as follows. The vertices of $\tet$ map to the four preimages of punctures in $\bdy \mathcal{R}$. The edges of $\tet$ map to the six saddle connections spanned by these four vertices. The orientation of $\tet$ is determined by the convention that the more-vertical edge (whose endpoints are on the horizontal edges of $\mathcal{R}$) lies above the more-horizontal edge. See \Cref{fig:maximal_rect}.

Performing this construction for all maximal rectangles
gives a countable collection of tetrahedra whose vertices map to punctures of $\Scirc$. If tetrahedra $\tet$ and $\tet'$ contain the same triple of saddle connections (equivalently, if maximal rectangles $\mathcal{R}$ and $\mathcal{R}'$ intersect along a sub-rectangle that meets three punctures), we glue $\tet$ to $\tet'$ along their shared face. By a theorem of Gu\'eritaud \cite{Gue16} (see also \cite[Theorem 2.1]{MT17}), the resulting $3$--complex is an ideal triangulation $\tri_q$ of $\Scirc \times \RR$:

\begin{theorem}[Gu\'eritaud]\label{Thm:Gueritaud}
The complex of tetrahedra associated to maximal rectangles of $q$ is an ideal triangulation $\tri_q$ of $\Scirc \times \RR$. The maps of tetrahedra to their defining rectangles piece together to form a fibration $\pi \from \Scirc \times \RR \to \Scirc$. 
%    
%    If $q$ corresponds to a pseudo-Anosov $\varphi \from X\to X$, then the action of $\varphi$ on $(X,q)$ lifts simplicially and $\pi$-equivariantly to $\Phi:X\times \RR \to X\times \RR$. The quotient is a triangulation $\tri$ of the mapping torus $M_\varphi$.
\end{theorem}

We call $\tri_q$ the \define{veering triangulation associated to $q$}.
Observe that a saddle connection of $q$ corresponds to an edge of $\tri_q$ if and only if it spans a singularity-free rectangle. This is because every singularity-free rectangle can be expanded to a maximal one. 
%    See \cite[Figure 2]{MT17}.

Now, suppose that the quadratic differential $q$ corresponds to a pseudo-Anosov homeomorphism $\varphi \from S \to S$. Restricting $\varphi$ to the punctured surface $\Scirc$ produces a pseudo-Anosov $\phicirc \from \Scirc \to \Scirc$. Then $\phicirc$ permutes the (maximal) singularity-free rectangles of $q$, and therefore acts simplicially and $\pi$-equivariantly on the ideal triangulation $\tri_q$ of $\Scirc \times \RR$. Consequently, $\tri_q$ projects to an ideal triangulation of the \define{punctured mapping torus}
\[
\mathring M_\varphi = M_\phicirc = \:
 \quotient{\Scirc \times [0,1]} {(x,1) \!\sim\! (\phicirc(x),0)}.
\]
The resulting \define{veering triangulation of $M_\phicirc$} is denoted $\tri_\varphi$.

%    
%    Given an arbitrary pseudo-Anosov homeomorphism $\varphi \colon S \to S$, the associated quadratic differential $q_\varphi$ may have singularities that are not at punctures (in fact this is the generic behavior). However, if we puncture $S$ at the singularities of $\varphi$ to obtain a new surface $\Scirc$, $\varphi$ restricts to a homeomorphism $\phicirc \colon \Scirc \to \Scirc$ which is also a pseudo-Anosov and whose associated quadratic differential $q_{\phicirc}$ is punctured at all singularities. We can then apply \Cref{thm:veering} to get a veering triangulation of the mapping torus $M_{\phicirc}$.
%    In this context, we call the pair $(\Scirc, \phicirc)$ \define{fully punctured} since the singularities of $\phicirc$ occur at punctures of $\Scirc$.  In general, the veering triangulation associated to $\varphi$ is by definition the ideal triangulation of $M_\phicirc$ constructed above.

%% !TEX root =rndm_veering.tex

\section{Convergence of quadratic differentials}
\label{sec:qd_convergence}

In this section, we establish a statement about convergence of quadratic differentials that will form a key component for proving \Cref{thm:convergence}. This statement requires a handful of definitions.

%For a pair of parametrized geodesics $\gamma_1(t)$ and $\gamma_2(t)$ in a metric space $X$, we say that $\gamma_2$ \define{$B$--fellow travels} with $\gamma_1$ for distance $D$ centered at $x=\gamma_1(t_0)$ if we have $d_X(\gamma_1(t),\gamma_2(t))<B$ whenever $d_X(\gamma_1(t),x)<D/2$ .

For a pair of geodesics $\gamma_1$ and $\gamma_2$ in a metric space $X$, we say that $\gamma_2$ is a $\rho$--\define{fellow traveler} with $\gamma_1$ for distance $D$ centered at $x=\gamma_1(t_0)$ if we have $d_X(\gamma_1(t),\gamma_2(t))<\rho$ whenever $d_X(\gamma_1(t),x) \leq D/2$, for some unit speed parametrizations of $\gamma_1$ and $\gamma_2$.

%    Recall that every conformal structure $X\in\Teich(S)$ defines a unique isotopy class of hyperbolic metrics on $S$.
For a given constant $\epsilon>0$, let $\Teich_\epsilon(S)$ denote the set of all $X\in\Teich(S)$
%     whose shortest geodesic in the hyperbolic metric has length less than $\epsilon$.
such that the hyperbolic metric defined by $X$ contains a closed geodesic shorter than $\epsilon$.
The complement $K_\epsilon := \Teich(S)\setminus \Teich_\epsilon(S)$ is called the $\epsilon$--\define{thick part} of Teichm\"uller space. Let $\gamma_0$ be a Teichm\"uller geodesic in a thick part of Teichm\"uller space, and let $X,Y\in \mathrm{Teich}(S)$ be two points on $\gamma_0$. By Masur's Criterion \cite{masur1992hausdorff}, the horizontal and vertical foliations of $\gamma_0$ are uniquely ergodic.
%(since $\gamma_0$ is in the thick part). 
Let $B_r(X)$ and $B_r(Y)$ be balls of radius $r$ about $X$ and $Y$, respectively. Define $\Gamma_r(X,Y)$ to be the set of all oriented geodesics $\gamma$ passing first through $B_r(X)$, then through $B_r(Y)$, such that the vertical and horizontal foliations $\fol^+$ and $\fol^-$ associated to $\gamma$ are uniquely ergodic. Recall that the space of uniquely ergodic foliations on $S$ is denoted by $\mathcal{UE}(S)$. By a result of Hubbard and Masur \cite{HM79}, any pair of uniquely ergodic foliations in $\PMF(S)\cong \del \Teich(S)$ determine a unique Teichm\"uller geodesic, so we can also think of $\Gamma_r(X,Y)$ as a subset of $\mathcal{UE}(S)\times \mathcal{UE}(S)$.

%a filling lamination (or the corresponding measured foliation) is  \define{uniquely ergodic} if the underlying topological lamination supports a unique projective measure class

%    The following result is a reformulation of a theorem of Gadre and Maher \cite[Theorem 1.1]{GM17}. 
%    Since our theorem statement is quite different from theirs, and not easily extracted from their paper,
%    we assemble the proof for completeness. 
% ---------
%Since our formulation of their theorem is quite different from theirs, 
%This proof is essentially the one given by \cite{GM17}, with some modifications to fit our purposes. 

\begin{theorem}[Gadre--Maher]\label{thm:GM}
Let $g\in \mathrm{Mod}(S)$ be a principal pseudo-Anosov,  with invariant geodesic $\gamma_g$. Let $\mu$ be a probability distribution on $\Mod(S)$ with finite first moment, such that $\semigroup{\mu}$ is non-elementary and contains $g$, and fix $D>0$.
Then there exists $\rho>0$ such that, for almost every 
bi-infinite sample path $\omega=(\omega_n)$, there is a positive integer $N$ such that for $n\ge N$, $\omega_n$ is a principal pseudo-Anosov  whose Teichm\"uller geodesic $\gamma_{\omega_n}$ is a $\rho$--fellow traveler with $h_n \gamma_g$ for distance $D$, for some $h_n \in\Mod(S)$.
\end{theorem}

In the above theorem, the statement that $\omega_n$ is principal appears in the statement of Gadre and Maher's  \cite[Theorem 1.1]{GM17}. The statement that $\gamma_{\omega_n}$ fellow travels with a translate of $\gamma_g$ forms a key ingredient in Gadre and Maher's proof that $\omega_n$ is principal. The claim that the fellow-traveling distance $D$ can be taken arbitrarily large follows from examining their argument, but is not explicitly stated. Since our application (\Cref{cor:q-conv}) requires fellow traveling for longer and longer distances, we write down a unified proof of \Cref{thm:GM} by reassembling many of the same tools used by Gadre and Maher. We remark that a similar theorem was obtained independently by Baik--Gekhtman--Hamenst\"adt \cite[Theorem 6.8]{baik2016smallest}.

\begin{proof}[Proof of \Cref{thm:GM}]
Fix a basepoint $X$ on the Teichm\"uller geodesic $\gamma_g$. By a theorem of Kaimanovich and Masur \cite{KM96}, it is almost surely true that $\omega_nX$ and $\omega_{-n}X$ converge to distinct uniquely ergodic measured foliations $\fol^+_\omega$ and $\fol^-_\omega$ as $n\to \infty$. Let $\gamma_\omega$ be the unique Teichm\"uller geodesic determined by these foliations, parametrized by arclength so that $\gamma_\omega(0)$ is the closest point on $\gamma_\omega$ to the basepoint $X$.

In the following argument, we will first  establish fellow traveling between $\gamma_\omega$ and $\gamma_{\omega_n}$ for large $n$. Then we establish fellow traveling between  $\gamma_\omega$ and a carefully chosen translate of $\gamma_g$. This will imply fellow traveling between $\gamma_{\omega_n}$ and the translate of $\gamma_g$, which will also imply that $\omega_n$ is principal. As the proof involves many constants, we point the reader to \Cref{fig:FT} for a sketch of how the ideas fit together.

Let $\epsilon>0$ be small enough so that $\gamma_g$ is in the $\epsilon$--thick part $K_\epsilon$. 
Let $l > 0$ be the drift of the random walk.
Given this $\epsilon$, we have the following proposition, originally proved by Dahmani and Horbez  \cite[Theorem 2.6]{DH15}. The formulation below appears in \cite[Proposition 3.1]{GM17} and holds for any fixed $\epsilon>0$:

\begin{proposition}{\cite[Proposition 3.1]{GM17}}\label{prop:DH}
	There are constants $F>0$ and $0<e_0<\frac{1}{2}$
%let $F_0>0$ and $0<e_0<\frac{1}{2}$ be the constants from \cite[Proposition 3.1]{GM17}. For 
such for almost every $\omega$, and $n$ sufficiently large, there are points $Y_0$ and $Y_1$ on $\gamma_{\omega_n}$ and points $\gamma_\omega(T_0)$, $\gamma_\omega(T_1)$ on $\gamma_\omega$ such that 
\begin{itemize}
\item[(1)] $d_\Teich(\gamma_\omega(T_i),Y_i) \leq F\,\,$ for $i=0,1$.
\item[(2)] $0 \leq T_0\le e_0ln \le (1-e_0)ln\le T_1\le ln$
\item[(3)] $\gamma_\omega(T_i)$ is in the thick part $K_\epsilon$ for $i=0,1$.
\end{itemize}
\end{proposition}

See the red box in \Cref{fig:FT} for an illustration.
%    That such constants $F$, $e_0$ exist was originally shown in the proof of \cite[Theorem 2.6]{DH15}.
 With both $\epsilon$ and $F$ now fixed, we can apply the following theorem of Rafi \cite[Theorem 7.1]{Raf14}:

\begin{theorem}{\cite[Theorem 2.3]{GM17}}\label{thm:Rafi}
For any constants $\epsilon>0$ and $F\ge 0$, there is a constant $B=B(\epsilon,F)$ such that if $[Y,Z]$ and $[Y',Z']$ are two Teichm\"{u}ller geodesics, with $Y$ and $Z$ in the $\epsilon$--thick part, and 
$$
d_\Teich(Y,Y')\le F \quad and\quad  d_\Teich(Z,Z')\le F,
$$
then $[Y,Z]$ and $[Y',Z']$ are parametrized $B$--fellow travellers.
\end{theorem}

Since \Cref{prop:DH} says that $d_\Teich(\gamma_\omega(T_i),Y_i) \leq F$ and $\gamma_\omega(T_i)$ is in the $\epsilon$--thick part for $i\in\{1,2\}$, this theorem guarantees  $B(\epsilon,F)$--fellow traveling between the segments $[Y_0,Y_1]\subset \gamma_{\omega_n}$ and $[\gamma_\omega(T_0),\gamma_\omega(T_1)]\subset \gamma_\omega$, for $n$ sufficiently large (see \Cref{fig:FT}). 
The following lemma due to Gadre and Maher is key:

\begin{lemma}{\cite[Lemma 4.1]{GM17}}\label{lem:GM}
Let $g$ be a pseudo-Anosov in the support of $\mu$ with invariant Teichm\"{u}ller geodesic $\gamma_g$. Then there is a constant $r>0$ such that for every $Y,Z\in\gamma_g$, the probability $\mathbb{P}(\gamma_\omega \in \Gamma_r(Y,Z))$ is strictly positive.
\end{lemma}

Now, if we apply \Cref{thm:Rafi} with $F$ replaced by the constant $r$ given by \Cref{lem:GM}, then we get a constant $B'=B'(\epsilon,r)$, which guarantees that any geodesic in $\Gamma_r(Y,Z)$ contains a sub-segment that $B'$--fellow travels with $\gamma_g$ on the entirety of $[Y,Z]$.

%By \cite[Lemma 4.1]{GM17}, there is a constant $r >0$ such that for every $Y,Z \in \gamma_g$, the probability $\mathbb{P}(\omega\in \Gamma_r(Y,Z))$ is strictly positive. 
%Applying \cite[Theorem 2.3]{GM17} again, there is a constant $\rho_0 = \rho_0(\epsilon, r)$, such that any geodesic in $\Gamma_r(Y,Z)$ contains a sub-segment that $\rho_0$--fellow travels with $\gamma_g$ on the entirety of $[Y,Z]$.

The following result is observed in the course of proving \cite[Proposition 4.3]{GM17}:
 
 \begin{proposition}\label{prop:GM}
 	Let $g$ be a pseudo-Anosov such that $\gamma_g$ is in the principal stratum. For any $\rho>0$, there is a constant $D_1=D_1(\rho,g)>0$ such that any geodesic that $\rho$--fellow travels with $\gamma_g$ for distance greater than $D_1$ also lies in the principal stratum.
 \end{proposition}
%     $\rho=2\max\{B,\rho_0\}$, 
%and let  $D_1>0$ be the constant from \cite[Proposition 4.3]{GM17}, which ensures that geodesics that $\rho$--fellow travel with $\gamma_g$ for distance greater that $D_1$ lie in the principal stratum. 
 
For our application, set $\rho = B + B'$. Without loss of generality, assume that the constant $D$ in the statement of the theorem is larger than $D_1$.
% DF: reworded.
%    Now fix any $D\ge D_1$. 
 Let $k \in \NN$ be the smallest positive integer such that 
$$
d_{\Teich}(g^{-k}X,g^kX)\ge D_0:=D+\rho.
$$
By our choice of $B' = B'(\epsilon,r)$, any geodesic in $\Gamma_r(g^{-k}X,g^kX)$ will $\rho$--fellow travel with $\gamma_g$ on an interval of length $D_0$ centered at $X$.
Let $\Omega\subset \mathrm{Mod}(S)^{\ZZ}$ consist of those sample paths $\omega$ such that the sequences $\omega_{-n} X$ and $\omega_n X$ converge to distinct uniquely ergodic foliations $(\fol^-,\fol^+)\in \Gamma_r(g^{-k}X,g^kX)$. By \Cref{lem:GM}, the subset $\Omega$ has positive probability $P$.

\begin{figure}
\centering
\includegraphics[width=\textwidth]{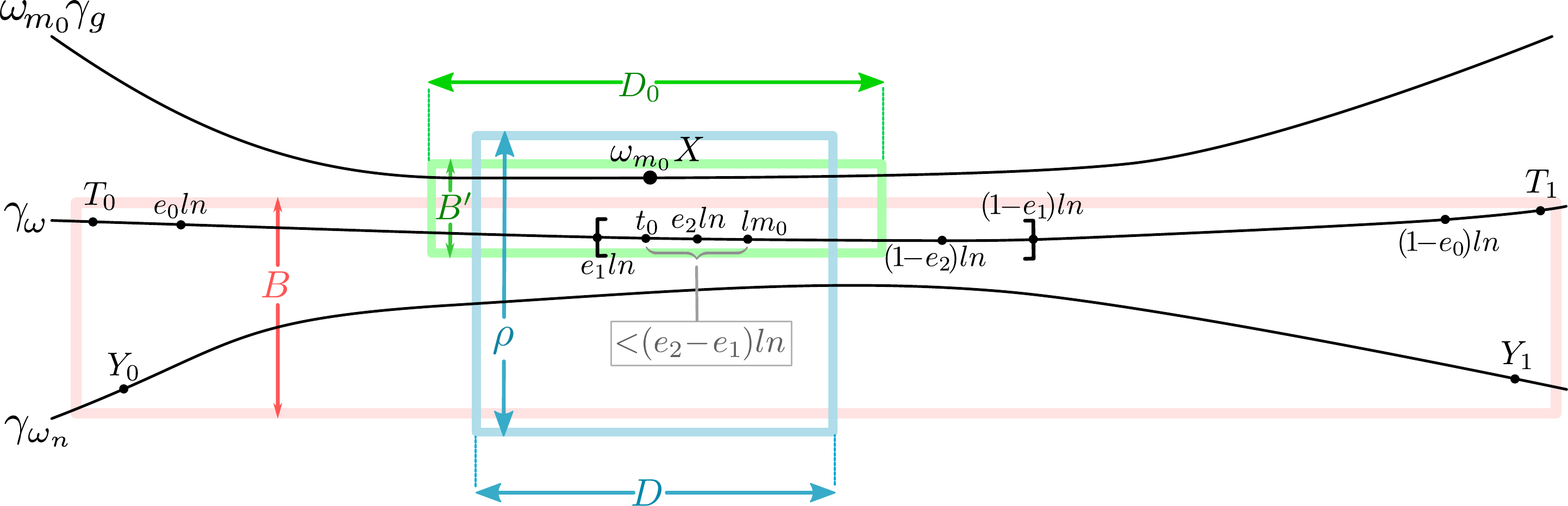}
\caption{ 
The logical structure of the proof of \Cref{thm:GM}. First, establish $B$--fellow traveling of $\gamma_\omega$ and  $\gamma_{\omega_n}$  (red). Then, establish $B'$--fellow traveling between $\gamma_\omega$ and a translate of $\gamma_g$ (green), on an interval contained in the red interval. This implies $\rho$--fellow traveling between $\gamma_{\omega_n}$ and the translate of $\gamma_g$ (blue).
Note that the points marked on the geodesic $\gamma_\omega$ are really $t$--values in the unit-speed parametrization: e.g. $e_0ln$ should be $\gamma_\omega(e_0ln)$.
}
\label{fig:FT}
\end{figure}

Let $\sigma \from \mathrm{Mod}(S)^{\ZZ} \to \mathrm{Mod}(S)^{\ZZ}$ be the shift map. Ergodicity of $\sigma$ implies that for almost every $\omega$, there is some $m\ge 0$ such that $\sigma^m(\omega) \in \Omega$. Since $\omega_m\gamma_{\sigma^m(\omega)}=\gamma_\omega$, it follows that for such $m$, $\gamma_\omega$ is a $B'$--fellow traveler with  $\omega_m\gamma_g$ for distance $D_0$, centered at $\omega_mX$.

%    the subsegment of $\gamma_\omega$ of length $D_0$, centered at the closest point on $\gamma_\omega$ to the point $\omega_nX$, $\rho_0$--fellow travels with $\omega_n\gamma_g$.

For almost every $\omega$, the proportion of $m\in\{1,\dots,n\}$ that satisfy $\sigma^m(\omega)\in \Omega$ tends to $P$ as $n\to\infty$. For any $e \in (e_0, \frac{1}{2})$, 
%    satisfying $e_0<e_1<\frac{1}{2}$,
 the proportion of $m$ in the range $en\le m \le (1-e)n$ that satisfy $\sigma^m(\omega)\in \Omega$ also tends to $P$ as $n \to \infty$. This implies that given $\omega$ and $e \in (e_0, \frac{1}{2})$, there exists $N$ such that for all $n\ge N$, there exists an $m$ such that $en\le m \le (1-e)n$ and $\sigma^m(\omega)\in \Omega$.

By sublinear tracking in Teichm\"uller space, due to Tiozzo \cite{Tio15},  almost every $\omega$ satisfies
$$
\lim_{m\to\infty} \frac{1}{m} d_\Teich(\omega_m X,\gamma_\omega(lm))=0.
$$
(Recall that we have given $\gamma_\omega$ a unit-speed parametrization so that $\gamma_\omega(0)$ is the closest point to $X$).
%    where $\gamma_\omega$ is parametrized so that $\gamma_\omega(0)$ is the closest point on $\gamma_\omega$ to the basepoint. 
Now, fix $e_1$ and $e_2$ so that $e_0<e_1<e_2<\frac{1}{2}$.                                     
%$\frac{1}{2}-e_1\ll e_1-e_0$. 
After possibly replacing $N$ with a larger number, we can assume that for any $n\ge N$ there exists $m_0$ such that
%$d_{\Teich}(\omega_mX,\gamma_\omega(lm))\le \frac{1}{4}(e_1-e_0)lm$ for all $m\ge e_1n$, and that $(e_1-e_0)ln \ge 2D_0$, for $n\ge N$. We may also assume that $B'\le \frac{1}{4}(e_1-e_0)ln$ for $n\ge N$.
%Combining the preceding two paragraphs, 
\begin{itemize}
    \item[(1)] $e_2ln\le lm_0 \le (1-e_2)ln\le ln$ and $\sigma^{m_0}(\omega)\in \Omega$.
    \item[(2)] $d_{\Teich}(\omega_{m_0}X,\gamma_\omega(lm_0))\le \frac{1}{2}(e_2-e_1)lm_0$.
\end{itemize}
Furthermore, we may also assume that $N$ is large enough (hence, $n$ is large enough) that  $D_0\le (e_1-e_0)ln$ and $B'\le \frac{1}{2}(e_2-e_1)ln$. Define $h_n\coloneq \omega_{m_0}$. For the sake of clarity, we will continue to denote this mapping class by $\omega_{m_0}$ as we finish the proof.

Since $\sigma^{m_0}(\omega)\in \Omega$, 
% DF: added
\Cref{thm:Rafi} implies that 
$\gamma_\omega$ and $\omega_{m_0}\gamma_g$ are $B'$--fellow travelers for distance $D_0$ centered at a point $p=\gamma_\omega(t_0)$ on $\gamma_\omega$, and centered at $\omega_{m_0}X$ on $\omega_{m_0}\gamma_g$. 
Since $d_{\Teich}(p,\omega_{m_0}X)\le B'\le \frac{1}{2}(e_2-e_1)ln$, and (2) implies
$$
d_{\Teich}(\omega_{m_0}X,\gamma_\omega(lm_0))\le \frac{1}{2}(e_2-e_1)lm_0\le \frac{1}{2} (e_2-e_1)ln,
$$
we have by the triangle inequality that $d_{\Teich}(p,\gamma_\omega(lm_0))\le (e_2-e_1)ln$. Since $e_2ln \le lm_0 \le (1-e_2)ln$, it follows that $e_1ln \le t_0 \le (1-e_1)ln$.
%Item (2) above guarantees that the point $\gamma_\omega(lm_0)$ is within $2(e_2-e_1)ln$ of the closest point $p$ (on $\gamma_\omega$) to $\omega_{m_0}X$. This is because $d_\Teich (\omega_{m_0}X,p)\le d_\Teich(\omega_{m_0}X,\gamma_\omega(lm_0)\le (e_2-e_1)ln$. The key here is that, by our assumption that $\frac{1}{2}-e_1\ll e_1-e_0$, the difference $|e_2ln-e_1ln|$ is small compared to $|e_1ln-e_0ln|$. Since the parametrization of $\gamma_\omega$ is unit-speed, this guarantees that $\gamma_\omega(e_1ln)$ is much closer to $\gamma_\omega(e_2ln)$ than it is to $\gamma_\omega(e_0ln)$ (and similarly for $(1-e_i)$). This means that $p$ must be much closer to $\gamma_\omega(lm_0)$ than to either of $\gamma_\omega(e_0ln)$ or $\gamma_\omega((1-e_0)ln)$. 
Our requirement that $D_0\le (e_1-e_0)ln$ therefore ensures that $d_{\Teich}(p,\gamma_\omega(e_0ln))\ge D_0$ and $d_{\Teich}(p,\gamma_\omega((1-e_0)ln))\ge D_0$. Since $p$ is the center of the $B'$--fellow traveling between $\gamma_\omega$ and $\omega_{m_0}\gamma_g$, this assures that this distance $D_0$ fellow traveling happens fully inside the range of $B$--fellow traveling between $\gamma_{\omega_n}$ and $\gamma_\omega$. See the green box in  \Cref{fig:FT}. 

Now, recall that we used \Cref{prop:DH} to find $T_0,T_1$ satisfying
$$
T_0\le e_0ln \le (1-e_0)ln\le T_1
%    \le lN
$$
so that the interval $[\gamma_\omega(T_0),\gamma_\omega(T_1)]$ is a $B$--fellow traveler with $\gamma_{\omega_n}$. It follows that $\gamma_{\omega_n}$ is a $\rho$--fellow traveler with $ \omega_{m_0}\gamma_g = h_n \gamma_g$ for a distance $D_0-(B+B')\ge D_0-\rho=D$. Since $D\ge D_1$, \Cref{prop:GM} implies that $\gamma_{\omega_n}$ lies in the principal stratum.  
\end{proof}

%In light of the above theorem, which asserts that random walks on $\Mod(S)$ converge to pseudo-Anosovs with Teichm\"uller geodesics in the principal stratum, we make the following definition: if the the quadratic differential associated to a pseudo-Anosov $\varphi$ is in the principal stratum $\GQD(S)$, then we will say that $\varphi$ is \define{generic}. 

We will use \Cref{thm:GM} in the form of \Cref{cor:q-conv} below. First, we need the
following lemma, which says that if a geodesic $\gamma$ fellow travels the axis of a pseudo-Anosov for sufficiently long, then $\gamma$ gets arbitrarily close to the axis. If this pseudo-Anosov axis
% lies in the principal stratum (i.e., has generic singularities),
is principal, then the openness of the principal stratum implies $\gamma$ will be principal as well.

\begin{lemma} \label{lem:qd_converge}
Let $g$ be a pseudo-Anosov mapping class with axis $\gamma_g$.
Fix $\rho > 0$, and suppose that $\gamma_n$ is a sequence of Teichm\"uller geodesics such that $h_n\gamma_n$ is a $\rho$--fellow traveler with $\gamma_g$ for distance $D_n$, where $h_n \in \Mod(S)$ and $D_n \to \infty$. Then there is a choice of quadratic differentials $q_n$ associated to points along $\gamma_n$ such that $h_n q_n$ converge to a quadratic differential $q$ associated to $\gamma_g$.
\end{lemma}

\begin{proof}
By replacing $\gamma_n$ with $h_n\gamma_n$ and translating by a power of $g$, we may suppose that $\gamma_n$ is a $\rho$--fellow traveler with $\gamma_g$ for a length $D_n$ subsegment of $\gamma_g$ centered at some point $s \in \gamma_g$. Let $q$ be the quadratic differential based at $s$ associated to $\gamma_g$. We first make the following claim:

\begin{claim}
There are $s_n \in \gamma_n$ such that $s_n \to s$.
\end{claim}

\begin{proof}[Proof of claim] 
If this were not the case, then after passing to a subsequence, $\gamma_n$ converges to a Teichm\"uller geodesic $\gamma$ with the properties that
\begin{enumerate}
\item $\gamma$ fellow travels $\gamma_g$, and
\item $s$ has distance at least $\delta > 0$ from $\gamma$.
\end{enumerate}
Now let $\gamma^+$ be a positive ray in $\gamma$. Since $\gamma^+$ stays bounded distance from (the positive end of) $\gamma_g$, it follows that $\gamma^+$ accumulates in $\PMF(S)$ to foliations that are topologically equivalent to the stable foliation of $g$  \cite[Theorem 3.8]{Mas10}. Since this foliation is uniquely ergodic, we see that $\gamma^+$ in fact converges to the stable foliation of $g$. Similarly, $\gamma^-$ converges to the unstable foliation of $g$. Hence, the vertical and horizontal foliations defining $\gamma$ and $\gamma_g$ agree, and so $\gamma$ and $\gamma_g$ are equal, up to a reparametrization. This, however, contradicts $(2)$, completing the proof of the claim.
\end{proof}

Returning to the proof of the lemma, let $q_n$ be the quadratic differential associated to $s_n \in \gamma_n$. Now we claim that $q_n \to q$ in $\QD(S)$. This follows exactly as in the proof of the claim: if not, then after passing to a subsequence, $q_n \to q' \neq q$ based at $s$. But then $\gamma_n$ would converge (uniformly on compact sets) to the Teichm\"uller geodesics determined by $q'$. 
Since we know that $\gamma_n$ converges to $\gamma$, this gives a contradiction and completes the proof.
\end{proof}

%\begin{remark}
%In the above lemma, it is only necessary that $A_f$ is lies in a fixed thick part of $\Teich(S)$.
%\end{remark}

%    \begin{remark}
%    Need to translate this statement to the Teichm\"uller space of the punctured surface $\mathring{S}$. Lifting from $\Teich(S)$ to $\Teich(\mathring{S})$ works for many reasons. One local reason, sufficient for our purposes, is that  the kernel of the Birman exact sequence is torsion-free. This lifting should probably happen later.
%    \end{remark}

Recall that $\GQD(S)$ denotes the principal stratum of quadratic differentials on $S$. 
%\comm{See discussion on notation in \Cref{sec:background}.}

%    \comm{Notation doesn't seem good. Elsewhere $\mathcal{P}$ means projective.\\
%    Agreed. $\QD_p$ seems like a better choice, except that we use that for something else. Could we denote fully puncture quadratic differentials by $\QD^\circ$, or is that too many circles floating around? -W \\
%    Or perhaps $\QD\mathcal{P}(S)$? -D}

\begin{corollary}\label{cor:q-conv}
Let $g\in \mathrm{Mod}(S)$ be a principal pseudo-Anosov with Teichm\"uller axis $\gamma_g$. Let $\mu$ be a probability distribution on $\Mod(S)$ with finite first moment, such that $\semigroup{\mu}$ is non-elementary and contains $g$.
 Then for almost every sample path $\omega=(\omega_n)$ in $\mathrm{Mod}(S)$, there is a positive integer $N$ such that for $n\ge N$, every $\omega_n$ is a principal pseudo-Anosov and $h_n q_{\omega_n} \to q_g$ in $\GQD(S)$, where $h_n\in \Mod(S)$ and $q_g$ is some quadratic differential along the axis $\gamma_g$.
\end{corollary}

\begin{proof}
Let $\rho > 0$ be the number guaranteed by \Cref{thm:GM}.
Let $PPA \subset \Mod(S)$ be the set of principal pseudo-Anosovs and define the set 
\begin{align*}
\mathcal{G}_D = \{\omega : \: &\exists N \ge0 \text{ such that } \forall n \ge N, \, \omega_n \in PPA \: \text{ and }   \\
	& \gamma_{\omega_n} \text{ is a  $\rho$--fellow traveler with a translate of $\gamma_g$ for distance } D\}.
\end{align*}
By \Cref{thm:GM}, $\mathbb{P} (\mathcal{G}_D) = 1$ for all $D$. Since $\mathcal{G}_{D'} \subset \mathcal{G}_{D}$ for $D \le D'$, we set $\mathcal{G} = \bigcap \mathcal{G}_D$ and conclude that $\mathbb{P}(\mathcal{G}) =1$.

Now for each $\omega \in \mathcal{G}$, there is a sequence of mapping classes $h_n$, such that $h_n \gamma_{\omega_n}$ is a $\rho$--fellow traveler with $\gamma_{g}$ for distance $D_n$, where $D_n \to \infty$ as $n \to \infty$. Applying \Cref{lem:qd_converge}  and recalling that $\GQD(S)$ is open in $\QD(S)$ completes the proof.
\end{proof}

\begin{remark}\label{Rem:q-conv-ptorus}
The only property of the principal stratum  that we used in this section is that $\GQD(S)$ is open. As a consequence, all of the results in this section hold for $S \cong \Sigma_{1,1}$, with $\GQD(S)$ replaced by $\QD(S)$. 
In particular,  \Cref{cor:q-conv} applies to every pseudo-Anosov in $\Mod(\Sigma_{1,1})$, with the word ``principal'' excised.
\end{remark}

%% !TEX root =rndm_veering.tex

\section{Transition to punctured surfaces}\label{sec:transition}

%    Notation for this section: 
%    \begin{itemize}
%    
%    \item $S = S_{g,n}$ is a hyperbolic surface, possibly with punctures.
%    \item $\Scirc = S_{g,n+m}$ is the result of puncturing $S$ at $m$ additional points.
%    %    \item $S = S_{g,n}$ is the surface obtained by filling some number of punctures of $\Scirc$, so $n \leq m$.
%    \item $\GQD(S)$ is the principal stratum of quadratic differentials.
%    \item $\Teich(S)$ and $\Teich(\Scirc)$ are Teichm\"uller spaces.
%    \item $\Torelli(S) = \ker \left( \Mod(S) \to Sp(2g, \ZZ) \right)$, the subgroup that acts trivially on $H_1(S_{g,0})$. It is torsion-free. (Proof: for $S = S_{g,0}$, see \cite[Theorem 6.8]{FM12}. For $S_{g,n}$, induct up the Birman exact sequence using the torsion-freeness of $\pi_1$.)
%    \end{itemize}

Recall from \Cref{Sec:VeeringBackground} that the construction of a veering triangulation starts with a quadratic differential $q \in \QD(S)$,  punctures $S$ along the singularities of $q$ to obtain a surface $\Scirc$, and then builds an ideal triangulation of $\Scirc \times \RR$. We need to analyze the veering triangulations not just for one $q \in \QD(S)$, but for an entire convergent sequence $q_n \to q$ that comes from \Cref{cor:q-conv}. To do this, we need a coherent way to map the sequence $q_n \to q$ to a convergent sequence $\qcirc_n \to \qcirc \in \QD(\Scirc)$.

%    Consider a surface with punctures $\Scirc$ be a hyperbolic surface with punctures. Let $S$ be the surface obtained by filling some number of punctures of $\Scirc$. Thus 

Let $\QD_p(\Scirc)$ denote the  subspace of $\QD(\Scirc)$ consisting of quadratic differentials whose singularities occur only at punctures. Then $\QD_p(\Scirc)$ is a union of strata.
If $\qcirc \in \mathcal{QD}_p(\Scirc)$ is a quadratic differential   with at least $2$ prongs at any puncture being filled in $S$, then $\qcirc$ defines a quadratic differential $q \in \mathcal{QD}(S)$. Implicit in this definition is the observation that markings on $\Scirc$ induce markings on $S$.  The following is immediate:

\begin{lemma}\label{Lem:FillQuadDiff}
Let $Q \subset \mathcal{QD}_p(\Scirc)$ be a stratum of  $\mathcal{QD}(\Scirc)$. Let $S$ be the result of filling some number of punctures of $\Scirc$, so that a representative element $\qcirc \in \Scirc$ has at least $2$ prongs at every puncture being filled. Then the assignment $\qcirc \mapsto q \in \mathcal{QD}(S)$ defines a continuous map $g \from Q \to \mathcal{QD}(S)$, whose image is a stratum.
 \qed
%    Let $\qcirc_n \to \qcirc$ be a convergent sequence in a single stratum of $\mathcal{QD}(\Scirc)$, with all singularities at the punctures. Let $S$ be the result of filling some number of punctures of $\Scirc$, where these quadratic differentials have at least $2$ prongs. Then the sequence $\qcirc_n \to \qcirc$ induces a convergent sequence $q_n \to q \in \mathcal{QD}(S)$. \qed
\end{lemma}

We need to go in the opposite direction, puncturing $S$ at singularities of $q \in \mathcal{QD}(S)$ to obtain $\Scirc$. This is is not as straightforward, because  the surjection $\Mod(\Scirc) \to \Mod(S)$ has a large kernel, hence there is no consistent way to turn markings on $S$ into markings on $\Scirc$. Nevertheless, this can be done locally in the principal stratum.

\begin{lemma}\label{lem:DrillQuadDiff}
Let $q \in \GQD(S)$ be a quadratic differential in the principal stratum. Let $\Scirc$ be the result of puncturing the singularities of $q$. Then there is an open neighborhood $U$ of $q$, with an embedding $f \from U \to \mathcal{QD}_p(\Scirc)$ such that $g \circ f = \operatorname{id}_{U}$ where $g$ is the map of \Cref{Lem:FillQuadDiff}.
%    
%    Let $q_n \to q \in \mathcal{QD}(S)$ be a convergent sequence of quadratic differentials, all in the principal stratum. Let $\Scirc$ be the result of puncturing $S$ along the singularities of $q$. Then there is a convergent sequence $q_n^\circ \to \qcirc \in \mathcal{QD}_p(\Scirc)$  corresponding to $q_n \to q$. That is,  the map of \Cref{Lem:FillQuadDiff} has a local section. 
\end{lemma}

\begin{proof}
Let $X(q) \in \Teich(S)$ be the marked conformal structure underlying $q$. Let  $y_1^q, \ldots, y_k^q \in X(q)$ be the singularities of $q$. Let $\epsilon > 0$ be such that there are pairwise disjoint regular neighborhoods $N_{\epsilon}(y_1^q), \ldots, N_\epsilon(y_k^q)$.

%    
%    Let $a_1, b_1, \ldots, a_g, b_g$ be a collection of smooth curves on $S = S_{g,n}$ that form a geometric symplectic basis for $H_1(S_{g,0})$. See \cite[Figure 6.1]{FM12}. We draw these curves on $X(q)$ so that they miss all singularities of $q$. Let $\epsilon > 0$ be such that $N_\epsilon(y_i^q) \cap (a_j \cup b_j) = \emptyset$ for all $i,j$, and furthermore these $\epsilon$-neighborhoods are pairwise disjoint.

Now, let $q' \in \GQD(S)$ be another quadratic differential in the principal stratum, with singularities $y_1^{q'}, \ldots, y_k^{q'}  \in X(q)$. There is a unique Teichm\"uller map $h \from X(q') \to X(q)$ which maps the singularities of $q'$ to a $k$--tuple of points $h \big( y_1^{q'} \big), \ldots, h\big(y_k^{q'} \big) \in X(q)$. Because $h$ is uniquely defined by the pair $(q', q)$, these points of $X(q)$ are uniquely determined up to reordering. Thus there is an open neighborhood $U$ of $q$ such that for $q' \in U$, 
the singularities of $q'$ can be ordered so that $h\big( y_i^{q'} \big) \in N_\epsilon(y_i^q)$, for a unique point $y_i^q$.

Let $\Scirc = X(q) \setminus \{ y_1^q,  \ldots y_k^q \} $. For every $q' \in U$, we will define a marked conformal structure on $\Scirc$, as follows. Let $\mathring{X}(q') = X(q') \setminus \{ y_1^{q'}, \ldots, y_k^{q'} \}$. This conformal structure is marked by the composition map
\begin{align*}
\Scirc = X(q) \setminus \{ y_1^q,  \ldots , y_k^q \} 
&\xrightarrow{\:\:\: r \:\:\: } X(q) \setminus \{ h \big( y_1^{q'} \big), \ldots, h\big(y_k^{q'} \big) \} \\
&\xrightarrow{\: h^{-1} } X(q') \setminus \{ y_1^{q'}, \ldots, y_k^{q'} \} \qquad \quad = \mathring{X}(q').
%    
%    Y(q') = X(q') \setminus \{ y_1^{q'}, \ldots, y_k^{q'} \}
%    &\xrightarrow{\:\: h \:\: } X(q) \setminus \{ h \big( y_1^{q'} \big), \ldots, h\big(y_k^{q'} \big) \} \\
%    &\xrightarrow{\:\: r \:\: } X(q) \setminus \{ y_1^q,  \ldots , y_k^q \} \qquad \quad = \Scirc,
\end{align*}
where $r$ is the identity on the complement of $N_{\epsilon}(y_1^q) \cup \ldots \cup N_\epsilon(y_k^q)$. The composition $h^{-1} \circ r$ is well-defined up to isotopy because the mapping class group of a punctured disk is trivial.

Now, the quadratic differential $q'$ on the marked Riemann surface $X(q')$  restricts to a quadratic differential ${\qcirc}'$ on the marked Riemann surface $\mathring{X}(q')$. By construction, all singularities of  ${\qcirc}'$ are at the punctures, hence  ${\qcirc}' \in \mathcal{QD}_p(\Scirc)$. The map $f \from U \to \mathcal{QD}_p(\Scirc)$ defined via   $q' \mapsto {\qcirc}'$ is continuous by construction. It is one-to-one because the map $g$ of \Cref{Lem:FillQuadDiff} provides an inverse.
\end{proof}

%    
%    \comm{This proof does not give a local section. Revise.}
%    
%    \begin{proof}
%    The convergent sequence $q_n \to q \in \mathcal{QD}(S)$ descends to a convergent sequence $\overline q_n \to \overline q \in \mathcal{MQD}(S) = \mathcal{QD}(S)/\Mod(S)$, the moduli space of \emph{unmarked} quadratic differentials.
%    
%    Let $x_1, \ldots, x_k \in S$ be the singularities of $q$. Then $\Scirc = S \setminus \{ x_1, \ldots, x_k \} $. Since $q$ is in the principal stratum, it has a trivalent singularity at every $x_i$. Furthermore, every $q_n$ has trivalent singularities at a $k$--tuple of points in bijective correspondence with $x_1, \ldots, x_k$. 
%    
%    Since we have forgotten marking data, the moduli space $\mathcal{M}(S)$ can be identified with a subspace of $\mathcal{M}(\Scirc)$, and similarly the principal stratum of $\mathcal{MQD}(S)$ can be identified with a subspace of $\mathcal{MQD}(\Scirc)$. This inclusion gives a convergent sequence $\overline q_n^\circ \to \overline \qcirc \in \mathcal{MQD}(\Scirc)$.
%    
%    Finally, recall that $\mathcal{MQD}(\Scirc) = \mathcal{QD}(\Scirc)/ \Mod(\Scirc)$, where the action of $\Mod(\Scirc)$ is discrete. Thus we can take a consistent sequence of preimages $q_n^\circ \to  \qcirc \in \mathcal{QD}(\Scirc)$.
%    \end{proof}
%    
%    \comm{Be more careful about this convergence.}

For the next two sections, we will work primarily in the punctured surface $\Scirc$.

%% !TEX root =rndm_veering.tex

\section{Convergence of veering triangulations}\label{sec:convergence}

Let $\Scirc$ be a surface with at least one puncture.
The main result of this section, \Cref{Cor:ComplexEmbeds}, says that veering triangulations of $\Scirc \times \RR$ depend continuously on their defining quadratic differentials. More precisely, we will show that given an appropriate convergent sequence $q_n \to q \in \QD(\Scirc)$,  the corresponding veering triangulations $\tri_{q_n}$ agree with $\tri_q$ on larger and larger finite sets of tetrahedra, limiting to the entire triangulation $\tri_q$.

Recall from \Cref{sec:transition} that $\QD_p(\Scirc)$ is the  subspace of $\QD(\Scirc)$ consisting of quadratic differentials whose singularities occur at punctures of $\Scirc$.
%      Note that $\QD_p(\Scirc)$ is a disjoint union of strata of $\QD(\Scirc)$. 
We define $\e\QD_p(\Scirc) \subset \QD_p(\Scirc)$ to be the subspace of quadratic differentials without vertical or horizontal saddle connections. 
In \Cref{Sec:VeeringBackground}, these are exactly the quadratic differentials on $\Scirc$ that define veering triangulations of $\Scirc \times \RR$. The symbol $\mathcal{E}$ stands for ``ending;'' see the discussion following \Cref{Thm:HomeoELC}.

The following easy lemma will be useful in \Cref{sec:elc}.

%    We record the following lemma for the reader's convenience.

\begin{lemma} \label{lem:fill}
For every $q \in \e\QD_p(\Scirc)$, the foliations $\F^+_q$ and $\F^-_q$ are filling. 
\end{lemma}

\begin{proof}
Suppose for a contradiction that $\F = \F^+_q$ is not filling. Then there is some closed essential curve $\alpha \subset \Scirc$ with $i(\alpha, \F) = 0$. The $q$--geodesic representative $\alpha_q$ of $\alpha$ is a concatenation of saddle connections (see \cite{Rafi2005} or \cite{DLR10}) and since $i(\alpha, \F) = 0$, each of these saddle connections must be vertical, a contradiction. The proof for $\F^-_q$ is identical.
\end{proof}

As in \Cref{Sec:VeeringBackground}, for each $q \in \e\QD_p(\Scirc)$ we have an associated veering triangulation $\tri = \tri_q$ of $\Scirc \times \RR$. (Note that no further puncturing is necessary because all singularities are already at the punctures of $\Scirc$.)
Let $\A(\tri) = \A(\tri_q)$ be the subset of $\A(\Scirc)$ consisting of arcs that correspond to edges of $\tri_q$.
As described immediately after \Cref{Thm:Gueritaud}, 
 $\A(\tri_q)$ is precisely the set of saddle connections of $q$ that span singularity free rectangles.

%    \begin{remark}
%    There are two other possibilities for the definitions of $\QD_p$, which are smaller. First, we could consider quadratic differential whose singularities happen exactly at the punctures. Second, we could require that the quadratic differentials have exactly the singularity data we get by puncturing a generic quadratic differential (here we just need to specify which punctures are zeros and which are poles). This last definition is the smallest space, but it is a sub-bundle of $\QD(\Scirc)$. I chose the above definition of $\QD(\Scirc)$ just because it is the largest.
%    \end{remark}
%    

%\begin{remark}
%We want that $\QD_p$ is closed in $\QD$ for local compactness. (This is used in the argument below.) There are two other possibilities for the definitions of $\QD_p$ which are more likely to be closed, but which are smaller. First, we could consider quadratic differential whose singularities happen exactly at the punctures. Second, we could require that the quadratic differentials have exactly the singularity data we get by puncturing a generic quadratic differential (here we just need to specify which punctures are zeros and which are poles). This last definition is the smallest space, but it is definitely closed since its a sub-bundle (I think).
%\end{remark}

Let $a, a_1, \ldots ,a_n \in \A(\Scirc)$ be arcs. We call the collection $\{a_1, \ldots, a_n\}$ a 
\define{homotopical decomposition} of $a$, and write  $a \sim \sum_i a_i$,  if these arcs have lifts $\widetilde a, \widetilde a_1, \ldots ,\widetilde a_n$ to the universal cover of $\Scirc$ which bound an  immersed ideal $(n+1)$-gon (which is degenerate if $n=1$). The decomposition is \define{nontrivial} if $n > 1$.

Recall from \Cref{Sec:Intersection} that the horizontal and vertical lengths of $a$ are denoted $h_q(a)$ and $v_q(a)$, whereas $\ell^1_q(a) = h_q(a) + v_q(a)$ is the total $\ell^1$ length.

\begin{lemma} \label{lem:veering_detect}
%Let $q \in \QD_p(\Scirc)$ be a quadratic differential with singularities at punctures
%and no vertical/horizontal saddle connections. 
Let $q \in \e\QD_p(\Scirc)$ and $a\in \A(\Scirc)$. Then $a \in \A(\tri_q)$ if and only if for any nontrivial homotopical decomposition $a \sim \sum a_i$ with $a_i \in \A(\Scirc)$, we have
\begin{equation}\label{Eqn:ShorterThanSum}
\ell_q^1(a) < \sum \ell_q^1(a_i).
\end{equation}
\end{lemma}

\begin{proof}
Suppose that $a$ is homotopic to an edge $\sigma$ of the veering triangulation and $a \sim \sum a_i$.  Since $\sigma$ spans a singularity free rectangle, the total horizontal or vertical length of the $a_i$ must be strictly greater than that of $\sigma$. This is because, after lifting to the universal cover of $\Scirc$, all $\ell^1$ geodesics between the endpoints of $\sigma$ must lie in the rectangle spanned by $\sigma$.
As we always have $h_q(a) \le \sum h_q(a_i)$ and $v_q(a) \le \sum v_q(a_i)$, the strict inequality \eqref{Eqn:ShorterThanSum} follows.

The converse direction follows from the work of Minsky and Taylor \cite{MT17}. 
%They observe that 
First recall that every arc $a$ has a unique $q$--geodesic representative $a_q$. See \cite[Proposition 2.2 and Figure 6]{MT17}. This geodesic $a_q$ follows a sequence of saddle connections, which we may call $a_1, \ldots, a_n$, such that $a \sim \sum a_i$. Since $a_q$ is a geodesic, we have
\[
\ell_q^1(a) = \ell_q^1(a_q)  = \sum \ell_q^1(a_i).
\]
Thus we have proved the negation of \eqref{Eqn:ShorterThanSum}, 
unless $a$ is itself homotopic to a saddle connection $c$, i.e., the sum $\sum a_i$ has only one term.

Now, suppose that $a = c$ is a saddle connection that is not an edge of $\tri_q$. Then $c$ does not span a singularity free rectangle of $q$. Hence, $c$ does not span a singularity free right triangle to one of its sides. To this side, we apply the map $\mathbf{t}$ that is defined in \cite[Section 4.2]{MT17}. 
The resulting object $\mathbf{t}(c)$ is a concatenation of (not necessarily disjoint) saddle connections $c_j$, forming a non-trivial decomposition $c \sim \sum_j c_j$. By \cite[Lemma 4.2]{MT17}, these saddle connections have the property that, working in the universal cover of $\Scirc$, each leaf of the vertical/horizontal foliation of $q$ meets the union $\bigcup_j c_j$ at most once. (The reader can see this property 
illustrated in \cite[Figure 12]{MT17}.)
Hence, 
\[
\ell_q^1(a) = \ell_q^1(c)= \sum \ell_q^1(c_j)
\]
and the sum is non-trivial, contradicting \eqref{Eqn:ShorterThanSum}.
\end{proof}

%    In what follows, we will only need that for each $a \in \A(\Scirc)$ the map $q \to \ell^1_q(a)$ is continuous on $\QD(\Scirc)$.

\begin{lemma} \label{lem:edges}
%Let $q \in \QD_p(\Scirc)$ be a quadratic differential with singularites at punctures
%and no vertical/horizontal saddle connections. 
Fix $q \in \e\QD_p(\Scirc)$.
For any $L\ge0 $, there is an open neighborhood $U$ of $q$ in $\QD(\Scirc)$ such that for any $q' \in U \cap \e\QD_p(\Scirc)$, every arc $\sigma \in \A(\tri_q)$ of length $\ell^1_q(\sigma) \le L$ is also in $\A(\tri_{q'})$.
\end{lemma}

\begin{proof}
For $q \in \e\QD_p(\Scirc)$ and $L\ge 0$, define $A_{q}(L) = \{a \in \A({\Scirc}) : \ell^1_q(a) \le L\}$. Note that $A_{q}(L)$ is always finite.
Now fix $L \ge0$ and let 
\[
U_1 = \{q' \in \QD(\Scirc) : \ell^1_{q'}(a) < L+1 \text{ for all } a \in A_q(L) \}.
\]
This is an open neighborhood of $q$ in $\QD(\Scirc)$. After making $U_1$ smaller if necessary, we can ensure that the closure $\overline U_1 \subset 
 \QD(\Scirc)$  is compact. 
This is done for the following claim:

\begin{claim}\label{Claim:finiteness}
The set 
\[
B = \big\{ a\in \A({\Scirc}) :  \ell^1_{q'}(a) \le L+1 \text{ for some } q' \in U_1 \big\}
\]
is finite. 
\end{claim}

\begin{proof}
For any arc $a \subset \Scirc$, there is an essential (multi-)curve $c_a$ constructed as follows. Consider the punctures of $\Scirc$ to be marked points in a larger surface $S$; build a regular neighborhood $P$ of $a$ and the marked points that it meets; then, take the $\Scirc$--essential components $\bdy P$. 
We remark that $P \cap \Scirc$ is a pair of pants containing $a$ as its only $\Scirc$--essential arc, hence $c_a$ determines $a$.
%    (There are at most two boundary components, at least one of which is essential in $\Scirc$.)
%    which is formed from the essential boundary components of a regular neighborhood of the union of $a$ with small circles enclosing the punctures at its endpoints. 
For any $q$, we have $\ell^1_q(c_a) \le 2\cdot \ell^1_q(a)$ because a 
representative of $c_a$ is given by traversing the $q$--geodesic representative for $a$ at most twice. 

Now suppose that the claim is false. Then there would be an infinite collection $a_i \in B$ and $q_i \in U_1$ with $\ell^1_{q_i}(a_i) <L+1$. Setting $c_i = c_{a_i}$, we obtain an infinite collection of distinct multi-curves $c_i$ with  $\ell^1_{q_i}(c_i) <2(L+1)$.
Since $\overline U_1$ is compact, we may pass to a subsequence such that
 $q_i \to q'$ for some  $q' \in \overline U_1$. Passing to a further subsequence and using compactness of $\PMF(\Scirc)$, there are constants $x_i \ge 0$ such that $x_ic_i$ converges in $\mathcal{MF}(\Scirc)$ to $\alpha \neq 0$. 
It is also easy to see that  $x_i \to 0$ as $i \to \infty$. Indeed, for an arbitrary (but fixed) hyperbolic metric $\rho$ on $S$, $x_i c_i \to \alpha$ implies that $x_i \ell_\rho(c_i) \to \ell_\rho(\alpha) \in \R_+$.  Since there are infinitely many distinct multi-curves $c_i$, we must have $\ell_\rho(c_i) \to \infty$, hence $x_i \to 0$.

Recall from \Cref{Sec:Intersection} that the $\ell^1$--length $\ell^1_q(\alpha)$ is continuous in both $q$ and $\alpha$. Thus
\[
%    \ell^1_q(\alpha) = i^1(q, \alpha) = \lim x_i \cdot i^1(q_i,c_i) \le 2(L+1) \lim x_i = 0.
 i(\fol_{q'}^+, \alpha) +  i(\fol_{q'}^-, \alpha) \: = \:  \ell^1_{q'}(\alpha) \: = \: \lim_{i \to \infty} \ell^1_{q_i}(x_i c_i) \: = \: \lim_{i \to \infty} x_i \, \ell^1_{q_i}(c_i) \: \leq \: 2(L+1) \lim x_i \: = \: 0.
\]
However, a measured foliation $\alpha$ cannot have 
%    zero $\ell^1$--length with respect to $q'$, since otherwise $\alpha$ would have 
intersection number $0$ with both $\fol_{q'}^+$ and $\fol_{q'}^-$, a filling pair of foliations.
This contradiction completes the proof of the claim.
\end{proof}

We now return to the proof of the lemma. For each $a \in \A(\tri_q) \cap A_{q}(L)$, we define the function $f_a \from U_1 \to \RR$:
\[
q' \in U_1 \: \longmapsto \: f_a (q') = \min \left\{ \sum_i \ell^1_{q'}(a_i) - \ell^1_{q'}(a) 
\: : \: a \sim \sum_i a_i \text{ is nontrivial and } a_i \in B
\right\}.
\]
%    where the minimum is over all non-trivial homotopical decompositions $a \sim \sum_i a_i$ with 
%    $a_i \in B$.
Since $B$ is finite, this is a minimum of finitely many continuous functions, hence $f_a$ is continuous on $U_1$. Furthermore, since $a \in \A(\tri_q)$, \Cref{lem:veering_detect} implies that $f_a(q) > 0$. Set $U_a = U_1 \cap \{q' : f_a(q') >0\}$, which is open in $\QD(\Scirc)$ and contains $q$.

Finally, define
\[
%    U = \bigcap \big\{ U_a : a \in \A(\tri_q) \cap A_{q}(L) \big \},
U = \bigcap_{ a \in \A(\tri_q) \cap A_{q}(L) } U_a,
\]
which by construction is an open neighborhood of $q$ in $\mathcal{QD}(\Scirc)$. To show that this neighborhood satisfies the conclusion of the lemma, let $q' \in U \cap \e\QD_p(\Scirc)$ and suppose that $\sigma \in \A(\tri_q)$ with $\ell^1_q(\sigma) \le L$. If $\sigma$ is \emph{not} in $\tri_{q'}$, then by \Cref{lem:veering_detect} there is a decomposition $\sigma \sim \sum \sigma_i$ with $\ell^1_{q'}(\sigma) = \sum \ell^1_{q'}(\sigma_i)$. Since $U \subset U_1$, we have that $\ell^1_{q'}(\sigma) < L+1$ and so similarly $\ell^1_{q'}(\sigma_i)< L+1$ for all $i$. Hence, by definition of $B$, we have $\sigma_i \in B$ for each $i$. Then the difference 
\[
\sum \ell^1_{q'}(\sigma_i) - \ell^1_{q'}(\sigma)
\]
appears in the definition of $f_\sigma(q')$. Since $f_\sigma(q') >0$, the difference $\sum \ell^1_{q'}(\sigma_i) - \ell^1_{q'}(\sigma)$ is strictly positive, contradicting the choice of our decomposition of $\sigma$. This completes the proof.
\end{proof}

\begin{remark}
An alternate proof of \Cref{Claim:finiteness} relies on the fact that there is a constant $K$, depending only on the compact set $\overline U_1$, such that for any $q_1,q_2 \in \overline U_1$ the induced map $(\widetilde{\Scirc}, \widetilde{q_1}) \to (\widetilde{\Scirc}, \widetilde{q_2})$ is a $K$--quasi-isometry. Then $B \subset A_q(K(L+1)+K)$, and the claim follows. 
\end{remark}

\Cref{lem:edges} has the following useful corollary:

\begin{corollary}\label{Cor:ComplexEmbeds}
%Let $q \in \QD_p(\Scirc)$ be a quadratic differential with singularites at punctures
%and no vertical/horizontal saddle connections. 
Let $q \in \e\QD_p(\Scirc)$, and let $K \subset \tau_q$ be any finite sub-complex. 
Then there is a neighborhood $U$ of $q$ in $\QD(\Scirc)$ such that
 $K \subset \tri_{q'}$ for every  $q' \in U \cap \e\QD_p(\Scirc)$.
\end{corollary}

\begin{proof}
Define
\[
L = \max \left\{ \ell_q^1(a) : a \in K^{(1)} \right\}.
\]
By \Cref{lem:edges}, there is a neighborhood $U$ such that every arc $\sigma \in \A(\tri_q)$ with $\ell^1_q(\sigma) \leq L$ also belongs to $\A(\tri_{q'})$ for $q' \in U \cap \e\QD_p(\Scirc)$. In particular, every arc in the $1$--skeleton of $K$ belongs to $\A(\tri_{q'})$. Since the edges of every tetrahedron $\tet \subset K$ belong to $\A(\tri_{q'})$, we have $\tet \subset \tri_{q'}$.
\end{proof}

\begin{remark}
The referee pointed out an alternative line of argument for \Cref{Cor:ComplexEmbeds}. Quadratic differentials near $q$ can be locally parametrized using the shapes of polygons that define the underlying translation surface. 
In these coordinates, the property $\sigma \in \A(\tri_{q'})$ can be characterized by finitely many linear inequalities, a property that persists on an open set. This is particularly easy to see using a proposition of Frankel \cite[Proposition 3.16]{Frankel:Comparison}, who proved that  a triangulation by saddle connections is veering if and only if every triangle has edges of both positive and negative slope. 
 \end{remark}

\section{Convergence of tetrahedron shapes}
\label{sec:elc}

%\comm{The following used to be in the background section. Incorporate it here.}

In this section, we prove \Cref{thm:convergence}. To do so, we establish a technical result
 (\Cref{Prop:TetraShapesConverge}) which roughly states that as quadratic differentials converge, so do the hyperbolic shapes of the associated veering tetrahedra. This result will also be used in the proof of \Cref{thm:counting} in \Cref{sec:counting}.
We begin by reviewing some needed background about doubly degenerate representations of surface groups.

%\subsection{Doubly degenerate representations} \label{sec:DD}

For a surface $S$, let $\AH(S)$ denote the space of conjugacy classes 
of discrete and faithful representations $\rho \from \pi_1(S,x_0) \to \mr{PSL}_2(\CC)$ such that peripheral elements map to parabolic isometries. In the algebraic topology on $\AH(S)$, conjugacy classes $[\rho_n]$ converge to $[\rho]$ if and only if there are conjugacy representatives $\rho_n \from \pi_1(S) \to \mr{PSL}(2,\CC)$ such that for every element $\gamma \in \pi_1(S)$, the images $\rho_n(\gamma)$ converge to $\rho(\gamma)$.
%    finite set of elements $\gamma_1, \ldots, \gamma_k \in \pi_1(S)$, the images $\rho_n(\gamma_i)$ converge to $\rho(\gamma_i)$ for every $i$. 

Setting $\Gamma_\rho=\rho(\pi_1(S,x_0))$, the manifold $N_\rho=\HH^3/\Gamma_\rho$ is then homeomorphic to $S\times \RR$ by work of Bonahon \cite{Bon86}. 
The \define{ends} of $N_\rho$ are the two components of $S\times (\RR \setminus I)$, where $I$ is an arbitrary compact interval.
The limit set $\Lambda_\rho$ is defined to be the smallest nonempty closed $\Gamma_\rho$--invariant set in $\bdy\HH^3$. 
The space $\DD(S)\subset \AH(S)$ of \define{doubly degenerate representations} of $\pi_1(S,x_0)$ is the subspace of $\AH(S)$ consisting of conjugacy classes
%$\Gamma_\rho = \rho(\pi_1(S,x_0))$
$[\rho]$
 such that $\Lambda_\rho=\del \HH^3$ and such that $\rho(\gamma)$ is parabolic if and only if $\gamma \in \pi_1(S,x_0)$ is peripheral. For such a $\rho$, the manifold $N_\rho$ has \define{geometrically infinite} ends, 
 which can be characterized by saying that for each end, there is a sequence of closed geodesics  in $N_\rho$ that exits that end.
%    That is, there exists a sequence of closed geodesics $\gamma_n\subset S\times\R$ such that $\gamma_n \cap S \times (-\infty, n] = \emptyset$, and similarly for the negative end $S \times [-n, \infty)$.

By work of Bonahon and Thurston \cite{Bon86,Thu78}, there are unique, distinct \define{end invariants} $\fol^+_\rho, \fol^-_\rho \in \el(S)$ associated to the positive and negative ends of $N_\rho$, such that if $\{ \alpha_i \}$ is a bi-infinite sequence of closed geodesics exiting both ends, then $ \alpha_i \to \fol^-$ as $i \to -\infty$ and $\alpha_i\to \fol^+$ as $i \to + \infty$.
%    the $\pm$ ends, then $\gamma_i^+ \to \fol^+$ and $\gamma^-_i\to \fol^-$. 
(Here, as in \Cref{sec:background:fols_and_lams}, we pass freely between foliations and laminations.)
Hence we get a well-defined function $\mathcal{E} \from \DD(S)\to \el(S)\times \el(S) - \diag$, where $\diag$ is the diagonal.
% (which must be removed since $\lam^+$ and $\lam^-$ are distinct).

Thurston conjectured that $\mathcal{E}$ is a bijection. This conjecture was proved by Minsky \cite{Min10} and Brock--Canary--Minsky \cite{BCM12}. Subsequently, Leininger and Schleimer observed that $\mathcal{E}$ is actually a homeomorphism \cite[Theorem 6.5]{LS09}.

%    For the proof, we will need the Ending Lamination Theorem of Brock--Canary--Minsky \cite{BCM12}. First, recall that $\dd(S,\partial S) \subset \mathrm{AH}(S,\partial S)$ is the subspace of doubly degenerate representations.  For $\rho \in \dd(S,\partial S)$, let 
%    $N_\rho = \mathbb{H}^3 / \rho(\pi_1(S))$ be the associated doubly degenerate manifold with end invariants $(\lam_\rho^-,\lam_\rho^+) \in \el(S) \times \el(S) - \diag$. 

\begin{theorem}[Ending Lamination Theorem, parametrized]\label{Thm:HomeoELC}
The end invariant function $\mathcal{E} \colon \dd(S) \to \el(S) \times \el(S) - \diag$ sending $\rho$ to the pair $(\fol_\rho^-,\fol_\rho^+)$ is a homeomorphism.
\end{theorem}

We now specialize to the case of the punctured surface $\Scirc$. 
%    We will only need one direction of \Cref{Thm:HomeoELC}: that $\mathcal{E}^{-1} \colon \el(\Scirc) \times \el(\Scirc) - \diag \to \dd(\Scirc)$ is continuous. 
%    That is, given a convergent sequence $(\fol^+_n, \fol^-_n) \in  \el(\Scirc) \times \el(\Scirc) - \diag$, the associated conjugacy classes $[\rho_n]$ of representations converge in the algebraic topology.
%    This means that there are conjugacy representatives $\rho_n: \pi_1(\Scirc) \to \mr{PSL}(2,\CC)$ such that for every finite set of elements $\gamma_1, \ldots, \gamma_k \in \pi_1(\Scirc)$, the images $\rho_n(\gamma_i)$ converge to $\rho(\gamma_i)$ for every $i$. 
%
 Recall from \Cref{sec:convergence} that
 $\e\QD_p(\Scirc)$ is the subspace of $\QD(\Scirc)$ consisting of quadratic differentials whose foliations have singularities only at punctures and which have no horizontal or vertical saddle connections. As described in \Cref{Sec:VeeringBackground}, every quadratic differential $q \in \e\QD_p(\Scirc)$ has an associated veering triangulation $\tri_q$ of $\Scirc \times \RR$. 

Every $q \in \e\QD_p(\Scirc)$ defines a doubly degenerate hyperbolic structure on $\Scirc \times \RR$,  constructed as follows. By \Cref{lem:fill}, there is a map  $\F \colon \e\QD_p(\Scirc) \to \el(\Scirc) \times \el(\Scirc) - \diag$ sending $q$ to the pair $(\fol^+_q, \fol^-_q)$ of filling foliations/laminations. 
%    (As discussed in \Cref{sec:background:fols_and_lams}, we pass freely between laminations and foliations without further comment.)
%    By \cite{HM79}, $\F$ is continuous. 
By \Cref{Thm:HomeoELC},  $\mathcal{E}^{-1}(\fol^+_q, \fol^-_q)$ is a doubly degenerate representation $\rho_q \from \pi_1(\Scirc) \to \mathrm{PSL}(2,\CC)$, unique up to conjugation. Then $N_q = \HH^3 / \rho(\pi_1 \Scirc)$ is a marked hyperbolic $3$--manifold triangulated via $\tau_q$. In summary,  the composition
\[
\e^{-1} \circ \F \: \colon \: \e\QD_p(\Scirc) \to \dd(\Scirc)
\]
maps $q$ to a conjugacy class of doubly degenerate representations, which we denote by $[\rho_q]$. This map $\e^{-1} \circ \F$ is $\Mod(S)$--equivariant by construction, and also continuous. Recall that $\e^{-1}$ is continuous  by \Cref{Thm:HomeoELC} and $\F$ is continuous by \cite{HM79}.

We remark that the end invariants $(\fol^+_q,\fol^-_q)$ of $N_q$ can be recovered directly from the edge set of $\tau_q$. By a theorem of Minsky and Taylor \cite[Theorem 1.4]{MT17}, the edge set $\A(\tau_q)$ is totally geodesic in $\AC(\Scirc)$. Furthermore, this edge set is quasi-isometric to a line, which has two endpoints at infinity. Any sequence of edges of $\A(\tau_q)$ whose slope in $q$ approaches $\infty$ will exit the positive end of $N_q$ and limit to $\fol^+_q$, while any sequence in $\A(\tau_q)$ whose slope in $q$ approaches $0$ will exit the negative end of $N_q$ and limit to $\fol^-_q$.

With this background, we can state and prove the main result of this section.

\begin{proposition} \label{Prop:TetraShapesConverge}
Fix $q \in \e\QD_p(\Scirc)$ and a finite, connected sub-complex $K \subset \tri_q$. 
Then the following holds for any convergent sequence $q_n \to q$, where $q_n \in  \e\QD_p(\Scirc)$:
%    \begin{enumerate}[$\bullet$]
\begin{itemize}
\item For all $n \gg 0$, $K$ is a sub-complex of  the veering triangulation $\tri_{q_n}$.
\item For every tetrahedron $\tet \subset K$, the shape of $\tet$ in $N_{q_n}$ converges to the shape of $\tet$ in $N_{q}$ as $n \to \infty$.
\end{itemize}
In particular, if $\tri_q$ is not geometric, then neither is $\tri_{q_n}$, for $n$ sufficiently large.
\end{proposition}

\begin{proof}
%    [Proof of \Cref{Prop:TetraShapesConverge}]
Fix $q \in \e\QD_p(\Scirc)$ and a finite, connected sub-complex $K \subset \tri_q$. We may assume without loss of generality that the dual $1$--skeleton of $K$ is connected (otherwise, add some number of tetrahedra). Let $Y$ be a maximal tree in the dual $1$--skeleton. Fix a base vertex $v_0 \in Y$, which corresponds to the barycenter of an oriented tetrahedron $\tet_0 \subset K$.

By construction, every vertex $v \in Y$ is a barycenter of some tetrahedron $\tet \subset K$, which has $4$ ideal vertices at punctures of $\Scirc$.
We use this structure to construct finitely many group elements in $\pi_1(\Scirc \times \RR, v_0)$.  For every vertex $v \in Y$, follow the unique path in $Y$ from $v_0$ to $v$. Walk from $v$ to the neighborhood of an ideal vertex of the ambient tetrahedron $\tet$, walk around the corresponding puncture of $\Scirc$, and then return to $v$ and back to $v_0$. 

This construction gives a collection of loops $\gamma_1, \ldots, \gamma_k$, where $k = 4 V(Y)$. Not all of these loops are homotopically distinct, but all of them are peripheral in $\Scirc$.
%
%    We distinguish three of these loops, labeled $\gamma_1, \gamma_2, \gamma_3$, which correspond to punctures of the base tetrahedron $\tet_0$.

Now, let $[\rho_q]$ be the conjugacy class of doubly degenerate representations corresponding to $q$. For every representation in this conjugacy class, the image of each peripheral element $\gamma_i$ must be parabolic, with a single fixed point on $\bdy \HH^3$. 
%    We fix an actual representation $\rho_q \in [\rho_q]$ by requiring the parabolic fixed points of $\gamma_1, \gamma_2, \gamma_3$ to be $0, 1, \infty$, respectively.

Next, consider a convergent sequence $q_n \to q$, where $q_n \in  \e\QD_p(\Scirc)$. By \Cref{Thm:HomeoELC}, there is a convergent sequence
\[
 \mathcal{E}^{-1} ( \mathcal{F}(q_n)) = [\rho_{q_n}] \to [\rho_q].
\]
After choosing a representative $\rho_q \in [\rho_q]$, this means there are choices of representatives $\rho_{q_n} \in [\rho_{q_n}]$ such that $\rho_{q_n}(\gamma_i) \to \rho_q(\gamma_i)$ for $1 \leq i \leq k$. 

%    In particular, the parabolic fixed points of $\rho_{q_n}(\gamma_1)$, $\rho_{q_n}(\gamma_2)$, $\rho_{q_n}(\gamma_3)$ converge to $0, 1, \infty$, respectively. Modifying every $\rho_{q_n}$ by an increasingly small isometry of $\HH^3$ (which converges to $I$ as $n \to \infty$), we can ensure that the parabolic fixed points of $\rho_{q_n}(\gamma_1)$, $\rho_{q_n}(\gamma_2)$, $\rho_{q_n}(\gamma_3)$ are exactly $0, 1, \infty$, while the convergence $\rho_{q_n}(\gamma_i) \to \rho_q(\gamma_i)$ is still preserved for every $i$.

Let $\tet \subset K$ be a tetrahedron, with ideal vertices $x_1, \ldots, x_4$. By the above construction, every $x_j$ corresponds to a peripheral group element $\gamma_{i_j}$ in the chosen collection. Let $p_j^n \in \bdy \HH^3$ be the parabolic fixed point of $\rho_{q_n}(\gamma_{i_j})$, and let $p_j$ be the parabolic fixed point of $\rho_q(\gamma_{i_j})$. Since $\rho_{q_n}(\gamma_i) \to \rho(\gamma_i)$, we also have convergence of the parabolic fixed points: $p_j^n \to p_j$ as $n \to \infty$.

For every $q_n$, let $\tri_{q_n}$ be the veering triangulation of $N_{q_n} \cong \Scirc \times \RR$. Since $q_n \to q$, \Cref{Cor:ComplexEmbeds} implies that $K$ embeds into $\tri_{q_n}$ for all $n \gg 0$.  (In fact, there is only one embedding consistent with the marking of $N_{q_n}$.) The shape parameter of $\tet$ in the hyperbolic metric on $N_{q_n}$ is the cross-ratio $[p_1^n, p_2^n, p_3^n, p_4^n]$. As $n \to \infty$, these cross-ratios converge to $[p_1, p_2, p_3, p_4]$, hence the shape of $\tet$ converges as well.
\end{proof}

\begin{remark}
\Cref{Prop:TetraShapesConverge} gives a concrete way to see that, with suitably chosen basepoints, the manifolds $N_{q_n}$ converge geometrically to $N_q$. Let $z \in N_q$ be an arbitrary basepoint, and let $B_R(z) \subset N_q$ be a metric $R$--ball about $z$. Since the edges of $\tau_q$ eventually exit the ends of $N_q$, there are only finitely many edges (hence finitely many tetrahedra) in $\tau_q$ that intersect $B_R(z)$. 
%    Let $K \subset \tau_q$ be a finite, connected complex containing those tetrahedra. 
By \Cref{Prop:TetraShapesConverge}, the shapes of these tetrahedra in $N_{q_n}$ converge to the shape in $N_q$, hence for $n \gg 0$, there is a metric ball in $N_{q_n}$ almost-isometric to $B_R(z)$.

The statement that the algebraic and geometric limits of $N_{q_n}$ agree is due to Canary \cite{Can96}, and is used in the proof of continuity in \Cref{Thm:HomeoELC}. Hence this remark does not give a new proof of geometric convergence.
\end{remark}

We can now complete the proof of \Cref{thm:convergence}.

\begin{proof}[Proof of \Cref{thm:convergence}]

Let $S$ be a hyperbolic surface, and let $\varphi \in \Mod(S)$ be a principal pseudo-Anosov.
Let $\mu$ be a probability distribution on $\Mod(S)$ with finite first moment, such that $\semigroup{\mu}$ is non-elementary and contains $\varphi$.
According to \Cref{cor:q-conv}, for almost every sample path $\omega=(\omega_n)$ of the random walk on $\Mod(S)$ we have for $n \gg 0$,
\begin{itemize}
\item[(1)] $\omega_n$ is a pseudo-Anosov with Teichm\"uller geodesic $\gamma_{\omega_n}$ in the principal stratum,
\item[(2)] $h_n q_{\omega_n}\to q_\varphi$ in $\GQD(S)$ for some $h_n\in\Mod(S)$ and for some quadratic differentials $q_{\omega_n}$ along $\gamma_{\omega_n}$ and $q_\varphi$ along $\gamma_\varphi$.	
\end{itemize}

\noindent Since $h_n q_{\omega_n}=q_{h_n \omega_n h_n^{-1}}$, and $\omega_n$ defines the same unmarked mapping torus as $h_n \omega_n h_n^{-1}$, the veering triangulations associated to the $\omega_n$ are simplicially isomorphic to those associated to $h_n \omega_n h_n^{-1}$. Let $q_n=h_n q_{\omega_n}$.

By \Cref{lem:DrillQuadDiff}, for sufficiently large $n$, we can pass from the sequence $q_n\to q_\varphi$ to a sequence $\mathring{q}_n\to \mathring{q}_\varphi$ in $\QD_p(\Scirc)$, where $\Scirc$ is the surface obtained by puncturing $S$ at the singularities of $q_\varphi$. By construction, $\mathring{q}_\varphi$ is a quadratic differential along the Teichm\"uller axis of $\phicirc \in \Mod(\Scirc)$, and similarly for the $\mathring{q}_n$. Thus, in fact, $\qcirc_n \to \qcirc_\varphi \in \e\QD_p(\Scirc)$.

 Let $\tri_n = \tri_{q_n}$ be the veering triangulation of $N_{\mathring{q}_n} \cong \Scirc \times \RR$ associated to the quadratic differential $\mathring{q}_n$, and let $ \tri_{q_\varphi}$ be the veering triangulation of  $N_{\mathring{q}_\varphi}$ associated to $\mathring{q}_\varphi$. Now let $K \subset  \tri_{q_\varphi}$ be any finite connected subcomplex as in the statement of the theorem.
 Applying \Cref{Prop:TetraShapesConverge}, we conclude that for $n$ sufficiently large, $K$ is a subcomplex of  the veering triangulation $\tri_{q_n}$ and that for every tetrahedron $\tet \subset K$, the shape of $\tet$ in $N_{\mathring{q}_n}$ converges to the shape of $\tet$ in $N_{\mathring{q}_\varphi}$ as $n \to \infty$. 

Hence, it only remains to show that $K$ embeds as a subcomplex of $\tau_\varphi$, the veering triangulation of the mapping torus $M_{\mathring{\omega}_n}$. That is, we must show that the covering map $N_{\mathring{q}_n} \to M_{\mathring{\omega}_n}$ is injective on $K$, once $n$ is sufficiently large. For this, we use a result of Maher \cite{Mah10}, which implies that the Teichm\"uller translation length of $\omega_n$ grows linearly in $n$. Since dilatation, and hence Teichm\"uller translation length, is unchanged after puncturing along singularities, we also have that the translation length of $\mathring{\omega}_n$ grows linearly in $n$. If $N_{\mathring{q}_n} \to M_{\mathring{\omega}_n}$ fails to be injective on $K$ then there are edges $k_1$ and $k_2$ of $K$ which, when viewed as arcs of $\Scirc$, satisfy $(\mathring{\omega}_n)^i k_1 = k_2$ for some $i > 0$. Since these arcs represent saddle connections of $\mathring{q}_n$, this implies that the stretch factor of $\mathring{\omega}_n$ is no more than the quantity
\[
\frac{\max_{k \in K^{(1)}} v_{q_n}(k)}{\min_{k\in K^{(1)}} v_{q_n}(k)},
\qquad
\text{which converges to}
\qquad
\frac{\max_{k \in K^{(1)}} v_{q_\varphi}(k)}{\min_{k\in K^{(1)}} v_{q_\varphi}(k)}
\]
as $n \to \infty$. Since this implies that the stretch factors of $\omega_n$ are eventually bounded, we obtain a contradiction and the proof is complete. \qedhere

%By construction of the map $f:U\to \QD_p(\Scirc)$ defined in \Cref{lem:DrillQuadDiff}, the veering triangulation $\tri_\varphi$ associated to $\mathring{q}_n$ is exactly $\tri$, which is non-geometric by assumption.
%It follows then from \Cref{Prop:TetraShapesConverge} that $\tri_n$ is non-geometric for $n$ sufficiently large. The conclusion follows.

\end{proof}

\begin{remark}
\Cref{thm:convergence} also holds for $S \cong \Sigma_{1,1}$, without the hypothesis that $\varphi$ is principal. Recall that the principal stratum of $\QD(\Sigma_{1,1})$ is empty. In this setting, \Cref{cor:q-conv} holds for every pseudo-Anosov $\varphi$. (See \Cref{Rem:q-conv-ptorus}.) There are no interior singularities, so $S = \Scirc$ and \Cref{sec:transition} is not needed. Now, the rest of the proof of \Cref{thm:convergence} using  \Cref{Prop:TetraShapesConverge} applies verbatim.
\end{remark}

\section{Counting non-geometric veering triangulations}\label{sec:counting}

In this section, we prove \Cref{thm:counting}, showing that geometric veering triangulations are atypical from the point of view of counting closed geodesics in moduli space. The proof of this result uses many of the same ingredients as the proof of \Cref{thm:convergence}. The main difference is that the appeal to Gadre and Maher's \Cref{thm:GM} will be replaced with results from Hamenst\"adt \cite{hamenstadt2013bowen} and Eskin--Mirzakhani \cite{eskin2011counting}. 

Fix a surface $S$ such that $\xi(S) \geq 1$. As in the introduction, let $\mathcal{G}(L)$ denote the set of conjugacy classes of pseudo-Anosov mapping classes in $\Mod(S)$ whose Teichm\"uller translation length is no more than $L \geq 0$. Since the veering triangulation $\tri_\varphi$ depends only on the conjugacy class of the pseudo-Anosov, each $[\varphi] \in \mathcal{G}(L)$ uniquely determines a veering triangulation of $\mathring M_\varphi$. As in Baik--Gekhtman--Hamenst\"adt \cite{baik2016smallest},  say that a \define{typical pseudo-Anosov conjugacy class} in $\Mod(S)$ has a property $\mathcal{P}$ if 
\[
\lim_{L \to \infty} \frac{|\left \{[\varphi] \in \mathcal{G}(L) \: : \: \varphi \text{ has } \mathcal{P} \right \}|}{|\mathcal{G}(L)|} = 1.
\]
In this terminology, \Cref{thm:counting} is implied by the following, slightly stronger statement.

\begin{theorem} \label{th:counting}
Let $S$ be a surface with $\xi(S) \ge 2$. Then a typical pseudo-Anosov conjugacy class  $[\varphi] \subset \Mod(S)$ is principal and defines a non-geometric veering triangulation $\tri_\varphi$.
\end{theorem}

For the proof of \Cref{th:counting}, let $\pi \colon \QD^1(S) \to \T(S)$ be the projection map sending a unit area quadratic differential to its underlying Riemann surface. As in \Cref{Sec:Flow}, denote the Teichm\"uller geodesic flow by $\Phi^t \colon \QD^1(S) \to \QD^1(S)$. We  will use the same notation to denote the corresponding flow on $\MQD^1(S) = \QD^1(S) / \Mod(S)$, namely the moduli space of unit area quadratic differentials.

Let $g \in  \Mod(S)$ be a pseudo-Anosov with Teichm\"uller axis $\gamma_g$. For each $\rho, T > 0$ we define the following subset of $\QD^1(S)$:
\[
\widetilde V(\gamma_g, \rho, T) = \big\{ q\in \QD^1(S) \: : \: \pi \circ \Phi^t(q) \in N_\rho(\gamma_g), \:\forall t\in [-T,T]  \big\},
\]
where $N_\rho(\cdot)$ denotes an open $\rho$--neighborhood with respect to the Teichm\"uller metric.
Observe that $\widetilde V(\gamma_g, \rho, T)$ is nonempty and open.

Our proof of \Cref{thm:convergence} has the following corollary.

\begin{corollary} \label{cor:open}
Let $g \in \Mod(S)$ be a principal pseudo-Anosov with non-geometric veering triangulation. For every $\rho > 0$ there is a number $T = T(\rho, g) >0$, such that if $\varphi \in \Mod(S)$ is a pseudo-Anosov with an associated quadratic differential $q_\varphi$ and $\Phi^t(q_\varphi) \in \widetilde V(\gamma_g,\rho,T)$ for some $t \in \mathbb{R}$, then $\varphi$ is principal and the veering triangulation $\tri_\varphi$ is also non-geometric. 
\end{corollary}

\begin{proof}
Fix $\rho > 0$. Once $T$ is larger than the constant $D_1= D_1(\rho,g)$ given by \Cref{prop:GM}, every $q_\varphi \in \widetilde V(\gamma_g,\rho,T)$ must be principal.

Now, suppose for a contradiction that no $T > 0$ suffices for the other conclusion of the corollary. Then there is a sequence $T_n \to \infty$ and an associated sequence of principal pseudo-Anosovs $\varphi_n$, such that the invariant axis $\gamma_{\varphi_n}$ is a $\rho$--fellow traveler of $\gamma_g$ for distance $2T_n$, but the veering triangulation $\tri_{\varphi_n}$ is geometric. By \Cref{lem:qd_converge}, there is a choice of quadratic differentials $q_n$ associated to points along $\gamma_{\varphi_n}$, which converge to a quadratic differential $q$ associated to $\gamma_g$. By \Cref{lem:DrillQuadDiff}, for sufficiently large $n$, we can pass from the sequence $q_n \to q$ to a sequence $\mathring{q}_n\to \mathring{q}$ in $\e\QD_p(\Scirc)$, where $\Scirc$ is the surface obtained by puncturing $S$ at the singularities of $g$. 
%    By construction, $\mathring{q}_\varphi$ is a quadratic differential along the Teichm\"uller axis of $\phicirc \in \Mod(\Scirc)$, and similarly for the $\mathring{q}_n$.  
By \Cref{Prop:TetraShapesConverge},  the veering triangulation $\tri_{q_n}$ covering $\tri_{\varphi_n}$ is non-geometric for $n$ sufficiently large. But this contradicts our assumption about $\varphi_n$. Thus some $T >0$ must suffice.
\end{proof}

We finish the proof of \Cref{th:counting} (hence also \Cref{thm:counting}) with the following argument, whose idea was suggested by I. Gekhtman.

\begin{proof}[Proof of \Cref{th:counting}]
We  identify a conjugacy class of pseudo-Anosovs on $S$ with the corresponding closed orbit of the Teichm\"uller flow  $\Phi^t \colon \MQD^1(S) \to \MQD^1(S)$. Following this identification, it makes sense to refer to typical closed orbits of the Teichm\"uller flow. 

Let $g \in \Mod(S)$ be a principal pseudo-Anosov whose associated veering triangulation is not geometric. (Such a mapping class exists by \Cref{Prop:NongeomExists}, which will be proved in \Cref{sec:poison,sec:norm}.) Fix $\rho =1$, and let $T=T(1,g)>0$
be given by \Cref{cor:open}.
Finally, let $V$  be the image of $\widetilde V(\gamma_g, \rho, T)$ in $\MQD^1(S)$. This set is also open and nonempty. We will show that a typical closed orbit of $\Phi^t \colon \MQD^1(S) \to \MQD^1(S)$ meets $V$.

Set $h = 2 \xi(S) = \mathrm{dim} \, \T(S)$.
For each closed orbit $\gamma$ of $\Phi^t$, let $\delta(\gamma)$ be the $\Phi^t$--invariant Lebesgue measure $\QD^1(S)$, supported on $\gamma$, of total mass $\ell(\gamma)$. Thus, for a Lebesgue measurable set $E \subset \QD^1(S)$, we have $\delta(\gamma)(E) = \ell(\gamma \cap E)$.
%Let $\mathcal{G}(L,\sigma)$ be the set of closed orbits for $\Phi^t$ whose length is in the interval $(L-\sigma, L]$.
%Then combining recent work of Hamenst\"adt \cite[Corollary 5.4]{Ham18} (building on her earlier work \cite{hamenstadt2013bowen}) with the orbit counting theorem of Eskin--Mirzakhani \cite{eskin2011counting} shows that the measures
%\[
%\nu_{(L,\sigma)} = \frac{(1-e^{-h\sigma})^{-1}}{L |\mathcal{G}(L)|} \sum_{\gamma \in \mathcal{G}(L,\sigma)}\delta(\gamma)
%\]
Hamenst\"adt \cite{hamenstadt2013bowen} proved that as $L \to \infty$, the measures
\[
he^{-hL} \sum_{\gamma \in \mathcal{G}(L)}\delta(\gamma)
\]
converge weakly to the Masur--Veech measure $\lambda$ on $\MQD^1(S)$. (See also \cite[Theorem 5.1]{Ham18}.) The probability measure $\lambda$ is in the Lebesgue measure class, has full support, and is ergodic for the Teichm\"uller flow $\Phi^t$ \cite{veech1986teichmuller, masur1982interval}.

Next, we recall the geodesic counting theorem of Eskin and Mirzakhani \cite{eskin2011counting}, which states that as $L \to \infty$,
\[
|\mathcal{G}(L)| \cdot hLe^{-hL} \to 1.
\]
Combining the above displayed equations, we have that the measures
\[
\nu_L = \frac{1}{L |\mathcal{G}(L)|} \sum_{\gamma \in \mathcal{G}(L)}\delta(\gamma)
\]
converge weakly to $\lambda$. This convergence is also noted in the proof of \cite[Proposition 5.1]{baik2016smallest}.

%The main theorem of \cite{hamenstadt2013bowen} (see also \cite[Proposition 5.1]{baik2016smallest}) gives that 
%\[
%\liminf_{L \to \infty} \frac{1}{L |\mathcal{G}(L)|} \sum_{[\varphi] \in \mathcal{G}(L)} \ell(\gamma_\varphi \cap V) = \lambda (V),
%\] 
%where $\lambda$ is the Masur--Veech measure on $\MQD^1(S)$, which is in the Lebesgue measure class and has full support.
%%and $\ell(\gamma \cap V)$ denotes the total length of time that $\gamma$ spends in $V$.
%In particular, $\lambda(V) >0$. Hence, it follows immediately that the typical closed orbit in $\MQD^1(S)$ meets the set $V$. 

Let $A \subset \MQD^1(S)$ be the union of all closed orbits of the flow $\Phi^t$ that are disjoint from $V$. Then the closure $\overline{A}$ is flow-invariant and disjoint from $V$ because $V$ is open. By the ergodicity of $\lambda$, we 
 must have $\lambda(\overline{A}) =0$.
Since $\overline{A}$ is closed, weak convergence $\nu_L \to \lambda$ and the Portmanteau Theorem imply that $\limsup \nu_L(\overline A) \le \lambda(\overline{A}) = 0.$ 

Now, we wish to show that a typical closed orbit is not contained in $A$. To that end, fix $\epsilon >0$ and choose $\sigma >0$ so that $e^{-h\sigma} < \epsilon$.  In the following computation for fixed $L > \sigma$, the symbol $\gamma$ denotes both a pseudo-Anosov conjugacy class and the corresponding closed orbit in $\MQD^1(S)$. We have
\begin{align*}
|\{\gamma \in \mathcal{G}(L) : \gamma \cap V = \emptyset\}| 
&=  
|\{\gamma : \gamma \subset  A, \: \ell(\gamma) \leq L \}| \\
&= |\{\gamma  : \gamma \subset  A, \, \ell(\gamma) < L - \sigma  \}| + |\{\gamma  : \gamma \subset  A,  \:  L - \sigma \leq \ell(\gamma) \leq L \}| \\
& = |\{\gamma  : \gamma \subset  A, \, \ell(\gamma) < L - \sigma  \}| + \sum_{\begin{smallmatrix}  \gamma \subset A \\ L - \sigma \leq \ell(\gamma) \leq L \end{smallmatrix} } 1 \\
& \leq  |\mathcal{G}(L-\sigma)| \quad + \quad \sum_{\begin{smallmatrix}  \gamma \subset A \\ L - \sigma \leq \ell(\gamma) \leq L \end{smallmatrix} } \frac{\ell(\gamma)}{L - \sigma} \\
& \leq  |\mathcal{G}(L-\sigma)| \quad + \quad \sum_{  L - \sigma \leq \ell(\gamma) \leq L } \frac{\ell(\gamma \cap \overline A)}{L - \sigma} \\
& \leq  |\mathcal{G}(L-\sigma)|  \quad +  \:  \frac{1}{L- \sigma} \sum_{ \mathcal{G}(L) } \ell(\gamma \cap \overline A) \\
 &= |\mathcal{G}(L-\sigma)| \:+ \: \frac{L |\mathcal{G}(L)| }{L-\sigma} \,  \nu_L(\overline A) .
\end{align*} 
Dividing by $|\mathcal{G}(L)|$ and taking limits as $L \to \infty$, we obtain
\begin{align*}
\limsup_{L \to \infty}  \frac{|\{\gamma \in \mathcal{G}(L) : \gamma \cap V = \emptyset\}|}{|\mathcal{G}(L)|} 
&\leq  \limsup_{L \to \infty}  \left(  \frac{|\mathcal{G}(L-\sigma)|}{|\mathcal{G}(L)|  } \: + \: \frac{L}{L-\sigma} \, \nu_L(\overline A) \right) \\
%&\le \limsup \nu_L(\overline A) + \limsup \frac{ e^{h(L-\sigma)}/h(L-\sigma)}{e^{hL}/hL}
&=  e^{-h\sigma} \:+\: \limsup_{L \to \infty}  \,  \nu_L(\overline A) \\
&<  \epsilon + 0.
\end{align*}
Since $\epsilon > 0$ was arbitrary, this shows that a typical closed orbit meets $V$. 

By the definition of $\widetilde V$ and $V$, this means that a typical pseudo-Anosov conjugacy class has a representative $\varphi$ with associated quadratic differential $q_\varphi$ such that $\Phi^t(q_\varphi) \in  \widetilde V(\gamma_g, \rho, T)$ for some $t \in \mathbb{R}$. By \Cref{cor:open}, this implies that a typical pseudo-Anosov conjugacy class $[\varphi]$ is principal and produces a non-geometric veering triangulation $\tri_\varphi$. 
\end{proof}

%% !TEX root =rndm_veering.tex

\section{A few non-geometric triangulations, via computer}
\label{sec:poison}

% need: definition of Thurston norm and basic facts about it (ball is centrally sym polyhedron with finitely many faces, norm linear iff on face, euler char of hom class is invariant mod 2, ); invariance of pA flow on a fibered face; 

Our remaining goal in this paper is to prove \Cref{Prop:NongeomExists}: for every surface $S$ of complexity $\xi(S) \ge 2$, there exists a principal pseudo-Anosov map $\varphi$ so that the veering triangulation of $M_{\phicirc}$ is non-geometric. 
%    For most surfaces, this mapping class $\varphi$ can be taken to be principal (\Cref{prop:candidates}). 
We will prove this result in two stages.

\begin{enumerate}
\item In this section, we use rigorous computer assistance to find a finite collection of non-geometric triangulations. See \Cref{prop:seeds}.
%    \item In \Cref{sec:covers}, we use finite covers to extend the finite list obtained in \Cref{prop:seeds} to an infinite list. This will prove \Cref{prop:candidates} for  closed surfaces and for planar surfaces.
\item In \Cref{sec:norm}, we use Thurston norm methods to show that (finite covers of) the finitely many mapping tori described in \Cref{prop:seeds} account for all fibers of complexity at least $2$. This will prove  \Cref{Prop:NongeomExists}.
\end{enumerate}

\begin{figure}[h]
 	\centering
   	\includegraphics[scale=.37]{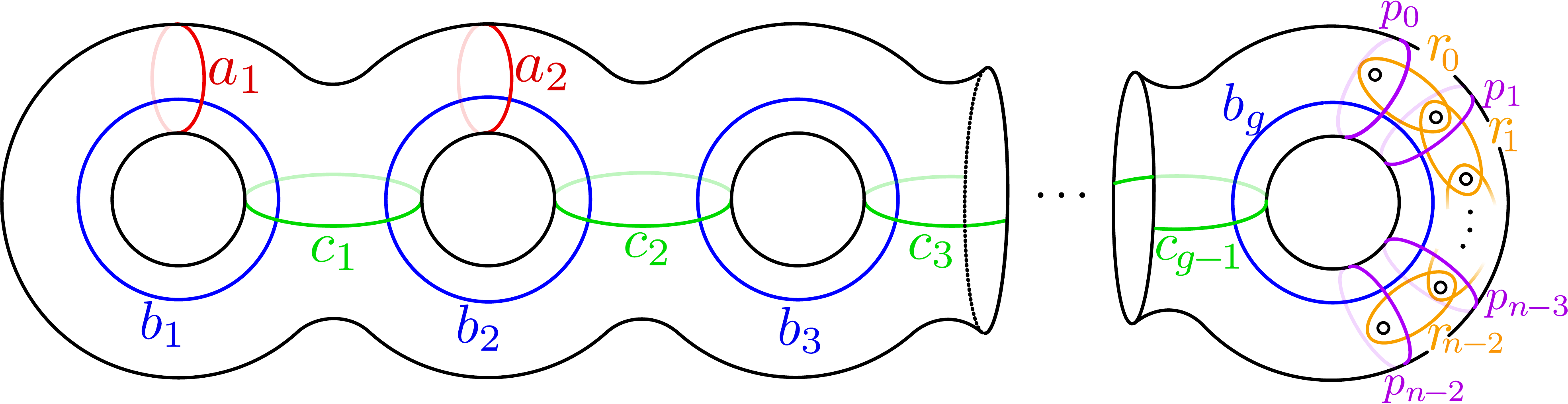}
   	\caption{The mapping class group $\Mod(\Sigma_{g,n})$, for genus $g>0$, is generated by half-twists about $r_0, \ldots, r_{n-2}$ and full twists about the following curves: $a_1,a_2$, $b_1, \ldots, b_g$, $c_1, \ldots, c_{g-1}$, and $p_0, \ldots, p_{n-2}$. When $n \leq 1$, there are no curves $p_i$ or $r_i$. See \cite[Figure 4.10]{FM12} and the subsequent discussion.}
 	\label{fig:surf_gens}
\end{figure}

\begin{figure}
   	\centering
   	\begin{subfigure}{0.4\textwidth}
   		\centering
      	\includegraphics[scale=.23]{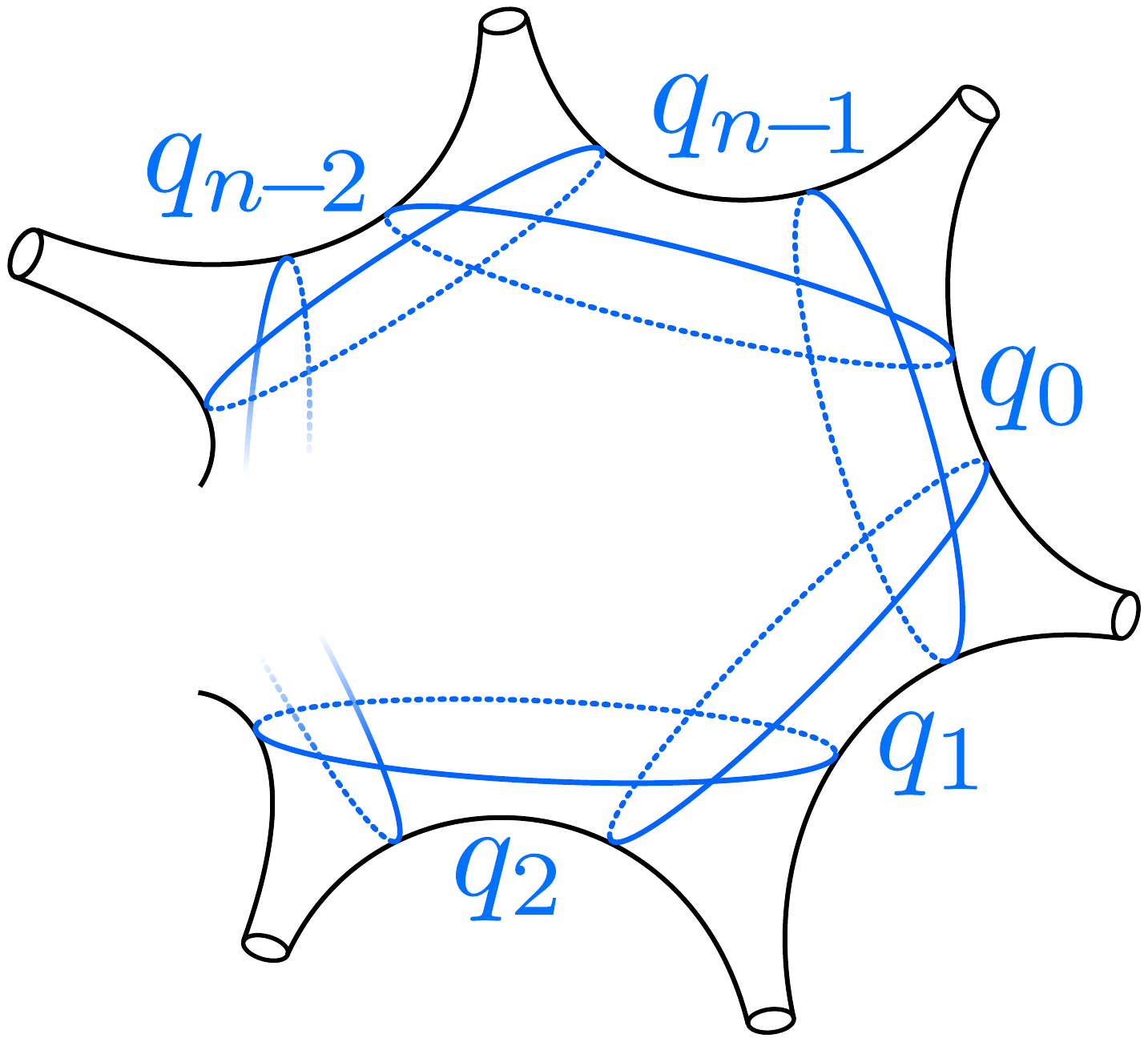}
      	\caption{}
    	\label{fig:S0n_gens}
  	\end{subfigure}
   	\begin{subfigure}{0.58\textwidth}
   		\centering
     	\includegraphics[scale=.3]{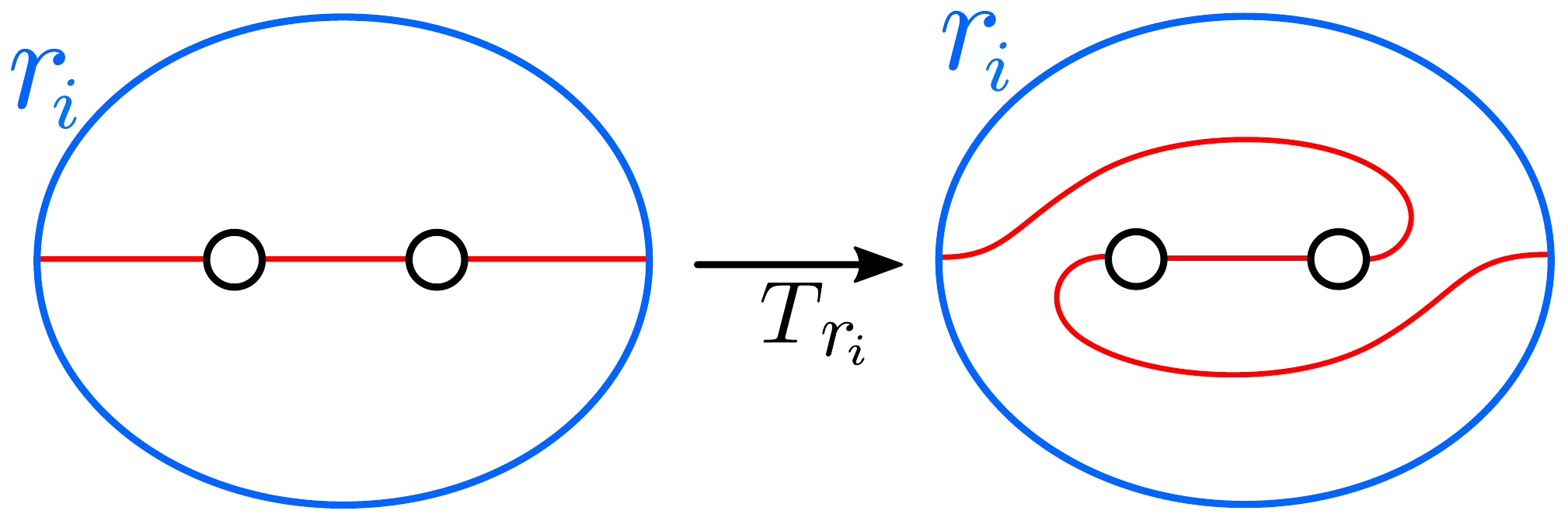}
      	\caption{}
     	\label{fig:half_twist}
  	\end{subfigure}       
 	\caption{Left: The mapping class group of the $n$--punctured sphere $\Sigma_{0,n}$, for $n\geq 4$, is generated by half-twists about curves $r_0, \ldots, r_{n-1}$. Right: A half twist about $r_i$ fixes $r_i$ and transposes the punctures in the twice-punctured disc bounded by $r_i$.}
  	\label{fig:S0n}
\end{figure}
	
    We begin by describing how  the computer programs \texttt{flipper} \cite{Bel13}, \texttt{SnapPy} \cite{CDGW09}, and \texttt{Regina} \cite{BBP99} are used to show that certain triangulations are non-geometric. Given a mapping class $\varphi$  on a hyperbolic surface $S$, described by a composition of left Dehn twists and/or half twists in the generators shown in \Cref{fig:surf_gens,fig:S0n}, we first use \texttt{flipper}  to certify that $\varphi$ is pseudo-Anosov. 
%        For a pseudo-Anosov mapping class, \texttt{flipper} computes the invariant measured stable lamination associated to $\varphi$, encoded as
The certificate consists of an invariant train track carrying the stable lamination of $\varphi$,  an integer-valued transition matrix, and certified floating-point intervals for the weights of the branches.
Every region in the complement of the train track contains exactly one singularity of $\varphi$, enabling \texttt{flipper} to certify whether $\varphi$ is principal.
Finally, \texttt{flipper} punctures the surface at the singularities of $\varphi$ and computes the veering triangulation $\tri_\varphi$ of the mapping torus $M_{\phicirc}$. All of the above  \texttt{flipper} computations are rigorous.

The program  \texttt{SnapPy} can find an approximate solution to the gluing equations for the veering triangulation $\tri_\varphi$ (see discussion below). This approximate solution is a good heuristic indication that $\tri_\varphi$ is not geometric. To rigorously certify that $\tri_\varphi$ is not geometric, we 
 follow the method of Hodgson, Issa, and Segerman \cite{HIS16}, 
relying on \Cref{Thm:Francaviglia,Thm:NeumannYang} below. 
%    Describing their method requires some background.

Let $\tri$ be an ideal triangulation of a hyperbolic 3-manifold $M$ with $k$ tetrahedra, and let $\vec{z}=(z_1,\dots,z_k)$ be a vector of complex numbers in bijection with the tetrahedra in $\tri$. Every $z_i$ has an associated \define{algebraic volume} $\mathrm{Vol}(z_i) \in \RR$, computed via the dilogarithm function \cite{NY99}. In particular, $\mathrm{Vol}(z_i)$ has the same sign as $\mathrm{Im}(z_i)$.
We define the algebraic volume $\mathrm{Vol}(\vec{z}) = \sum_{i=1}^k \mathrm{Vol}(z_i)$.
%    \end{definition}

The \define{gluing and completeness equations} for  $\tri$ are a system of polynomial equations in  $z_1, \ldots, z_k$ \cite[Chapter 4]{Thu78}. Any solution $\vec{z}=(z_1,\dots,z_k)$ to this system of equations defines a representation $\rho \from \pi_1 M \to \mathrm{PSL}(2,\CC)$, unique up to conjugacy, in which peripheral elements map to parabolics. In the resulting structure on $M$, the shape parameter of the tetrahedron $\tet_i$ is exactly $z_i$.
If $\vec{z}$ is geometric, hence $\mathrm{Im}(z_i) > 0$ for each $i$, then $\rho$ is the discrete, faithful representation that gives the complete hyperbolic metric on $M$. In general, the number of solutions to the gluing and completeness equations that yield the discrete, faithful representation $\rho$ is either $0$ or $1$, with no solutions when $\tri$ is degenerate in the sense described below. 
%    By Francaviglia's theorem \cite{Fra04}, the discrete, faithful representation $\rho$ corresponds to the solution with the largest algebraic volume. 
The following statement combines \cite[Theorem 5.4.1 and Remark 4.1.20]{Fra04}.

\begin{theorem}[Francaviglia]\label{Thm:Francaviglia}
Let $\tri$ be an ideal triangulation of a hyperbolic 3-manifold $M$ whose volume is $\mathrm{Vol}(M)$.
Then every solution $\vec{z}$ to the gluing and completeness equations for $\tri$ satisfies $\mathrm{Vol}(\vec{z}) \leq \mathrm{Vol}(M)$, 
with equality if and only if $\vec{z}$ is the unique solution yielding the complete hyperbolic metric on $M$.
%    
%    Let $\tri$ be an ideal triangulation of a hyperbolic 3-manifold $M$. Then there exists at most one solution $\vec{z}$ to the gluing and completeness equations for $\tri$ such that $\mathrm{Vol}(\vec{z})=\mathrm{Vol}(M)$, where $\mathrm{Vol}(M)$ is the hyperbolic volume of $M$.
\end{theorem}

\begin{figure}
 	\centering
   	\includegraphics[scale=.45]{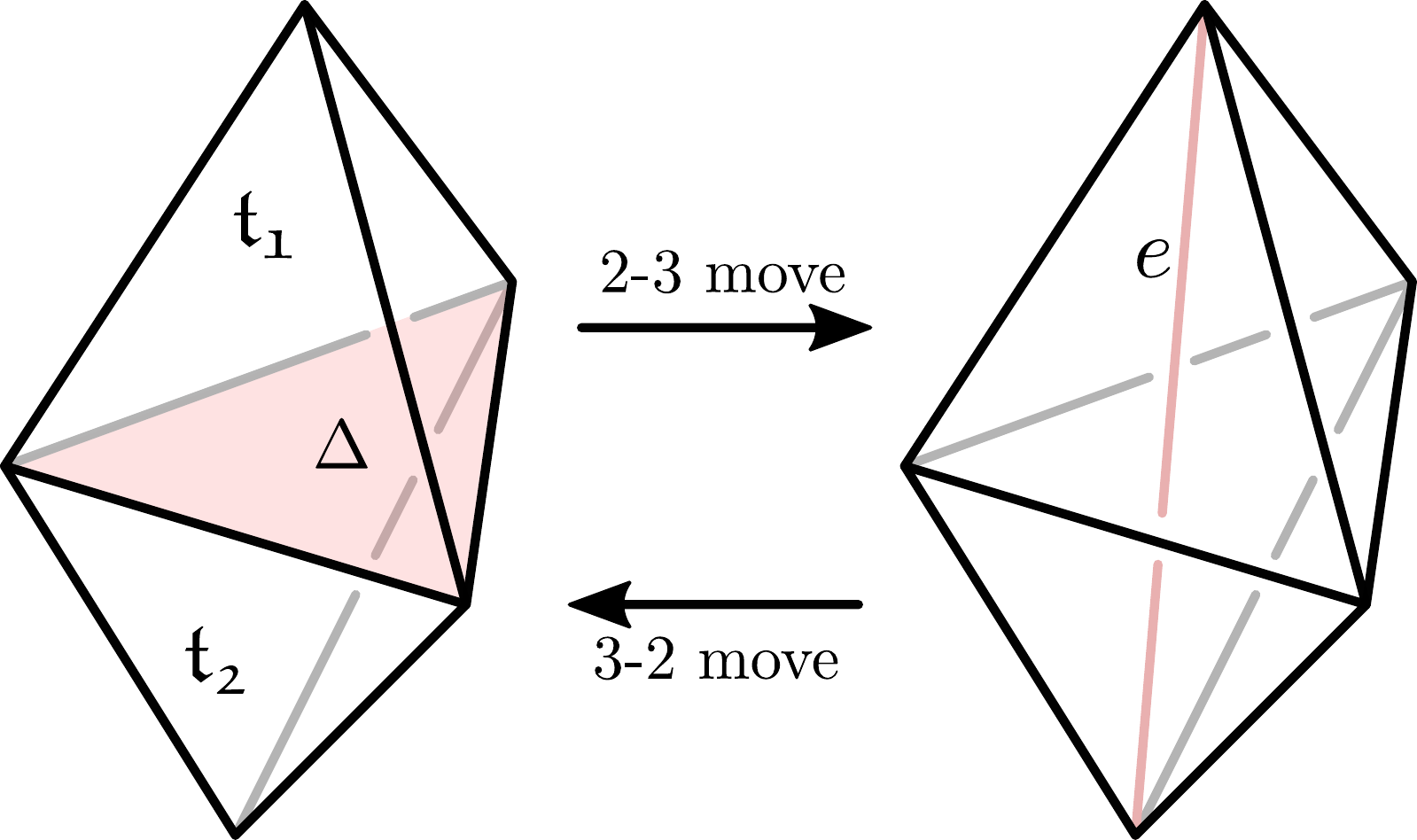}
   	\caption{In a Pachner 2-3 move, two tetrahedra meeting at a face $\tetface$ are replaced by three tetrahedra meeting at an edge dual to $\tetface$. A Pachner 3-2 move is the reverse of a 2-3 move.}
 	\label{fig:Pachner}
\end{figure}

%        \begin{definition}
Now, suppose the triangulation $\tri$ changes as follows.
Let $\tetface$ be a 2-simplex in $\tri$ which is the face of distinct tetrahedra $\tet_1$ and $\tet_2$. We can obtain a new triangulation $\tri'$ by replacing $\tetface$ by a dual edge $e$, and adding three faces each meeting $e$ and a distinct vertex of $\tetface$. Then  $\tri \to \tri'$ is called a \define{Pachner 2-3 move}, while the reverse operation  $\tri' \to \tri$ is called a \define{Pachner 3-2 move}. See \Cref{fig:Pachner}.
	
%    \end{definition}

%    \begin{definition}
An ideal tetrahedron $\tet$ in a hyperbolic $3$--manifold $M$ is called \define{degenerate} if some edge of $\tet$ is homotopic into a cusp of $M$. If $\tet$ is non-degenerate, its lift to $\widetilde M = \HH^3$ is homotopic to a straight tetrahedron, hence $\tet$ can be assigned a shape parameter $z_\tet \in \CC \setminus \{0,1\}$. A triangulation $\tri$ is called non-degenerate if all of its tetrahedra are non-degenerate.

If a triangulation $\tri$ is equipped with a solution $\vec{z}$ of the gluing and completeness equations, and $\tri'$ is obtained from $\tri$ via a Pachner move, then from $\vec{z}$ we get an associated solution $\vec{z}\,'$ to the gluing and completeness equations for $\tri'$. Furthermore, if $\tri$ and $\tri'$ are both non-degenerate, we get algebraic volume information as well. The following result, due to  Neumann and Yang \cite[Proposition 10.1]{NY99}, is a consequence of the ``five-term relation," an identity of the dilogarithm function.

\begin{theorem}[Neumann--Yang]\label{Thm:NeumannYang}
Let $M$ be a hyperbolic 3-manifold with ideal triangulation $\tri$, and let $\vec{z}$ be a solution to the gluing and completeness equations for $\tri$. Let $\tri'$ be an ideal triangulation obtained from $\tri$ by a Pachner move, with no degenerate tetrahedra. If  $\vec{z}\,'$ is the solution to the gluing equations for  $\tri'$ corresponding to the solution $\vec{z}$ for $\tri$, then $\mathrm{Vol}(\vec{z})=\mathrm{Vol}(\vec{z}\,')$.
\end{theorem}

Combining  \Cref{Thm:Francaviglia,Thm:NeumannYang}, we get the following corollary:

\begin{corollary}\label{cor:certify-non-geometric}
	Let $M$ be a hyperbolic 3-manifold, and let $\tri_1,\dots,\tri_n$ be a sequence of non-degenerate ideal triangulations, with consecutive triangulations related by a Pachner move. Suppose $\tri_n$ has a geometric solution $\vec{z}_n$ to the gluing and completeness equations. For $i \in \{0, \ldots, n-1\}$, let $\vec{z}_i$ be the solution for $\tri_i$ obtained from $\vec{z}_{i+1}$ via the corresponding Pachner move. If the solution $\vec{z}_1$ is non-geometric, then $\tri_1$ is a non-geometric triangulation.
%    
%    	such that $\tri_{i}$ is obtained from $\tri_{i-1}$ via a Pachner move. Suppose that $\tri_1$ has a non-geometric solution $\vec{z}_1$ to the gluing and completeness equations, and let $\vec{z}_i$ be the solution for $\tri_i$ obtained from $\vec{z}_{i-1}$ via the corresponding Pachner move. If $\vec{z}_n$ is a geometric solution (i.e., $\tri_n$ is geometric), then $\tri_1$ is non-geometric.
\end{corollary}

To establish that a given veering triangulation $\tri$ is non-geometric, we will first find a Pachner path from $\tri=\tri_1$ to a triangulation $\tri_n$ that has a geometric solution $\vec{z}_n$. For each tetrahedron in $\tri_n$, \texttt{SnapPy} gives a rigorously verified rectangle in $\CC$ containing its shape parameter, using an algorithm derived from \texttt{HIKMOT} \cite{HIKMOT16}. In other words, we get a box $K_n\subset\CC^{|\tri_n|}$ which is guaranteed to contain a geometric solution to the gluing and completeness equations. We then follow the path backwards, from $\tri_n$ to $\tri_1$, obtaining for each intermediate triangulation $\tri_i$ a box $K_i \subset \CC^{|\tri_i|}$ containing the corresponding solution $\vec{z}_i$. This certifies that $\tri_i$ is non-degenerate. When $\tri_1$ is reached, we check that the corresponding box $K_1$ has at least one coordinate (i.e., at least one shape parameter) whose imaginary part is negative. Since $K_1$ is guaranteed to contain the solution $\vec{z}_1$, it follows that this solution is non-geometric, so by \Cref{cor:certify-non-geometric} the triangulation $\tri=\tri_1$ is non-geometric.

In practice, a path to a geometric triangulation is found by  randomly choosing Pachner moves based at edges and faces of negatively oriented tetrahedra. In our examples, the paths we find range in length from 4 to 18. We use \texttt{Regina} to help keep track of the labeling of 
edges on the reverse path from $\tau_n$ to $\tau_1$ (in particular, this requires the ability to construct explicit isomorphisms of triangulations).

The following proposition gives a list of mapping classes which have been rigorously verified to be principal pseudo-Anosovs with non-geometric veering triangulations. These mapping classes are described as words in the left Dehn twists and half-twists about the curves shown in \Cref{fig:surf_gens,fig:S0n}. 

\begin{proposition} \label{prop:seeds}
For each of the following surfaces $S_i$, the mapping class $\varphi_i$ described below is a principal pseudo-Anosov. Furthermore, the veering triangulation of $M_{\phicirc_i}$ is non-geometric.
\begin{itemize}
\item $S_1 = \Sigma_{2,0}$ and
$\varphi_1 = T_{a_1}T_{c_0}^{-2}T_{b_0}T_{a_0}^{-1}T_{b_1}^{-1}T_{b_0}^{-1}T_{a_0}T_{b_0}.$

\smallskip
\item
$S_2 = \Sigma_{2,1}$ and $\varphi_2 = T_{c_0}T_{b_0}^{-2}T_{c_0}T_{a_0}T_{b_0}T_{a_1}^{-1}T_{a_0}T_{b_1}.$

\smallskip
\item
$S_3 = \Sigma_{1,2}$ and $\varphi_3 = T_{a_0}^{-2}T_{r_0}^{-1}T_{b_0}^{-1}T_{p_0}^{-1}T_{a_0}T_{b_0}T_{p_0}^{-3}$.

%    \item
%    $S_4 =\Sigma_{0,7}$ and  $\varphi_4 = T_{r_4}^2T_{r_3}T_{r_2}^{-1}T_{r_0}^{-1}T_{r_3}T_{r_1}^{-1}T_{r_2}T_{r_0}T_{r_1}^{-1}.$

\smallskip
\item $S_5 = \Sigma_{0,5}$ and
$\varphi_5 =T_{r_3}^{-2}T_{r_0}^{-3}T_{r_1}^2T_{r_0}^{-1}T_{r_2}^{-1}T_{r_4}$.

\smallskip
\item $S_6 = \Sigma_{0,6}$ and 
$ \varphi_6 = T_{r_5}^{-1}T_{r_1}T_{r_2}^{-1}T_{r_3}T_{r_4}^{-1}T_{r_0}^{-1}T_{r_1}T_{r_2}T_{r_1}^2$.

\smallskip
\item $S_7 = \Sigma_{0,7}$ and
$ \varphi_7 = T_{r_5}^2T_{r_4}T_{r_3}^{-1}T_{r_1}^{-1}T_{r_4}T_{r_2}^{-1}T_{r_3}T_{r_1}T_{r_2}^{-1}$.
%% before cyclic permutation, this was:
%%$ \varphi_7 = T_{r_4}^2T_{r_3}T_{r_2}^{-1}T_{r_0}^{-1}T_{r_3}T_{r_1}^{-1}T_{r_2}T_{r_0}T_{r_1}^{-1}$.
% We could also use: $T_{r_0}^{-1}T_{r_1}^{-1}T_{r_5}T_{r_6}^{-1}T_{r_0}^{-1}T_{r_1}T_{r_2}^{-1}$.

\end{itemize}
\end{proposition}

\begin{proof}
For each $S_i$, \texttt{flipper} certifies that $\varphi_i$ is a principal pseudo-Anosov. Then, \texttt{SnapPy} combined with \Cref{cor:certify-non-geometric} certifies that the veering triangulation of $M_{\mathring{\varphi}_i}$ is non-geometric. A detailed certificate of non-geometricity, including a path from the veering triangulation to a geometric triangulation, appears in the ancillary files \cite{FTW:Auxiliary}.

For working with $\varphi_1$ in \texttt{flipper}, we actually consider this mapping class on $\Sigma_{2,1}$, with the puncture located as in \Cref{fig:surf_gens}. This is because the data structure used by \texttt{flipper} requires at least one puncture. The program computes that $\varphi_1$ has a trivalent singularity at the unique puncture of $\Sigma_{2,1}$. Thus we may fill the puncture, recovering a principal pseudo-Anosov on the 
 closed surface $\Sigma_{2,0}$. For the other $\varphi_i$, we work with the surface $S_i$ exactly as given.
\end{proof}

\begin{comment}
\begin{remark}
The same proof strategy, using \Cref{cor:certify-non-geometric}, tends to work on other surfaces of moderate complexity. However, the brute-force search for candidate mapping classes becomes a lot harder (DISCUSS RUNTIME), and the certification process also becomes a lot slower (DISCUSS RUNTIME). 

The largest surface on which we have found a principal pseudo-Anosov with non-geometric veering triangulation using this method is $\Sigma_{0,22}$. The corresponding mapping class $\varphi_{22}$ takes five lines of text to write down. Finding this example took XXXXX computer hours, and certifying it as correct via a Pachner path took XXXXXXX computer time.
\end{remark}
\end{comment}

%    \input{covers.tex}

%% !TEX root =rndm_veering.tex

\section{Non-geometric triangulations via the Thurston norm}\label{sec:norm}

In this section, we use Thurston norm theory to show the existence of a mapping class with non-geometric veering triangulation for every hyperbolic surface of complexity at least $2$. The eventual result will be that (finite covers of) the mapping tori of the classes $\varphi_1, \ldots, \varphi_7$ from \Cref{prop:seeds} contain fibers homeomorphic to every surface $S$ with $\xi(S) \geq 2$. This will imply  \Cref{Prop:NongeomExists} from the Introduction.

We begin by reviewing some classical results about the Thurston norm, and then proceed to find the desired fiber surfaces in the mapping tori of  $\varphi_1, \ldots, \varphi_7$.

\subsection{The Thurston norm}

Let $M$ be a compact orientable 3--manifold with $\bdy M$ a possibly empty union of tori, such that the interior of $M$ is hyperbolic. We will pass freely between $M$ and its interior. Thurston \cite{Thu86} showed that there is a norm ${\|\cdot\| \from H_2(M,\bdy M ; \RR)\to \RR}$ on second homology, defined on integral classes by the property
\[
\|x\| = \min  \left\{-\chi(S) \: :  \: S \text{ is an embedded surface without $S^2$ components representing } x \right\}.
\]
%    where $\chi_-(S)=\max\{0, -\chi(S)\}$. 
He proved that this norm, now called the \define{Thurston norm}, has the following properties:

\begin{enumerate}
	\item The unit ball $B = \{x : \|x \| \leq 1 \}$  is a centrally symmetric polyhedron.
	\item If $M$ is a fibered 3--manifold with fiber $F$, then the class $[F]\in H_2(M,\bdy M)$ lies on a ray from the origin that passes through an open top-dimensional face $\face \subset  \bdy B$. In this case, $\face$ is called a \define{fibered face}, and the open cone $\RR_+\face$ is called a \define{fibered cone}.
\item If  $x \in \RR_+\face$ is a primitive integral homology class lying in a fibered cone, then $x$ is represented by a fiber surface $S$. Furthermore,  $\|x\|= - \chi(S)$. In particular, if $\dim(H_2(M,\bdy M)) \ge 2$, then  $\RR_+ \face$ contains infinitely many fiber classes.

\end{enumerate}

When $M$ is fibered with fiber $F$, the pseudo-Anosov monodromy $\varphi \colon F\to F$ of $M$ induces a \define{suspension flow} $\eta$ on $M$. We also have $\eta$--invariant $2$--dimensional foliations $\Lambda^{\pm}$, which are suspensions of the invariant foliations $\fol^{\pm}$ associated to $\varphi$.  Let 
$\RR_+ \face$ be the fibered cone containing $F$.
%    An immediate consequence of (3) is that if $\dim(H_2(M,\bdy M)) \ge 2$, then any fibered cone $\RR_+ \face$ must contain infinitely many fibered classes.
%%each fibered face $\face$, being at least 1-dimensional, must contain infinitely many fibered classes. 
Then $\face$  determines $\Lambda^{\pm}$ and the flow $\eta$ (up to isotopy and reparametrization), independent of the fiber $F$.
Moreover, for every fiber $S$ in $\RR_+\face$, the foliations $\Lambda^+$ and $\Lambda^-$ are transverse to $S \subset M$, and the intersections $\Lambda^{\pm}\cap S$ are isotopic to the stable and unstable foliations $\fol_S^{\pm}$ associated to the monodromy of $S$. 
See Fried \cite{Fri82} and McMullen \cite{McM00} for more details.
%    These facts were proven by Fried \cite{Fri82}; see also McMullen \cite{McM00}.

A \define{slope} on a torus $T$ is an isotopy class of simple closed curves, or equivalently an (unsigned) primitive homology class in $H_1(T; \ZZ)$.
In a fibered $3$--manifold $M$, with boundary tori $T_1, \ldots, T_m$, any fibration of $M$ determines two slopes on each torus $T_i$.
First, a fiber $F$  must meet every  $T_i$ in a union of disjoint, consistently oriented simple closed curves. The isotopy class of these simple closed curves is called the \define{boundary slope} of $F$ on $T_i$. 
%    Equivalently, we may consider the boundary homomorphism $\bdy \from H_2(M, \bdy M) \to H_1(\bdy M) = \bigoplus_{i=1}^m H_1(T_i)$. Then the boundary slope of $F$ on $T_i$ is the primitive homology class corresponding to the $i$-th coordinate of $\bdy[F]$. 
Second, the orbit under the flow $\eta$ of a singular leaf of $\fol^\pm$ is a ($2$--dimensional) singular leaf of $\Lambda^\pm$. Every singular leaf traces out a simple closed curve on some $T_i$, whose slope is called the \define{degeneracy slope} on $T_i$. We emphasize that the degeneracy slope is entirely determined by $\Lambda^\pm$, hence by the fibered cone containing $F$.

%%%% Noted by Sam: the flow $\eta$ is really only defined on the interior of $M$. However, we want it to define a slope on a component of $\bdy M$. The solution is to situate the torus $T_i$ as a cusp torus in the interior of the hyperbolic manifold, and consider the intersection between $T_i$ and a singular leaf. Then everything is well-defined up to isotopy.

%     
%     Given a fibered manifold $M$ with fibered cone $\RR_+\face$, each homology class  $x \in \RR_+\face \subset H_2(M,\bdy M ; \RR)$ determines a \define{boundary slope} for each boundary component of $M$ that it meets. In particular, if $F$ intersects a torus boundary component $T$ of $M$, then this intersection defines an isotopy class $\alpha$ of simple closed curves in $T$, and the boundary slope of $F$ at $T$ is the slope of $\alpha$ with respect to some basis for $H_1(T)$. We will only be concerned with whether two slopes are the same on $T$, so we will not need to refer to a basis.

\begin{lemma}
\label{Lem:DegeneracySlope}
Let $M = M_\varphi$ be the mapping torus of a principal pseudo-Anosov $\varphi \from F \to F$. Then, on every component of $\bdy M$, the boundary slope of $F$ intersects the degeneracy slope once.
If $S \subset M$ is another fiber surface in the same fibered cone as $F$, then the monodromy of $S$ is principal if and only if the boundary slope of $S$ intersects the degeneracy slope once.
\end{lemma}

\begin{proof}
Let $\fol^+_F$ be the stable foliation of $\varphi$ on $F$. Since $\varphi$ is principal, $\fol^+_F$ has $3$--prong singularities at interior points of $F$. Thus $\Lambda^+$ also has $3$--prong singularities at interior points of $M$. In addition, every puncture of $F$ meets exactly one singular leaf of $\fol^+_F$. Thus, on every cusp torus $T_i \subset \bdy M$, a loop about the puncture  of $F$ intersects the degeneracy slope in exactly one point.

Let $S$ be another fiber surface in the same fibered cone as $F$. As mentioned above, the stable foliation $\fol^+_S$ is isotopic to $\Lambda^+ \cap S$, hence has $3$--prong singularities at interior points of $S$. Meanwhile, the singular prongs of $\fol^+_S$ at a given puncture of $S$ are in bijective correspondence with points of  $\Lambda^+ \cap \gamma_i$, where $\gamma_i$ denotes a loop about the puncture. Thus the monodromy of $S$ is principal if and only if every $\gamma_i$ intersects the degeneracy slope once.
\end{proof}

For a pseudo-Anosov $\varphi \colon S \to S$, recall that  $\phicirc \colon \Scirc \to \Scirc$ denotes the restricted map obtained by puncturing $S$ at the singularities of $\varphi$.
The mapping torus $\mathring{M} = M_{\phicirc}$ can  be constructed by drilling $M_\varphi$ along the singular flow-lines of $\Lambda^\pm$, i.e. the orbits of the singularities of $\fol^\pm$ under the flow $\eta$; in particular, $\mathring{M}$ depends only on the face $\face$ containing $S$.
Furthermore,  the fibered face of $M_{\phicirc}$ whose cone contains $\Scirc$ depends only on $\face$.
%    , and every other fiber in the cone over $\facecirc$ is fully punctured, in the sense that the singularities of the associated foliations occur at punctures. See Minsky and Taylor \cite{MT17} for details.

% In light of the discussion in the previous paragraph, we see that every fiber in the cone over $\facecirc$ is fully punctured and so we call $\facecirc$ a \define{fully punctured fibered face}. 
 
 The following lemma is a special case of \cite[Proposition 2.7]{MT17}.
 
 \begin{lemma}[Agol]
\label{Lem:FaceInvariance}
Let $M$ be a fibered hyperbolic 3-manifold with fibered face $\face$. Then any two fibers $F_1, F_2 \in \RR_+ \face$ 
%    $\Fcirc_1, \Fcirc_2$ of $\facecirc$ 
produce the same veering triangulation of $\mathring{M}$, up to isotopy. 
\end{lemma}

%    
%    Note that although $\varphi$ may be principal on $F$, $\phicirc$ will not be principal on $\Fcirc$ because it has $3$--prong singularities at the places where $F$ was punctured. 

We close this background section with two easy but useful observations that date back to Thurston \cite{Thu86}.

\begin{fact}\label{Fact:Linearity}
For $x,y \in H_2(M,\partial M;\RR)$, the equality $\|x+y\|=\|x\|+\|y\|$ holds if and only if $x$ and $y$ are in the same cone over a face of the unit Thurston norm ball.
\end{fact}

This follows by the definition of a norm, combined with the property that the unit ball $B$ is a polyhedron.

\begin{fact}\label{Fact:Mod2}
If $S$ represents $x\in H_2(M,\partial M;\ZZ)$, then $\|x\|\equiv -\chi(S) \mod 2$. 
\end{fact}

This follows from the existence of the boundary map $\bdy \from H_2(M,\partial M;\ZZ)\to H_1(\partial M,\ZZ)$, and the fact that the number of components of an embedded multi-curve representing an element of $H_1(\partial M,\ZZ)$ is invariant mod 2. 

%    This observation will be useful in proving the next lemma, which records Thurston norm calculations that will be needed later. We will also need the following fact, which is a consequence of the fact that the unit ball of $|| \cdot ||$ is a polyhedron.

\subsection{Finding desired fibers}
The following lemma will be used to find fibers of almost every topological type.

\begin{lemma}
\label{lem:thurston_norm}
Let $M$ be a one-cusped fibered hyperbolic manifold with $H_2(M,\bdy M;\RR) \cong \RR^2$, and suppose $M$ contains embedded surfaces $S_1 \cong \Sigma_{1,1}$ and $S_2 \cong \Sigma_{2,0}$ representing non-trivial classes in $H_2(M,\bdy M)$. 

\begin{enumerate}
\item If $M$ has a fiber $F \cong \Sigma_{2,1}$, then the vertices of the unit Thurston norm ball are $\pm [S_1]$ and $\pm\frac{1}{2}[S_2]$. Furthermore, the fibered cone containing $F$ also contains fibers homeomorphic to $\Sigma_{g,n}$ for all $g \geq 2$ and $n \geq 1$ such that $(g-1,n)$ are relatively prime. All of these fibers have the same boundary slope as $F$.

\smallskip

\item If $M$ has a fiber $F \cong \Sigma_{1,2}$, then the vertices of the unit Thurston norm ball are $\pm [S_1]$, and either $\pm([S_2]+[S_1])$ or $\pm([S_2]-[S_1])$. Furthermore, the fibered cone containing $F$ also contains fibers homeomorphic to $\Sigma_{1,n}$ for all $n \geq 2$. All of these fibers have the same boundary slope as $F$.

\end{enumerate}
\end{lemma}

\begin{proof}

Let $x_1=[S_1]$ and $x_2=[S_2]$. Since $\chi(\Sigma_{1,1})=-1$ and $\|x_1\| > 0$, we conclude that $\|x_1\| = 1$. Similarly, since $\chi(\Sigma_{2,0})=-2$ and $\| x_2 \| > 0$, \Cref{Fact:Mod2} implies that $\|x_2 \| = 2$. Thus the four classes  $\pm x_1$ and $\pm\frac{1}{2}x_2$ all lie in $\bdy B$, where $B$ is the unit ball of the norm.

Recall the boundary homomorphism  $\bdy \from H_2(M,\bdy M)\to H_1(\bdy M)$, and fix the homology class $l = \bdy x_1 \in H_1(\partial M ; \ZZ)$. Note that $\pm l$ are the unique primitive classes in $H_1(\bdy M; \ZZ)$  that are trivial in $H_1(M)$. In particular, this implies that any alternate fiber $F'$ must have the same boundary slope as $F$.

Observe that, $x_1 \neq x_2$ because $\bdy x_1 = l \neq 0 = \bdy x_2 \in H_1(\partial M ; \ZZ)$.
Now, consider  classes $x'=x_1+x_2$ and $x''=x_1-x_2$. Since $\|x'\|\le \|x_1\|+\|x_2\|=3$ and $\|x'\|\equiv \chi (\Sigma_{1,1})+\chi(\Sigma_{2,0}) \mod 2$, we have $\|x'\|\in \{1,3\}$. Similarly, $\|x''\|\in\{1,3\}$.

\medskip

\begin{figure}
       	\centering
       	\begin{subfigure}{0.33\textwidth}
       	\centering
              	\includegraphics[scale=.64]{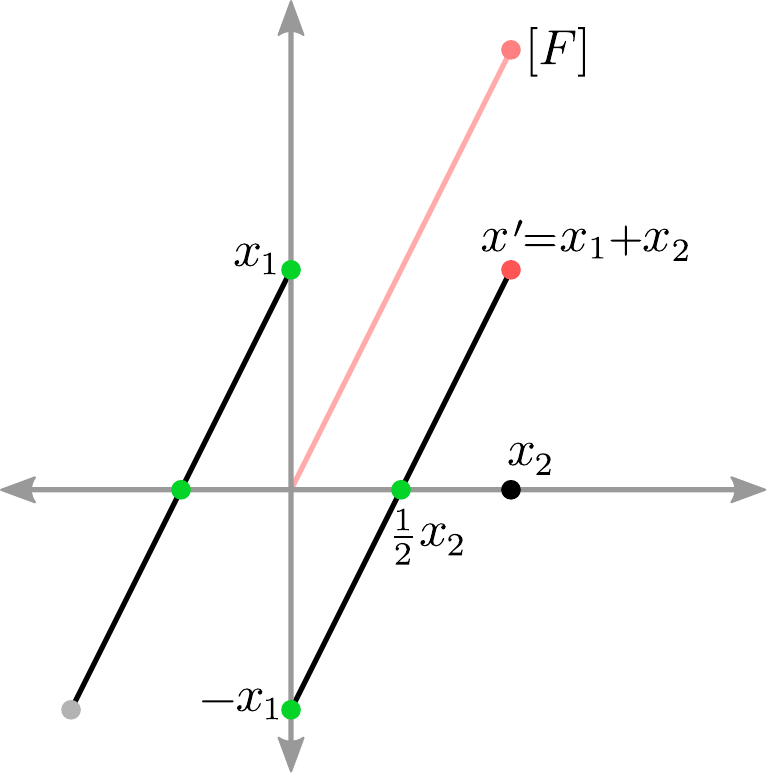}
          	\caption{}
        	\label{fig:H2_case1_n}
      	\end{subfigure}
       	\begin{subfigure}{0.32\textwidth}
       	\centering
            	\includegraphics[scale=.64]{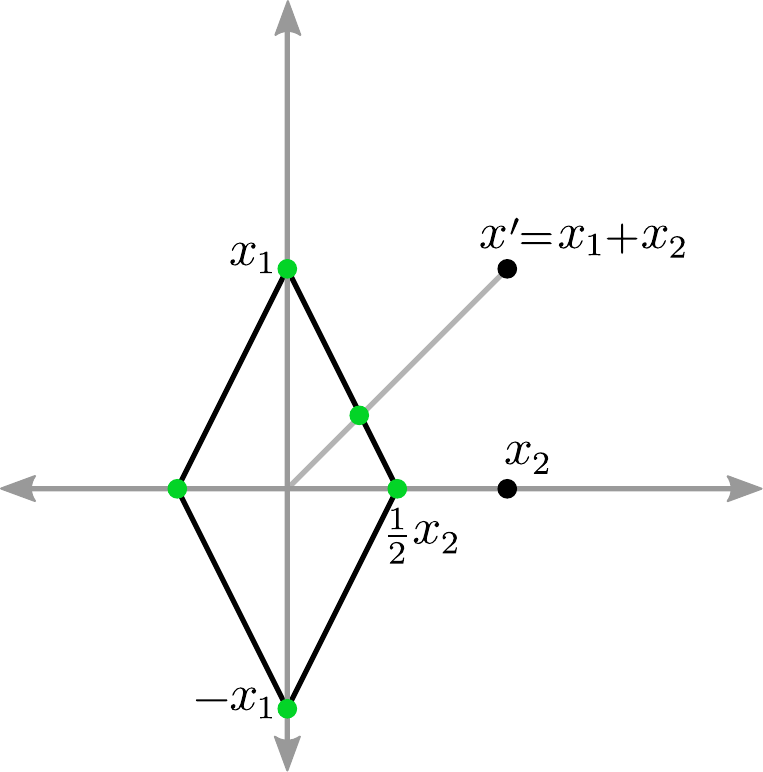}
          	\caption{}
         	\label{fig:H2_case1_y}
      	\end{subfigure}
      	\begin{subfigure}{0.33\textwidth}
      	\centering
            	\includegraphics[scale=.64]{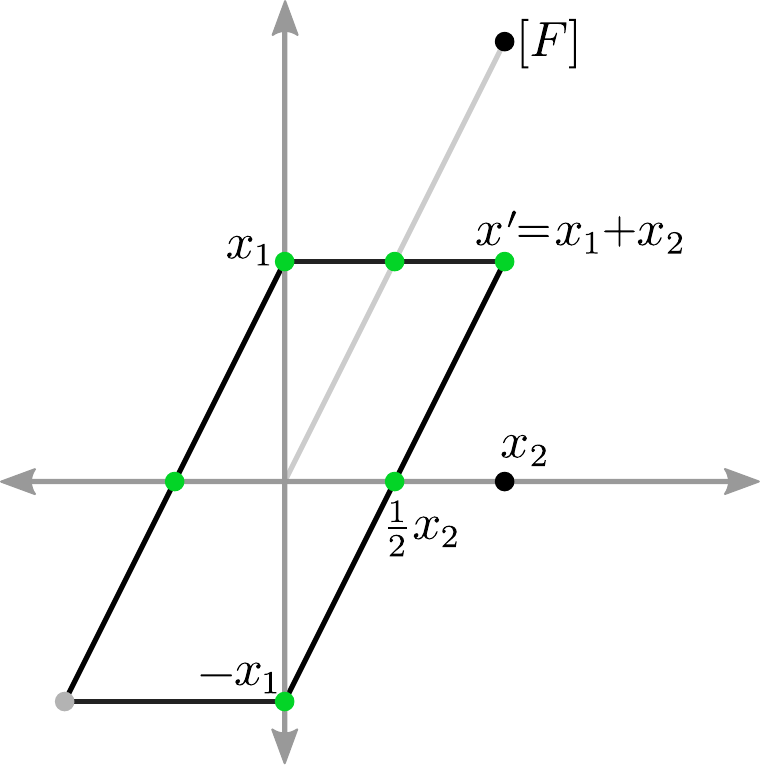}
          	\caption{}
         	\label{fig:H2_case2}
      	\end{subfigure}       
 	\caption{(a) For the case $F \cong \Sigma_{2,1}$, the assumption $\|x'\|=1$ leads to a contradiction. (b) When $F \cong \Sigma_{2,1}$, the Thurston norm ball is the rhombus shown. (c) When $F\cong \Sigma_{1,2}$, the Thurston norm ball  is either the polygon shown, or the mirror image of this polygon across the vertical axis.}
  	\label{fig:H2_ball}
\end{figure}

 \textbf{Case 1:} $M$ has a fiber  $F \cong\Sigma_{2,1}$, which implies $\bdy [F] = \pm l$. 

Suppose for a contradiction that $\|x'\|=1$. Then, by \Cref{Fact:Linearity}, the points $-x_1$, $x'$, and their average $\frac{1}{2}x_2$ all lie in a line segment contained in a face of $\bdy B$. Since $M$ has a cusp, and the interior point $\frac{1}{2}x_2$ is represented by (half) the closed surface $S_2 \cong \Sigma_{2,0}$, this cannot be a fibered face. Similarly, the points $x_1, \frac{-1}{2}x_2, -x' \in \bdy B$ all lie in a line segment in a non-fibered face of $\bdy B$. (See \Cref{fig:H2_case1_n}.)
 It follows that the fiber $F$ lies in a cone over some other face, which we may assume is in the first quadrant by changing the orientation of $F$ if necessary. Hence we can write $[F]=ax_1+bx'$, for some $a,b\in \QQ_{>0}$. 
%Let $[\bdy S_1]=l\in H_1(\bdy M)$, and observe that $l$ is the unique class in $H_1(\bdy M)$ that is trivial in $H_1(M)$, so that $[\bdy F]=\pm l$. 
%
Since $x'=x_1+x_2$, we have
%    Letting $\bdy \colon H_2(M,\bdy M)\to H_1(\bdy M)$ be the homomorphism $[S]\mapsto [\bdy S]$, we have 
\[
 \bdy x'=\bdy x_1+\bdy x_2=\bdy x_1 =l \in H_1(\bdy M),
\]
implying
\[ 
\pm l=\bdy [F] = \bdy (ax_1+bx')=(a+b) l. 
\]
Hence $a+b=\pm 1$, but $a+b\ge 0$, so we must have $a+b=1$. It follows that $[F]$ lies on the line segment joining $x_1$ and $x'$, which is impossible since $\|[F]\|=3$. From this contradiction, it follows that $\|x'\|=3$. 

If we consider  $x''=x_1-x_2$ in place of $x'$, and assume that $\|x''\|=1$, then the argument above with the obvious modifications again gives a contradiction. Hence $\|x_1-x_2\|=3=\|x_2-x_1\|$. Therefore the points $\pm x_1$, $\pm\frac{1}{2}x_2$, $\pm\frac{1}{3}x'$, $\pm\frac{1}{3}x''$ all have norm 1, hence \Cref{Fact:Linearity} implies that these points determine the unit norm ball. It follows that the vertices are $\pm v_1, \pm v_2$, where $v_1 = x_1 = [S_1]$ and $v_2= \frac{1}{2}x_2 =  \frac{1}{2}[S_2]$. See \Cref{fig:H2_case1_y}. 

Now, let $\face$ be the face containing $F$. Without loss of generality, $\face$ has vertices
$\{v_1, v_2 \}$. Fix a pair $(g,n)$ where $g\geq 2$ and $n \geq 1$, and where $\gcd(g-1,n)=1$. 
Then $y=n x_1+ (g-1)x_2$ is a primitive homology class in $\face$, which is represented by a fiber $F'$. Since the norm is linear on faces by \Cref{Fact:Linearity}, $\|y\|= n \|x_1\|+(g-1) \|x_2\|=n+2(g-1)$. Furthermore, $F'$ has exactly $n$ boundary components, since $S_2$ is closed and $S_1$ has one boundary component. Thus 
\[
2(g-1) + n =\|y\|=-\chi(F')=2\, \mathrm{genus}(F')-2+n
\]
hence $\mathrm{genus}(F')=g$, as desired. 
%    The fiber $F'$ has the same boundary slope as $F$, because $\bdy F$ is the unique slope on  $\bdy M$ that is null-homologous in $M$.

\medskip

\textbf{Case 2:} $F\cong \Sigma_{1,2}$, which implies $\bdy[F] = \pm 2l$.  

Suppose, for a contradiction, that $\|x'\|= \|x''\|=3$. Then the points $\pm x_1$, $\pm\frac{1}{2}x_2$, $\pm\frac{1}{3}x'$, $\pm\frac{1}{3}x''$ all have norm 1, and determine the unit norm ball must be as shown in \Cref{fig:H2_case1_y}. Hence, up to changing signs, we may assume that $[F]=ax_1+b(\frac{1}{2}x_2)$ for some $a,b\in\QQ_+$.
%$x_1,\frac{1}{2}x_2$ form a basis, $[F]=ax_1+b(\frac{1}{2}x_2)$ for some $a,b\in\QQ$.
% With $l\in H_1(\bdy M)$ defined as in Case 1, we must have $[\bdy F] =\pm 2l$, so that 
Then
\[
\pm 2l =\bdy [F]=\bdy \big( ax_1+b(\tfrac{1}{2}x_2) \big)=al 
\quad \implies \quad  a=\pm 2.
\]
Since the Thurston norm is linear in the cone over a face, we have 
$$
2=\|[F]\|=\|ax_1+b(\tfrac{1}{2}x_2)\|=|a|+|b|=2+|b|
\quad \implies \quad b=0,
$$
which is impossible, because the fiber $F$ must be in the interior of a fibered cone. 
This contradiction implies that either $\|x'\|=1$ or $\|x''\|=1$. 

If $\|x'\|=1$, the unit sphere $\bdy B$ contains the segments connecting $\pm x'$ and $\mp x_1$, as shown in \Cref{fig:H2_case2}.
%     or the reflection of these segments across the vertical axis. 
As above, we observe that $\frac{1}{2}x_2$ cannot lie in the interior of a fibered face, so (after possibly reversing the orientation on $F$) we must have $[F]=ax_1+bx'$
for  $a,b\in\QQ_+$.
%    , or $[F]=ax_1+bx''$ in the case of the reflected segments. 
%    Focusing for now on the case $\|x'\|=1$, and again using the boundary homomorphism, 
Applying the boundary homomorphism gives
 \[
 2l =\bdy [F]=\bdy(ax_1+bx')=(a+b)l
%     \quad \implies \quad
%     a+b=2.
 \]
which implies
 \[
 a+b = 2=\|[F]\|\le a\|x_1\|+b\|x'\|=a+b.
 \]
  Since the norm is only linear in the cone over a face (\Cref{Fact:Linearity}), the segment joining $x_1$ to $x'$ must lie in a face of $\bdy B$. It follows that $\pm x_1,\pm x'$ are the only vertices. 
  
If $\|x''\|=1$, an identical argument applies with $x'$ replaced by $x''$. In this case, the vertices of $\bdy B$ are $\pm x_1,\pm x''$.
Thus, in both cases, the vertices of the unit  norm ball are $\pm v_1$ and $\pm v_2$, where $v_1 = [S_1]$, and $v_2$ is either $[S_2]+[S_1]$ or $[S_2]-[S_1]$.

Now, let $\face$ be the face containing $F$. Without loss of generality, say $v_2 = [S_1] + [S_2]$ and $\face$ has vertices
$\{v_1, v_2\}$. The norm-realizing surface $P$ representing $v_2$ has $\chi(P) = -1$ and $\bdy[P] = l$. Thus $P$ is either a pair of pants or a one-holed torus. If $P$ is a pair of pants, then two boundary components of $P$ must cancel in $H_1(\bdy M)$, which means they can be tubed together to obtain an embedded one-holed torus. Thus, in either case,  $v_2$ is represented by an embedded $\Sigma_{1,1}$. Fix an integer $n \geq 2$, and let
$y = v_1 + (n-1) v_2 = n[S_1] + (n-1)[S_2]$. As before, $y$ is primitive and therefore represented by a fiber $F'$. Since the Thurston norm is linear on the fibered cone, $\|y\|=\|v_1\|+(n-1)\|v_2\|= n$. Furthermore, since $\partial S_2 = \emptyset$, $F'$ must have exactly $n$ boundary components. This gives that $n = \|y\| = -\chi(F')=2g(F')-2+n$ which implies  $g(F')=1$. We conclude that $F' \cong \Sigma_{1,n}$, as required. 
%
%    The fiber $F'$ has the same boundary slope as $F$, because $\bdy F$ is the unique null-homologous slope.
%
%Since $[S_2]=v_2-v_1$, the parallel boundary components of $v_1$ and $v_2$ are not compatible (i.e., their orientations do not agree, so they cannot be glued together). It follows that $F'$ has exactly $n$ boundary components, so $n = \|y\| = -\chi(F')=2g(F')-2+n$ which implies that $g(F')=1$, i.e., $F' \cong \Sigma_{1,n}$. 
\end{proof}

Before proving \Cref{Prop:NongeomExists}, we need a straightforward lemma about covers.

\begin{lemma}
\label{lem:lift_to_covers}
Let $\varphi \colon S \to S$ be a pseudo-Anosov homeomorphism and $f \colon \widehat S \to S$ a degree $d < \infty$ covering. Then the following holds.
\begin{enumerate}
\item\label{Itm:PowerLifts} There exists a pseudo-Anosov $\widehat \varphi : \widehat{S}\to \widehat{S}$ that is a lift of some power $\varphi^k$ of $\varphi$.
\item\label{Itm:TriLifts} The veering triangulation $\tri_{\varphi}$ is a geometric triangulation of $\mathring{M}_{\varphi}$ if and only if $\tri_{\widehat{\varphi}}$ is a geometric triangulation of $\mathring{M}_{\widehat{\varphi}}$. 
\item\label{Itm:PrincipalLifts} If $\varphi$ is principal and each peripheral curve of $S$ has $d$ lifts to $\widehat{S} $, then $\widehat{\varphi}$ is also principal.
\end{enumerate}
\end{lemma}

\begin{proof}
Conclusion \eqref{Itm:PowerLifts} is standard. Let $d$ be the degree of the cover. Then the finitely many index $d$ subgroups of $\pi_1(S)$  are permuted by the induced isomorphism $\varphi_\ast$. Thus some power of $\varphi_\ast$ must stabilize the subgroup $f_\ast \pi_1 (\widehat S) \subset \pi_1 (S)$, allowing the lifting criterion to be applied. 

Conclusion \eqref{Itm:TriLifts} follows from the fact that every simplex in the veering triangulation $\tri_{\varphi}$ of $\mathring{M}_{\varphi}$ lifts to a simplex in the veering triangulation $\tri_{\widehat{\varphi}}$ of $\mathring{M}_{\widehat{\varphi}}$, with the same shape.

For conclusion \eqref{Itm:PrincipalLifts}, note that since $\varphi$ is principal, every singularity of $\varphi$ is either $3$--pronged and occurs at an interior point of $S$ or $1$--pronged and occurs at a puncture. Since each peripheral curve of $S$ has $d$ lifts to $\widehat S$, the same is true for $\widehat \varphi$. Thus $\widehat \varphi$ is principal.
\end{proof}

We can now begin proving \Cref{Prop:NongeomExists}, case by case.

\begin{proposition}\label{Prop:NonGeomPosGenus}
Let $S \cong \Sigma_{g,n}$ be a hyperbolic surface of genus $g \geq 1$, excluding $\Sigma_{1,1}$. Then  there exists a principal pseudo-Anosov $\varphi \in \Mod(S)$ such that the associated veering triangulation of the mapping torus $\mathring M_{\varphi}$ is non-geometric.
\end{proposition}

\begin{proof}%[Proof of \Cref{Prop:NongeomExists}]
We consider three different cases. \smallskip

\textbf{Case 1: $g=1$ and $n \geq 2$.} Let $F = \Sigma_{1,2}$ and let $\varphi = \varphi_3$ be the third mapping class described in \Cref{prop:seeds}. 
By \Cref{prop:seeds}, $\varphi$ is a principal pseudo-Anosov, such that the veering triangulation of $M_{\phicirc}$ is non-geometric.

Let $M_{\varphi}$ be the mapping torus of $\varphi \colon F \to F$. This manifold has a single cusp.
%This is the un-punctured mapping torus, without removing singularities.
According to \texttt{Regina}, $M_\varphi$ contains embedded surfaces $S_1 \cong \Sigma_{1,1}$ and $S_2 \cong \Sigma_{2,0}$ which are non-trivial in $H_2(M_\varphi, \bdy M_\varphi; \RR) \cong \RR^2$. To verify this, \texttt{Regina} computes the complete list of embedded vertex normal surfaces for $M_\varphi$. (See e.g.\ \cite{Burton:optimizing} for a discussion of vertex normal surfaces and the role they play in computation.) Among these vertex normal surfaces are  $S_1 \cong \Sigma_{1,1}$ and $S_2 \cong \Sigma_{2,0}$. Cutting $M_\varphi$ along these surfaces ensures that they are homologically non-trivial.  The dimension of the homology is also rigorously computed by \texttt{Regina}. See the ancillary files \cite{FTW:Auxiliary} for full details.

Thus, by \Cref{lem:thurston_norm},  the fibered cone containing $F$ also contains fibers homeomorphic to $\Sigma_{1,n}$ for all $n \geq 1$. All of these fibers have the same boundary slope as $F$, hence 
 the mapping classes of these fibers are all principal by  \Cref{Lem:DegeneracySlope}.
Finally, \Cref{Lem:FaceInvariance} says that all of these fibers
 induce the same non-geometric veering triangulation of $M_{\phicirc}$.

\smallskip

 \textbf{Case 2: $g \geq 2$ and $n=0$.} Let $F = \Sigma_{2,0}$, and let $\varphi = \varphi_1 \in \Mod(F)$ be the first mapping class described in \Cref{prop:seeds}. By that proposition, $\varphi$ is a principal pseudo-Anosov, such that the veering triangulation of $M_{\phicirc}$ is non-geometric. Now, recall that every closed hyperbolic surface $S$ is a finite cover of $F$. Thus  \Cref{lem:lift_to_covers} gives the desired result for $S$.
 
 \smallskip
 
\textbf{Case 3: $g\geq 2$ and $n \geq 1$.} Let $F = \Sigma_{2,1}$ and let $\varphi = \varphi_2$ be the second mapping class described in \Cref{prop:seeds}. 
By \Cref{prop:seeds}, $\varphi$ is a principal pseudo-Anosov, such that the veering triangulation of $M_{\phicirc}$ is non-geometric.

Let $M_{\varphi}$ be the mapping torus of $\varphi \colon F \to F$. Using \texttt{Regina}, as in Case 1, we check that  $M_\varphi$ contains embedded surfaces $S_1 \cong \Sigma_{1,1}$ and $S_2 \cong \Sigma_{2,0}$ which are non-trivial in $H_2(M_\varphi, \bdy M_\varphi; \RR) \cong \RR^2$. By \Cref{lem:thurston_norm},  the fibered cone containing $F$ also contains fibers  homeomorphic to $\Sigma_{g,n}$ for all $g \geq 2$ and $n \geq 1$, where $(g-1,n)$ are relatively prime. All of these fibers have the same boundary slope as $F$. Thus, by 
\Cref{Lem:DegeneracySlope,Lem:FaceInvariance}, we obtain the desired conclusion for all $g \geq 2$ and $n \geq 1$ such that $\gcd(g-1,n) = 1$.

%    $$\varphi = T_{c_0}T_{b_0}^{-2}T_{c_0}T_{a_0}T_{b_0}T_{a_1}^{-1}T_{a_0}T_{b_1}.$$
%T_{a_1}T_{b_1}T_{b_0}T_{c_0}^2T_{a_1}^3T_{b_1}T_{a_0}^{-2}T_{a_1}T_{c_0}^{-1}T_{b_1}^{-1}T_{b_0} may also work

Finally, suppose $S \cong \Sigma_{g,n}$, with $\gcd(g-1,n)= d > 1$. Then $g' -1 = (g-1)/d$ and $n' = n/d$ are relatively prime, with $g' \geq 2$ and $n' \geq 1$. Thus, by the above paragraph, the fibered cone of $M_\varphi$ containing $F$ also contains a fiber $F' \cong \Sigma_{g',n'}$. Observe that $S$ is a $d$--fold cyclic cover of $F'$ (realize $S$ with $d$ groups of $g' -1$ doughnut holes and $n'$ punctures, arranged symmetrically around a central doughnut hole). By construction, peripheral curves of $F'$ lift to peripheral curves of $S$.
Thus, by  \Cref{lem:lift_to_covers}, a power of the monodromy of $F'$ lifts to a principal pseudo-Anosov on $S$, and 
 the non-geometric veering triangulation of $M_{\phicirc}$ lifts to a non-geometric veering triangulation of the corresponding finite cover of $M_{\phicirc}$.
\end{proof}

\begin{proposition}\label{Prop:NonGeomZeroGenus}
Let $S \cong \Sigma_{0,n}$ be a surface of genus $g =0$, with $n \geq 5$ punctures. Then  there exists a principal pseudo-Anosov $\varphi \in \Mod(S)$ such that the associated veering triangulation of the mapping torus $\mathring M_{\varphi}$ is non-geometric.
\end{proposition}

\begin{proof}
If $n =5$ or $n=6$, the mapping classes $\varphi_5$ and $\varphi_6$ described in \Cref{prop:seeds} satisfy the desired conclusion.
From now on, we treat planar surfaces with $n \geq 7$ punctures.

Let $F=\Sigma_{0,7}$, and let $\varphi$ be the mapping class
$$\varphi_7 = T_{r_5}^2T_{r_4}T_{r_3}^{-1}T_{r_1}^{-1}T_{r_4}T_{r_2}^{-1}T_{r_3}T_{r_1}T_{r_2}^{-1}$$ 
given in \Cref{prop:seeds}.
By \Cref{prop:seeds}, $\varphi$ is a principal pseudo-Anosov and the veering triangulation of the mapping torus $M_{\phicirc}$ is non-geometric. We will show that the fibered cone of $H_2(M_\varphi, \bdy M_\varphi)$ containing $[F]$ also contains a fiber homeomorphic to $\Sigma_{0,n}$ for every $n \geq 7$. Then, we will show that all of these fibers have principal monodromy.

\begin{figure}
        \centering
                \includegraphics[scale=.73]{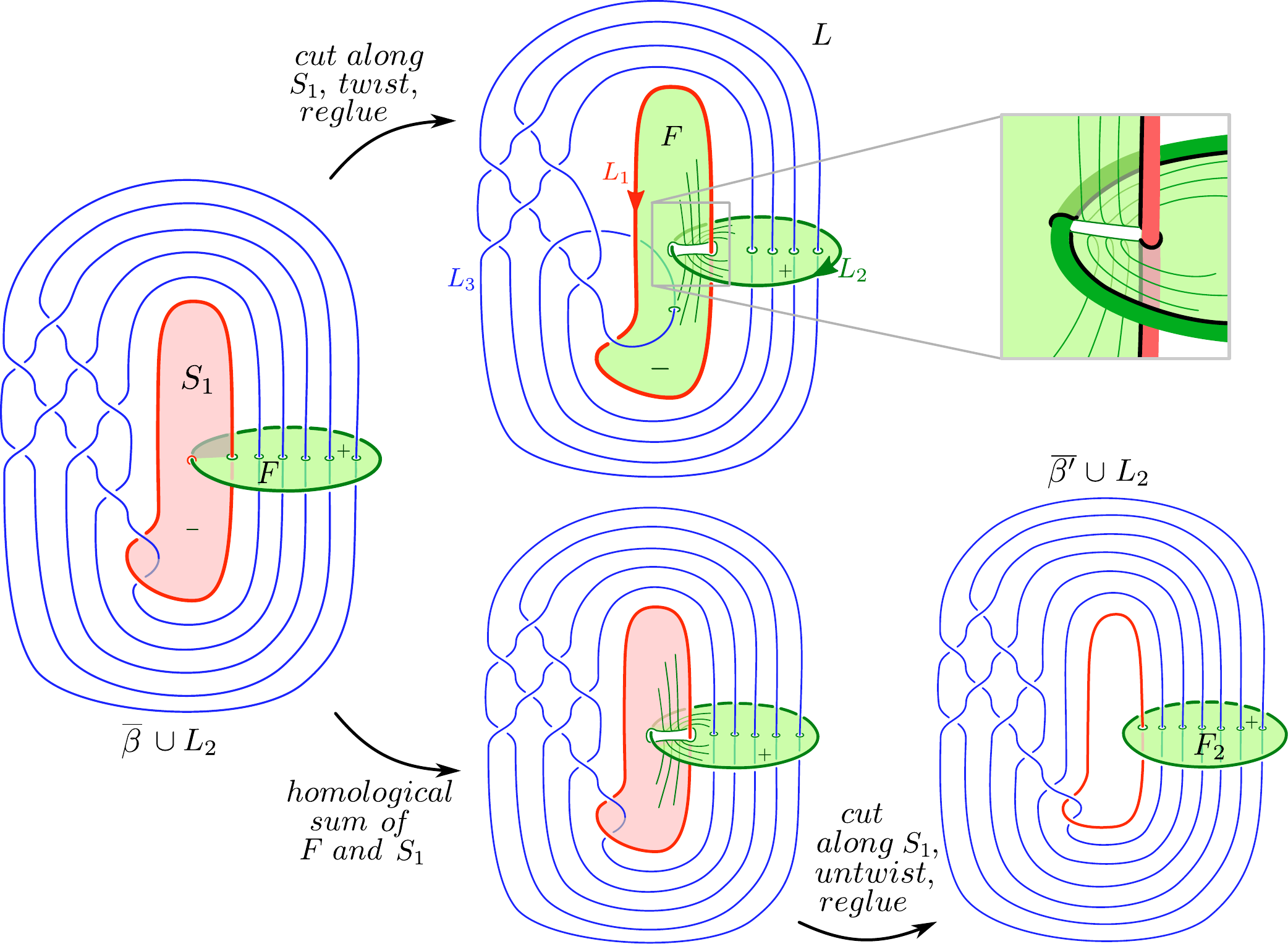}
       \caption{The mapping torus $M_\varphi$ of 
        $\varphi = T_{r_5}^2T_{r_4}T_{r_3}^{-1}T_{r_1}^{-1}T_{r_4}T_{r_2}^{-1}T_{r_3}T_{r_1}T_{r_2}^{-1}$ has many embeddings as a link complement in $S^3$. In the left panel, we realize $\varphi$ as a braid word $\beta$, whose braid generators are read from the bottom up. In the top center panel, we cut, twist, and reglue along the twice-punctured disk $S_1$, giving a re-embedding of $M$. In the bottom right panel, we twist along $S_1$ in the opposite direction, making it easier to see the fiber $F_2$ that is homologous to $F + S_1$ and compute the monodromy of $F_2$. Note that it is possible to obtain $F_a$ in a similar way---i.e., by successive applications of the process described by the bottom two arrows. This gives intuition for how to think about $F_a$  and compute its monodromy. 
%            (which we leave as an exercise).
         }
        \label{fig:link_seed_pod}
\end{figure}

To begin the proof, we embed $M = M_\varphi$ as a link complement in $S^3$. Note that the generators $T_{r_6}$ and $T_{r_0}$ do not appear in $\varphi$, hence one of the punctures of $F$ is fixed. We can therefore think of $\varphi$ as a mapping class on the $6$--punctured disk. More precisely, let $B_k$ be the braid group on $k$ strands, and consider the natural homomorphism $B_k \to \Mod(\Sigma_{0,k+1})$ defined by $\sigma_i \mapsto T_{r_i}$.
Then $\varphi$ is the mapping class corresponding to the braid
\[
\beta = \sigma_5^2\sigma_4\sigma_3^{-1}\sigma_1^{-1}\sigma_4\sigma_2^{-1}\sigma_3\sigma_1\sigma_2^{-1}.
\]
Consequently, the mapping torus $M_\varphi$ is homeomorphic to $S^3 \setminus (\overline{\beta} \cup L_2)$, where $\overline{\beta}$ is the braid closure of $\beta$ and $L_2$ is the braid axis. See the left panel of \Cref{fig:link_seed_pod}. In this embedding of $M_\varphi$, the fiber $F$ becomes the $6$--punctured disk shown in green.

Next, we re-embed $M$ into $S^3$ via a Rolfsen twist. That is: cut $M$ along the twice-punctured disk $S_1$ (colored pink in \Cref{fig:link_seed_pod}), perform one counter-clockwise full twist from the underside, and re-glue along $S_1$. After this operation, we have $M_\varphi  \cong S^3 \setminus L$, where $L=L_1\cup L_2\cup L_3$ is the three-component link in the upper center of \Cref{fig:link_seed_pod}.  The image of the fiber $F$ under this re-embedding is shown again in light green. 

%    We orient $L$ so that the orientations on $L_1$ and $L_2$ induce a consistent orientation on $F$. That is: as you walk around $L_i$ with $F$ to your right, the transverse orientation on $F$ (corresponding to the flow $\eta$) should point upward.

%    As Thurston observes in \cite{Thu86}, we can identify the set $\{ L_1, L_2,  L_3 \}$ with a basis for $H_2(M,\partial M)$, via the sequence of isomorphisms
%    \[ H_2(M,\partial M) \cong H_2(\SS^3,L)\cong H_1(L),\] 
%    where the first isomorphism comes  from excision and second comes from the long exact sequence of a pair. In particular, the punctured disk bounded by each $L_i$, transversely oriented by $\eta$, forms a basis element of 
%     $H_2(M,\partial M)$ in correspondence to $L_i$ itself.
%     (The reader can readily verify that $L_3$ is indeed unknotted, and hence bounds a disk punctured by $L_1 \cup L_2$.) 
 
The link $L$ allows a clear view of two surfaces that will be important for our homological computations:
the $2$--punctured disk $S_1$ bounded by $L_1$, and the $5$--punctured disk $S_2$ bounded by $L_2$. The top center of \Cref{fig:link_seed_pod} shows their (transverse) orientations: we are looking at the back side of $S_1$ and the front side of $S_2$.
Then, setting $x_1 = [S_1]$ and $x_2 = [S_2]$, we have 
 $[F]=x_1+x_2$. Since 
$$
5=\|[F]\|\le \|x_1\|+\|x_2\|\le1+4=5,
$$
 we learn that $\|x_1\| = 1$ and $\|x_2\| = 4$.
Furthermore, since the Thurston norm is only linear in the cone over a face (see \Cref{Fact:Linearity}), it follows that the segment joining $x_1$ to $x_2$ must lie in the fibered cone containing $[F]$.

Now, let $a$ be a positive integer and consider $y=ax_1+x_2$. Since $y$ is primitive, it is represented by a fiber $F_a$. In \Cref{fig:link_seed_pod}, $F_a$ can be visualized as the sum of $a$ copies of $S_1$ and one copy of $S_2$. We wish to compute the topological type of $F_a$, starting with the number of punctures.

Let $\bdy \colon H_2(M, \bdy M)\to H_1(\bdy M)$ be the boundary homomorphism $[S]\mapsto [\bdy S]$. To compute $\bdy y =\bdy(ax_1+x_2)= a \bdy x_1 +\bdy x_2$, it suffices to take the homological sum (in $H_1(\bdy M)$) of $a$ copies of  $[\bdy S_1]$ and one copy of $[\bdy S_2]$. Let $T_i$ be the torus of $\bdy M$ corresponding to the link component $L_i$. Then the only intersections of $\bdy S_1$ with $\bdy S_2$ occur on  $T_1$ and $T_2$. On the torus  $T_1$, there are $a$ copies of the longitude, coming from $a [\bdy S_1]$, and one copy of the meridian, coming from $[\bdy S_2]$. The homological sum of these is a single curve of slope $1/a$. The situation on $T_2$ is similar: there are $a$ copies of the meridian, and one copy of the longitude, giving a single curve of slope $a$. \Cref{fig:hom_sum} demonstrates this for $a=3$. Since $L_3$ intersects $S_1$ once and $S_2$ four times, the boundary of $y$ also contains $a+4$ copies of the meridian on the torus $T_3$. Furthermore, the orientations on $S_1$ and $S_2$ induce the same orientation on each of these $a+4$ copies of the meridian, so none of them cancel in $H_1(\bdy M)$. We conclude that the number of boundary components of $F_a$ is $1+1+(a+4)=a+6$.

\begin{figure}
        \centering
        \includegraphics[scale=1.2]{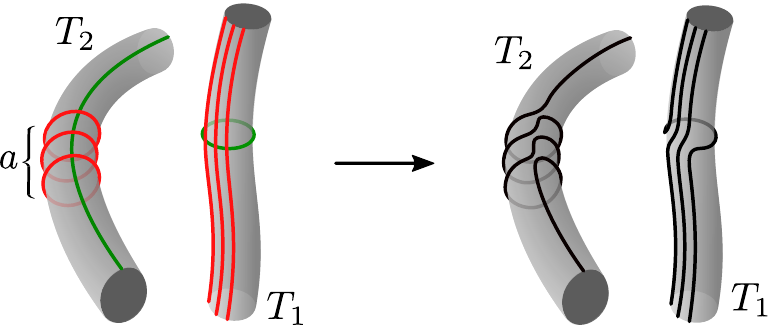}
       \caption{The homological sum $ a \bdy x_1 + \bdy x_2   \in H_1(\bdy M)$. The left frame shows $a \cdot \bdy S_1 \cap (T_1\cup T_2)$ in red and $\bdy S_2 \cap(T_1\cup T_2)$ in green. After the sum (right), there is a curve of slope $1/a$ on $T_1$ and a curve of slope $a$ on $T_2$.}
        \label{fig:hom_sum}
\end{figure}

Since the norm is linear on the cone over a face, we get $\|y\|=a\|x_1\|+\|x_2\|=a+4$. Furthermore, since $F_a \cong \Sigma_{g,n}$ is a fiber, hence norm-realizing, we have 
$$
a+4 = \|y\|=-\chi(F_a)=2g-2+n=2g-2+(a+6) = 2g + (a+4).
$$
We conclude that the genus of $F_a$ is $g=0$. Hence $F_a \cong \Sigma_{0,a+6}$. By varying the value of $a  \in \NN$, we get all surfaces $\Sigma_{0,n}$ for $n\ge 7$. By \Cref{Lem:FaceInvariance}, the veering triangulation associated to the monodromy for $F_a$ is the same as the veering triangulation for $F = F_1$, hence non-geometric. 

Next, we compute the monodromy of $F_2$ and show that it is principal.
Since $[F] = [F_1] = x_1 + x_2$, we have
\[
[F_2] = x_1 + x_1 + x_2 = [S_1] + [F].
\]
The fiber $F_2$ is shown in the bottom center frame of \Cref{fig:link_seed_pod}.
We may visualize the monodromy of $F_2$ by again re-embedding $M$ into $S^3$, via a Rolfsen twist in the opposite direction. 
That is: cut $M$ along the twice-punctured disk $S_1$, perform a full clockwise twist (from the underside), and reglue. This realizes $M$ as the complement of a new link, shown in \Cref{fig:link_seed_pod}, bottom right.
This link is $\overline{\beta'}\cup L_2$, where $\beta'=\sigma_6^2\sigma_5\sigma_4\sigma_3^{-1}\sigma_1^{-1}\sigma_4\sigma_2^{-1}\sigma_3\sigma_1\sigma_2^{-1} \in B_7$. After the re-embedding, the fiber $F_2$ becomes the green $7$-punctured disk shown, with monodromy $\psi$ corresponding to the braid word $\beta'$:
$$
\psi= T_{r_6}^2T_{r_5}T_{r_4}T_{r_3}^{-1}T_{r_1}^{-1}T_{r_4}T_{r_2}^{-1}T_{r_3}T_{r_1}T_{r_2}^{-1}.
$$
Using \texttt{flipper}, we confirm that $\psi$ is in fact principal. 

It remains to show that the monodromy of $F_a$ is principal for every $a \in \NN$. We already know this for $F_1$ and $F_2$.
We finish the proof using \Cref{Lem:DegeneracySlope} and linear algebra. For each cusp torus $T_i$ of $M$, let $\delta_i$ be a simple closed curve realizing the degeneracy slope of $F$, oriented in the direction of the flow $\eta$.  Then, for $i \in \{ 1,2,3 \}$, we have a sequence of homomorphisms ($\ZZ$ coefficients are presumed):
\[
H_2(M, \bdy M) \xrightarrow{\: \: \bdy \: \:} H_1(\bdy M) \xrightarrow{\: \: \pi_i \: \:} H_1(T_i) \xrightarrow{\: \iota(\cdot, \, \delta_i) \:} \ZZ,
\]
where $\pi_i \from H_1(\bdy M) = \oplus_{j=1}^3 H_1(T_j) \to H_1(T_i)$ is the projection map to the $i$-th coordinate and $\iota(\cdot, \delta_i)$ is the algebraic intersection pairing. The composition of these homomorphisms is a linear functional $\nu_i \from H_2(M, \bdy M) \to \ZZ$. Consider its values for $[F_1]$ and $[F_2]$.

On the torus $T_1$, both fibers $F_1$ and $F_2$ have a single boundary component (see \Cref{fig:hom_sum}, right). Since the monodromies of $F_1$ and $F_2$ are principal, both $\bdy F_1$ and $\bdy F_2$ intersect $\delta_1$ once. With our orientations, $\nu_1([F_1]) = \nu_1([F_2]) = 1$. Thus, by linearity, we have $\nu_1([F_a]) = 1$ for every $a$. By an identical argument, $\nu_2([F_1]) = \nu_2([F_2]) = 1$, hence $\nu_2([F_a]) = 1$ for every $a$. Finally, on the torus $T_3$, we have seen that $\bdy F_a$ consists of $a+4$ parallel components whose slope is independent of $a$. Since $F_1$ is principal, each of these components intersects $\delta_3$ once. Since every boundary component of $F_a$ intersects the degeneracy slope once, 
\Cref{Lem:DegeneracySlope} implies that the monodromy of $F_a$ is principal for every $a \in \NN$.
%    
%    For $j\in\{1,2,3\}$, let $\del_j M$ be the boundary component of $M$ corresponding to $L_j$. Choose $\delta_j\in H_1(\del_j M)$ so that slope of $\delta_j$ is the degeneracy slope on $\del_j M$. Let $\del_j(F)\in h_1(\del_j M)$ denote (a component of) the boundary of the fiber $F$ on $\del_j M$. Since $F$ is principal, $\iota(\delta_j,\del_j(F))=1$, so $\delta_j,\del_j(F)$ are a basis for $H_1(\del_j M)$. For any other class $x\in H_2(M,\del M)$ in the fibered cone containing $[F]$, we can write $\del_jx = d\delta_j+f\del_j(F)$ for some $d,f\in \ZZ$. Furthermore, since the transverse orientations of the fibers in a given cone all agree, we must have $f>0$. It follows that the function $\iota(\delta_j,\cdot):H_2(M,\del M;\ZZ)\to \ZZ$ defined by $x\mapsto \iota(\delta_j,\del_jx)$ is linear on the cone containing $F$. Indeed, for $\del_jx = d\delta_j+f\del_j(F)$ we have $\iota(\delta_j,\del_jx)=f$. Since $\iota(\delta_j,\del_j(F))=1=\iota(\delta_j,\del_j(F_2))$, it follows that $\iota(\delta_j,\cdot)$ must be constant on the intersection of the fibered cone containing $F$ with the line determined by $[F]$ and $[F_2]$. It follows that $\iota(\delta_j,\del_j(F_a))=1$ for all $a$, and therefore $F_a$ has principal monodromy.
\end{proof}

\begin{remark}
    \texttt{flipper} has the capability to compute degeneracy slopes from a veering triangulation, using
     \cite[Observation 2.9]{FG13}. A combination of \texttt{flipper} and \texttt{Snappy} shows that in $M \cong S^3 \setminus(L_1 \cup L_2 \cup L_3)$, the degeneracy slope $\delta_1$ is the longitude of $L_1$; 
meanwhile, $\delta_2$ is the meridian of $L_2$; and $\delta_3$ is (meridian$\, - \,$longitude) on $L_3$. This fact, combined with \Cref{fig:hom_sum}, gives an alternate proof that every $F_a$ has principal monodromy. The above argument using the linear functionals $\nu_i$ avoids the need to ever identify $\delta_i$.
    \end{remark}

\begin{proof}[Proof of \Cref{Prop:NongeomExists}]
Let $S$ be a hyperbolic surface. If $\xi(S) = 0$, then $S \cong \Sigma_{0,3}$, hence $\Mod(S)$ is finite. If $\xi(S) = 1$, then $S \cong \Sigma_{0,4}$ or $\Sigma_{1,1}$, and the work of Akiyoshi \cite{Aki99}, Lackenby \cite{Lackenby:Bundle}, and Gu\'eritaud \cite{Gue06} shows that  all pseudo-Anosov mapping classes in $\Mod(S)$ have geometric veering triangulations. 

Now, assume that $\xi(S) \geq 2$. Under this hypothesis, \Cref{Prop:NonGeomPosGenus,Prop:NonGeomZeroGenus} show that there exists a principal pseudo-Anosov $\varphi \in \Mod(S)$ such that the associated veering triangulation of the mapping torus $\mathring M_{\varphi}$ is non-geometric.
\end{proof}

\bibliographystyle{hamsplain}
\bibliography{biblio}

\end{document}